\documentclass{amsart}

\usepackage{geometry}
\usepackage{amsmath}
\usepackage{amssymb}
\usepackage{amsthm}
\usepackage{color}
\usepackage[shortlabels]{enumitem}
\usepackage{hyperref}

\usepackage{hyperref}

\newtheorem{prop}{Proposition} 

\newtheorem{lemma}{Lemma}
\newtheorem{thm}{Theorem}
\newtheorem{defn}{Definition}

\numberwithin{equation}{section}

\usepackage[dvipsnames]{xcolor}

\begin{document}

\begin{abstract}
Polymath 8b~\cite{Polymath:2014:VSS} proved that $H_1= \liminf_{n \rightarrow \infty} (p_{n+1}-p_n) \leq 246$.  In this paper we show how the Bombieri-Vinogradov theorem  can be combined  with newer equidistribution estimates for smooth moduli to obtain the improved bound   $H_1 \leq 240$.  
\end{abstract}

\title{ bounded gaps between primes}
\date{}
\author{Julia Stadlmann}
\maketitle

\vspace{-5mm} 

\section{Introduction}\label{sec:intro}

\vspace{3mm} In this paper we study  $H_1= \liminf_{n \rightarrow \infty} (p_{n+1}-p_n)$, the smallest gap between consecutive primes which appears infinitely often.  The twin prime conjecture suggests that $H_1=2$, but only since  Zhang's breakthrough~\cite{Zhang:2015:BGP} in 2013 is it  known that $H_1$ is finite. Zhang's bound $H_1 \leq 70000000$ has since been improved several times. In particular, Maynard~\cite{Maynard:2015:SGP} showed that $H_1 \leq  600$, and the Polymath project~\cite{Polymath:2014:VSS} improved this to the current best known bound $H_1 \leq  246$. A core ingredient of Zhang's result were his newly proved equidistribution estimates for primes in arithmetic progressions with smooth moduli. However, Maynard's new sieve weights later allowed for the use of the older and   simpler Bombieri-Vinogradov theorem~\cite{Maynard:2015:SGP}. Polymath's bound $H_1 \leq 246$ also only uses the Bombieri-Vinogradov theorem and no newer equidistribution estimates~\cite{Polymath:2014:VSS}.

\medskip

More generally, one may consider $H_m=\liminf_{n \rightarrow \infty} (p_{n+m}-p_n)$, the gap between $m+1$ consecutive primes.   Maynard~\cite{Maynard:2015:SGP} first showed that $H_m$ is bounded, giving estimate $H_m \ll m^3 \exp(4m)$, and  Tao obtained a similar bound around the same time. The Polymath project~\cite{Polymath:2014:VSS} improved this bound to  $H_m \ll \exp((4-\frac{28}{157})m)$. Baker and Irving~\cite{Baker:2017:BIP} first introduced Harman's sieve into this problem, giving  $H_m \ll \exp(3.815 m)$, which we improved to $H_m \ll \exp(3.8075 m)$~\cite{Stadlmann:2023:PAP}. Concrete upper bounds for $H_2, \dots, H_5$ are also known, see~\cite{Polymath:2014:VSS} and \cite{Stadlmann:2023:PAP}.
 The best known upper bounds on $H_m$ for $m\geq 2$ all use equidistribution estimates of Zhang type. 
 
 \medskip 
 Equidistribution estimates of Zhang type give an exponent of distribution larger than $\frac{1}{2}$,  improving on the exponent in the  Bombieri-Vinogradov theorem, but  only apply to a limited choice of moduli, like  $x^{\delta}$-smooth numbers.   When applying the GPY method with Maynard-Tao sieve weights, a larger exponent of distribution corresponds to a bigger function support in the relevant optimization problem, whereas a restricted choice of moduli also restricts the support. 
 In Polymath's approach to proving upper bounds on $H_m$~\cite{Polymath:2014:VSS}, the gains from having a larger exponent of distribution outweigh the losses from being restricted to  $x^{\delta}$-smooth numbers when $m\geq 2$, but not  when $m=1$. 
 
 \medskip The aim of this paper is to combine the function support given by the Bombieri-Vinogradov theorem with that of newer equidistribution estimates, so that we may   benefit from the advantages of both, and improve the bound on $H_1$. We prove the following:

\begin{thm} \label{thm}
The quantity $H_1=\liminf_{n \rightarrow \infty} (p_{n+1}-p_n)$ satisfies $ H_m \leq 240$. 
\end{thm}

The value $246$ corresponds to the length of the shortest admissible $50$-tuple, while $240$ is the length of the shortest admissible $49$-tuple. The bound $H_1 \leq 240$ is therefore the smallest possible improvement. However, while Polymath~\cite{Polymath:2014:VSS} used symmetric polynomials of degree $\leq 27$ to achieve $H_1 \leq 246$, we only needed polynomials of degree $\leq 21$ to get $H_1 \leq 240$. As such,  we believe that  our method could achieve bigger improvements, but due to limited  resources and limited time, we have decided to present this modest improvement as a proof of concept.

\subsection{Proof strategy} By using   equidistribution estimates for moduli with a large smooth factor, we can choose a considerably larger support for the functions used in the GPY optimization problem, giving better results in the corresponding eigenvalue computation. Our strategy is   outlined below.

 \medskip
 
 First we need to determine a suitable function support. On the one hand, this support should capture many tuples to which the available equidistribution estimates are applicable, but on the other hand it should be simple enough that we can efficiently integrate over the chosen  region.  Our basic idea is to consider a union $S=S_{BV} \cup S_{Z}$ of the following sets: 
 \begin{align*} 
 &S_{BV}=\left\{(t_1, \dots, t_k) \in [0,1]^k: t_1+\dots+t_k \leq \frac{1}{4}\right\}, \\
 &S_{Z}(\delta,\omega,\underline{b})=\left\{(t_1, \dots, t_k) \in [0,1]^k:    t_1+\dots+t_k \in \left[ \frac{1}{4},  \frac{1}{4}+\omega \right] \mbox{ and }  \sum_{i \in I} t_i \leq b_{|I|} \mbox{ for }  I=\{i: t_i > \delta\} \right\}.
 \end{align*}
 To decide whether $S$ is a suitable support, we must ask if we have the necessary  equidistribution estimates for all moduli $m = d_1^{t_1} \dots d_k^{t_k} e_1^{s_1} \dots e_k^{s_k}$ with $((t_1, \dots, t_k), (s_1, \dots, s_k)) \in S \times S$. For elements of $S_{BV} \times S_{BV}$ we can simply apply the   Bombieri-Vinogradov theorem. On the other hand, the condition $\sum_{i \in I} t_i \leq b_{|I|}$ ensures that elements of $(S \times S) \setminus (S_{BV} \times S_{BV})$ have a large $x^\delta$-smooth factor. This $x^\delta$-smooth factor is convenient for applying equidistribution estimates, provided we choose the bounds $ \underline{b} = (b_1, b_2, \dots)$ suitably small. 
 (To increase the support further, we will also apply the epsilon-enlargement trick of Polymath~\cite{Polymath:2014:VSS}. Additionally, to restrict the choice of sequences for which  equidistribution estimates  must apply, we will    implement Harman's sieve.)
 
 \medskip However, while the moduli corresponding to $(S \times S) \setminus (S_{BV} \times S_{BV})$ have a reasonably large  $x^\delta$-smooth factor, they are not necessarily $x^\delta$-smooth themselves. In particular, moduli coming from $S_{BV} \times S_Z$ may have a prime factor of size up to $x^{1/4}$, meaning they are not   even $i$-tuply $x^\delta$-densely divisible and thus neither the equidistribution estimates of Polymath~\cite{Polymath:2014:EDZ} nor those proved by the author in~\cite{Stadlmann:2023:PAP} are directly applicable to the chosen support. As such, we must prove variations of the equidistribution estimates of~\cite{Polymath:2014:EDZ} and~\cite{Stadlmann:2023:PAP} in which the restrictions on the moduli have been considerably relaxed.  
 
 \medskip  Having chosen a support and proved suitable equidistribution estimates, we need to solve the  optimization problem produced by the GPY method with Maynard-Tao sieve weights. This involves maximizing a ratio of integrals, and can be transformed into an eigenvalue computation for a matrix whose entries are determined by these integrals. Here one needs to choose a basis of functions to integrate. To prove $H_1 \leq 246$, Polymath~\cite{Polymath:2014:VSS} used   symmetric polynomials of the form $ \cdot p(t_1, \dots, t_k)^2 (1-t_1-\dots-t_k)^b$, where $p(t_1, \dots, t_k)$ is itself a symmetric polynomial of degree $a$ and where $2a+b \leq 27$. To show the strength of our method, we will use the same kind of basis to prove $H_1 \leq 240$, but with the stricter condition $2a+b \leq 21$. This means we get an improved result despite working with a significantly smaller basis. 
 
 \medskip However, trying to integrate symmetric polynomials over $S=S_{BV} \cup S_Z$ causes significant complications: We need to obtain exact integrals as even minor errors in the matrix entries can quickly lead to incorrect results. However, trying to use Mathematica to integrate a polynomial in $49$ variables over a region like this seems pretty hopeless: even for small degree polynomials and 10 variables a computation can already take 10 minutes or more, and the duration appears to increase exponentially, so that we would never get any result for $49$ variables. Therefore we need to find a better way to integrate polynomials over $S$. 
 
 \medskip Some of the key ideas of our integral computation are as follows: We first consider sets 
 \begin{align*} 
 &T_s(k)=\{(t_1, \dots, t_k) \in [0,\delta]^k: t_1+\dots+t_k \leq 1\}, \\ 
 &T_b(k)=\{(t_1, \dots, t_k) \in [\delta,1]^k: t_1+\dots+t_k \leq 1\} 
  \end{align*} 
 and polynomials $ t_1^{a_1} \dots t_k^{a_k} (1-t_1-\dots-t_k)^b$. The integrals of $ t_1^{a_1} \dots t_k^{a_k} (1-t_1-\dots-t_k)^b$ over $T_s(k)$ and $T_b(k)$ can be expressed as polynomials in $\delta$, with coefficients only dependent on the value of $\lfloor \frac{1}{\delta} \rfloor$. To compute these coefficients we find a recursion which relates an integral over $T_s(k)$ to integrals of polynomials of equal or lower degree over $T_s(k-1)$ and $T_b(k-1)$. We then also integrate over more complicated sets, like for instance $T_m=\{(t_1, \dots, t_r, t_{r+1}, \dots, t_k) \in [\delta,1]^r\times [0,\delta]^{k-r} : t_1+\dots+t_k \leq 1\}$. Such integrals can be expressed as   matrix multiplications,   using the coefficient vectors for $T_s(k)$ and $T_b(k)$ produced in the previous step. The actual integral  over $S=S_{BV} \cup S_Z$ is then split up into integrals over scaled $T_m$, so that we can use the results of the previous steps. 
 
 \medskip Overall, our calculation of the needed matrix entries involves no integration at all, it only requires matrix multiplication. However, many very large matrices need to be multiplied together, and so despite a lot of  effort  to speed up our code, it can still take   days to compute all necessary matrix entries for symmetric polynomials of degree $2a+b \leq 21$. Our code also uses a lot of memory, a trade off from focusing on computation speed. Due to these time and memory constraints, it was not practical for us  to consider symmetric polynomials of degree  $2a+b \leq 27$.  As such,  we believe that  with greater resources the same method could be used to get a better bound  on $H_1$.

\subsection{Structure of the paper} 

In Section~\ref{sec:gpysummary} we carry out the key steps of the GPY method with Maynard-Tao sieve weights and  Polymath's epsilon-enlargement trick. Our arguments  follow Polymath's approach~\cite{Polymath:2014:VSS}, but many adjustments need to be made   to deal with a more complicated function support and account for the use of  Harman's sieve. 

\smallskip

In Section~\ref{sec:equiest} we record the equidistribution estimates used in our proof. Although the relevant estimates are  closely based on~\cite{Polymath:2014:EDZ} and~\cite{Stadlmann:2023:PAP}, the moduli restrictions in these papers are too strong, and we need to prove relaxed conditions on the moduli. 

\smallskip

The function support used in Section~\ref{sec:gpysummary} is  quite general and involves many different parameters.   Dependent on these parameters, the  equidistribution estimates of Section~\ref{sec:equiest} may or may not be good enough to cover all relevant moduli. In Section~\ref{sec:domain} we show how to calculate good parameters for which the available  equidistribution estimates   work well enough.

\smallskip In Section~\ref{sec:code} we briefly discuss how integrals over the chosen support can be expressed in terms of matrix multiplication and recursion.  

\smallskip In Section~\ref{sec:proof} we combine the results of the previous sections to complete the proof of Theorem~\ref{thm}.

\subsection{Notation}  Below is a list of notation and terminology used throughout this paper.

\medskip
In this paper, lower case Roman letters  usually denote integers, and letters $p$ and $p_i$ denote primes.

 $1_S$ is the indicator function of set $S$ and  $1_{\mathbb{P}}$ is the prime indicator function. 

 $\mu(n)$ is the Möbius function and $\phi(n)$ is Euler's totient function.  $\tau(n)$ is the number of divisors of $n$.

\medskip
We say a positive integer $n $ is  $y$-smooth if all its prime factors are of size less than $y$. 

We write $P(y)= \prod_{p <y} p$. If $n \mid P(y)$, then $n$ is $y$-smooth and squarefree.

 \medskip
 We   use Vinogradov  notation ($\ll$ and $\gg$) and write $X \asymp Y$ if $Y \ll X \ll Y$. If we write $X \ll_A Y$, then the implied constant may depend on $A$.  

\section{The gpy sieve }\label{sec:gpysummary}

In this section we present a version of the GPY method. The GPY method was first developed by Goldston, Pintz and Yıldırım~\cite{Goldston:2009:PTI}, who proved that $\liminf_{n \rightarrow \infty} \frac{p_{n+1}-p_n}{\log p_n}=0$. Since then, it has been used in all proofs of upper bounds on $H_1$, from Zhang's proof of $H_1 \leq 70000000$~\cite{Zhang:2015:BGP} to Polymath's proof of $H_1 \leq 246$~\cite{Polymath:2014:VSS}. We will use the GPY sieve weights developed by Maynard in~\cite{Maynard:2015:SGP}, and a modification of   Polymath's epsilon enlargement trick with a new function support. Additionally, like Baker and Irving~\cite{Baker:2017:BIP}, we will use a minorant for the primes, constructed via Harman's sieve.

\subsection{Important definitions}

 Before we   state an amended version of the GPY sieve, we need to introduce some   definitions. We first define the function support we  use  to construct   Maynard-Tao sieve weights.

 \begin{defn}[Simplex subset]\label{def:intregion} Let $k \in \mathbb{N}$,   $\delta >0$ and $\varepsilon>0$. Let $n \in \mathbb{N}$ and $\underline{A}=(A_0,A_1, \dots,A_n)$ with  $$-\varepsilon =A_0 <  \dots < A_n <\frac{1}{2}-\varepsilon.$$ 
 Let $\underline{B}$ be a $n \times \lfloor \frac{1}{\delta} \rfloor$ matrix with  $(j,m)th$ entry equal $B_{j,m}$. Matrix  $\underline{B}$ must satisfy   that for any $m \geq 1$,  $$\delta<B_{j,m} \leq B_{j,m+1} \leq B_{j,m} + \delta.$$
 
  Then $T_k(\delta,\underline{A},\underline{B},\varepsilon)$ is defined as follows:
\begin{align*}
T_k(\delta,\underline{A},\underline{B},\varepsilon)=\bigcup_{j=1}^n \left\{(t_1, \dots, t_k) \in [0,1]^k:     \sum_{i=1}^k t_i \in \left[A_{j-1}+\varepsilon,  A_j+\varepsilon \right) \mbox{ and }  \sum_{i \in I} t_i \leq B_{j,|I|} \mbox{ for }  I=\{i: t_i > \delta\} \right\}.
\end{align*}
 \end{defn}
Specific values for parameters $A_j$ and $B_{j,m}$ will be chosen at a  later stage, once we have derived  equidistribution estimates for moduli with large $x^\delta$-smooth factors. 
 The condition $\delta<B_{j,m} \leq B_{j,m+1} \leq B_{j,m} + \delta$
 is quite important: If $\sum_{i \in I} t_i \leq B_{j,|I|} \mbox{ for }  I=\{i: t_i > \delta\}$, this condition implies that also $\sum_{i \in I'} t_i \leq B_{j,|I'|}$ for any $I' \subseteq I$.  This property will be used in several proofs in this section. 
 
 \medskip
 Next we define a set of corresponding moduli for which we need equidistribution estimates. 
 
 \begin{defn}[Relevant moduli]\label{def:neededmoduli}
Let $x >0$ and $\varepsilon_0>0$. Let $T_k(\delta,\underline{A},\underline{B},\varepsilon)$ be as given in Definition~\ref{def:intregion}. 

Further let $Q(x;\delta,\underline{A},\underline{B},\varepsilon,j,j',m,m',\varepsilon_0)$ be the following set of integers in $[1,x]$:
\begin{align*}\left\{ e e' \prod_{i=1}^m f_i \prod_{i'=1}^{m'} f'_{i'}    \in [1,x]: 
\begin{array}{ll}
  \log_x(f_1 \dots f_m) \leq (1-\varepsilon_0)B_{j,m}, & \log_x(f'_1 \dots f'_{m'}) \leq  (1-\varepsilon_0)B_{j',m'}, \\
  \log_x(f_1 \dots f_m e) \leq (1-\varepsilon_0)(A_j-\varepsilon), &  \log_x(f'_1 \dots f'_{m'} e') \leq (1-\varepsilon_0)(A_{j'}+\varepsilon), \\
ee' \mbox{ is } x^\delta\mbox{-smooth}, & \log_x(f_i), \log_x(f'_i) \geq  \delta \mbox{ for all } i
\end{array}\right\}
\end{align*}
We then say the moduli corresponding to region $T_k(\delta,\underline{A},\underline{B},\varepsilon)$ are the set 
\begin{align*}
Q^*(x;\delta,\underline{A},\underline{B},\varepsilon,\varepsilon_0) = \bigcup_{j=1}^n \bigcup_{j'=1}^n \bigcup_{m=0}^{\lfloor \frac{1}{\delta} \rfloor }
\bigcup_{m'=0}^{\lfloor \frac{1}{\delta} \rfloor } Q(x;\delta,\underline{A},\underline{B},\varepsilon,j,j',m,m',\varepsilon_0).
\end{align*} 
 \end{defn}
 
 Unlike in Polymath's proof, we will not always work directly with primes in arithmetic progressions. Instead we can replace the prime indicator function with a suitable minorant $\rho(n) \leq 1_{\mathbb{P}}(n)$, constructed via Harman's sieve, which satisfies an equidistribution statement of the type described below.
 
\begin{defn}[Equidistribution]\label{def:equidist}
Let  $T_k(\delta,\underline{A},\underline{B},\varepsilon)$ be as given in Definition~\ref{def:intregion} and let 
$Q^*(x;\delta,\underline{A},\underline{B},\varepsilon,\varepsilon_0)$ be as given in Definition~\ref{def:neededmoduli}.

We say that a function $f: \mathbb{N} \times (1,\infty) \rightarrow \mathbb{C}$ has equidistribution properties for  moduli corresponding to $T_k(\delta,\underline{A},\underline{B},\varepsilon)$ provided that the following holds: For  any $x>1$, $\varepsilon_0>0$ and  $A>0$, and for any $a \in \mathbb{Z}$ with $(a,p)=1$ for all primes $p \leq x$, we have
\begin{align*}
\sum_{\substack{q \in Q^*(x;\delta,\underline{A},\underline{B},\varepsilon,\varepsilon_0) \\  q \mbox{ \scriptsize  is squarefree  }   }} \Bigg| \sum_{\substack{n \in [x,2x] \\ n \equiv a (q)}} f(n;x) - \dfrac{1}{\phi(q)} \sum_{\substack{n \in [x,2x] \\(n,q)=1 }}f(n;x) \Bigg| \ll_{A,\varepsilon_0} \dfrac{x}{\log(x)^A}.
\end{align*}
\end{defn}

Next we write down the GPY sieve with a minorant of the prime indicator function. Recall here that $(h_1, \dots, h_k)$ is an admissible $k$-tuple if $h_1 < \dots < h_k$ are integers which avoid at least one residue class modulo $p$ for each prime $p$.  

\begin{defn}[GPY sieve]\label{def:gpysieve}
Let $(h_1, \dots, h_k)$ be an admissible $k$-tuple. 

Let $W= \prod_{p \leq \log\log\log x} p$ and choose $b \in \mathbb{N}$ with  $(b+h_i,W)=1$ for all $i \in \{1,\dots,k\}$. 

Let $\rho: \mathbb{N} \times (1,\infty) \rightarrow \mathbb{R}$ with $\rho(n;x) \leq 1_{\mathbb{P}}(n)$ for all large $x$ and $n \in [x,2x]$.

\medskip

For given $L \in \mathbb{N}$, $c_1, \dots, c_L \in \mathbb{R}$ and  smooth compactly supported functions $F_{l,i}:[0,\infty) \rightarrow \mathbb{R}$, we define the following sieve weights:
\begin{align*}
\nu(n)=\left(\sum_{l=1}^L c_l \lambda_{F_{l,1}}(n+h_1) \dots \lambda_{F_{l,k}}(n+h_k) \right)^2  \quad \mbox{ with  }  \quad \lambda_F(n)=\sum_{d \mid n } \mu(d) F(\log_x(d)).
\end{align*}

For the chosen weight function $\nu: \mathbb{N} \rightarrow [0,\infty)$, we then define the GPY sieve function
\begin{align*}
&N(x;\rho,\nu)=  \sum_{\substack{x \leq n \leq 2x \\  n \equiv b(W) }} \nu(n) \left( \sum_{i=1}^k \rho(n+h_i;x) - 1 \right).
\end{align*}
\end{defn}
If $N(x;\rho,\nu) >0$ then there must exist some $n_0 \in [x,2x]$  for which $\sum_{i=1}^k \rho(n_0+h_i;x) > 1 $. But then    $(n_0+h_1, \dots, n_0+h_k)$ must contain at least two primes. Thus $N(x;\rho,\nu) > 0$ for large $x$ would imply that $H_1 \leq h_k - h_1$. The difficulty is thus  to find a choice of $F_{l,i}$ (and thus a choice of $\nu$) for which we can prove the inequality $N(x;\rho,\nu) > 0$. This can be reduced to an optimization problem involving integrals:

\begin{defn}[Key integrals]\label{def:int}
Let $T_k(\delta,\underline{A},\underline{B},\varepsilon)$ be as given in Definition~\ref{def:intregion}. Let $F$ be a symmetric, square-integrable function supported on $T_k(\delta,\underline{A},\underline{B},\varepsilon)$. Then   define
\begin{align*} 
&I(F;\delta,\underline{A},\underline{B},\varepsilon)=\mathop{\int \dots \int}_{T_k(\delta,\underline{A},\underline{B},\varepsilon)} F(t_1, \dots, t_k)^2 \mbox{d}t_1 \dots \mbox{d}t_k, \\
&J(F;\delta,\underline{A},\underline{B},\varepsilon)=\sum_{m=1}^n\sum_{m'=1}^n\mathop{\int \dots \int}_{\substack{t_1+\dots+t_{k-1} \leq \max\{A_m-\varepsilon,A_{m'}-\varepsilon\} \\ t_1+\dots+t_{k-1}+t_k \in [A_{m-1}+\varepsilon,A_m+\varepsilon] \\  t_1+\dots + t_{k-1}+t'_k \in [A_{m'-1}+\varepsilon,A_{m'}+\varepsilon]  }} \!\!\!\!\!\!\!\!\!\!\!\!\!\!\!\!\!\!\!\!\!\!\!\!\! F(t_1,\dots, t_{k-1}, t_k)  F(t_1, \dots, t_{k-1}, t'_k)\mbox{d}t_1 \dots   \mbox{d}t_{k-1}\mbox{d}t_k\mbox{d}t'_k,\\
&K(F;\delta,\underline{A},\underline{B},\varepsilon)=\sum_{m=1}^n\sum_{m'=1}^n\mathop{\int \dots \int}_{\substack{ t_1 + \dots + t_{k-1} > \max\{A_m-\varepsilon,A_{m'}-\varepsilon\}\\ t_1+\dots+t_{k-1}+t_k \in [A_{m-1}+\varepsilon,A_m+\varepsilon] \\  t_1+\dots + t_{k-1}+t'_k \in [A_{m'-1}+\varepsilon,A_{m'}+\varepsilon]  }} F(t_1, \dots, t_k)^2 \mbox{d}t_1 \dots \mbox{d}t_k.
\end{align*} 
\end{defn}

\subsection{Main sieve statement}

\medskip We are now ready to state the main result of this section: 
\begin{prop}\label{prop:GPYsieve}
Let  $T_k(\delta,\underline{A},\underline{B},\varepsilon)$ be as given in Definition~\ref{def:intregion}. Suppose that function $\rho: \mathbb{N} \times (1,\infty) \rightarrow \mathbb{R}$ and constants $c_1$ and $c_2$ have the following properties:
\begin{enumerate}[{\rm(1)}]
\item  Prime minorant: $-c_2 \leq \rho(n;x) \leq 1_{\mathbb{P}}(n)$ for all large $x$ and $n \in [x,2x]$.  
\item  $\rho$ has  equidistribution properties for  moduli corresponding to $T_k(\delta,\underline{A},\underline{B},\varepsilon)$. $($See Definition~\ref{def:equidist}.$)$
\item If $\rho(n;x) \neq 0$, then all prime factors of $n$ exceed $x^\beta$, for some $\beta > \max\{ B_{j,1}: 1\leq j \leq n \}$.  
\item The sum of $\rho(n;x)$ over $[x,2x]$ satisfies
$\sum_{n \in [x,2x]} \rho(n;x) = (1-c_1 +o(1)) \frac{x}{\log(x)}$.
\end{enumerate}
Further let $I(F;\delta,\underline{A},\underline{B},\varepsilon)$,  $J(F;\delta,\underline{A},\underline{B},\varepsilon)$ and $K(F;\delta,\underline{A},\underline{B},\varepsilon)$ be as given in Definition~\ref{def:int}. Suppose there exists a symmetric, square-integrable function $F$ supported on $T_k(\delta,\underline{A},\underline{B},\varepsilon)$ with 
\begin{align}\label{inequ:intmain}
\frac{k(1-c_1)  J(F;\delta,\underline{A},\underline{B},\varepsilon)-k c_2   K(F;\delta,\underline{A},\underline{B},\varepsilon)}{I(F;\delta,\underline{A},\underline{B},\varepsilon)} > 1.
\end{align}
Then $H_1 = \liminf_{n \rightarrow \infty} (p_{n+1}-p_n) \leq H(k)$, where $H(k)$ is the diameter of the shortest admissible $k$-tuple.
\end{prop}

\subsection{Construction of sieve weights} 

We first show how a function $F$ satisfying~(\ref{inequ:intmain}) is used to construct sieve weights $\nu(n)$. Our arguments   follow the steps of Section~5 of Polymath~\cite{Polymath:2014:VSS}, adapted to the new support $T_k(\delta,\underline{A},\underline{B},\varepsilon)$. 

\begin{lemma}\label{lem:sieveweights} Let  $T_k(\delta,\underline{A},\underline{B},\varepsilon)$ be as described in Definition~\ref{def:intregion}. Let $\varepsilon_0 >0$ and define region
\begin{align*}
&R_k^+(\delta,\underline{A},\underline{B},\varepsilon,j,\varepsilon_0)=  \left\{(t_1, \dots, t_k) \in [0,1]^k:    \begin{array}{l} \sum_{i=1}^k t_i  < (1-\varepsilon_0)(A_j+\varepsilon), \\  \sum_{i \in I} t_i <(1-\varepsilon_0) B_{j,|I|} \mbox{ for }  I=\{i: t_i > \delta\}     
\end{array}\right\}.
\end{align*}

Suppose there exists a symmetric, square-integrable function $F$ supported on $T_k(\delta,\underline{A},\underline{B},\varepsilon)$ with 
\begin{align*}
\frac{k (1-c_1) J(F;\delta,\underline{A},\underline{B},\varepsilon) - kc_2 K(F;\delta,\underline{A},\underline{B},\varepsilon)}{I(F;\delta,\underline{A},\underline{B},\varepsilon)} > 1.
\end{align*}
If $\varepsilon_0>0$ is chosen sufficiently small, we can then  find integers $L_1, \dots, L_n$, constants $c_{j,l}$    and   smooth compactly supported functions $f_{j,l,i}: [0,\infty) \rightarrow \mathbb{R}$  with the following two properties:
\begin{enumerate}[{\rm(1)}]
\item For every  $j \in \{1, \dots, n\}$ and $l \in \{1, \dots, L_j\}$, the support of  product $f_{j,l,1}(t_1) \dots f_{j,l,k}(t_k)$ is contained in $  R_k^+(\delta,\underline{A},\underline{B},\varepsilon,j,\varepsilon_0)$.
\item  Denoting by $\mathcal{L}(j,i)$  the set of $ l \in  \{1, \dots, L_j\}$ for which the  support of  product $\prod_{s\neq i} f_{j,l,s}(t_s) $ is contained in $\{(t_1, \dots, t_{i-1},t_{i+1}, \dots, t_k): \sum_{s \neq i} t_i < (1-\varepsilon_0)(A_j-\varepsilon)\}$, and denoting by    $\mathcal{U}(j,i)$  its   complement $\{1, \dots, L_j\} \setminus \mathcal{L}(j,i)$, 
the terms 
\begin{align*}
\mathcal{I}&=\sum_{j=1}^n\sum_{l=1}^{L_j} \sum_{j'=1}^n \sum_{l'=1}^{L_{j'}} c_{j,l} c_{j',l'}  \prod_{s=1}^k
\int_0^\infty    f'_{j,l,s}(t_s)f'_{j',l',s}(t_s) \, \mbox{d}t_s, \\
\mathcal{J}_i&=\sum_{j=1}^n \sum_{l \in  \mathcal{L}(j,i)}  \sum_{j'=1}^n  \sum_{l' \in  \mathcal{L}(j',i)}  c_{j,l}c_{j',l'} f_{j,l,i}(0)  f_{j',l',i}(0)  \prod_{s\neq i} \mathop{\int_0^\infty   }  f'_{j,l,s}(t_s)     f'_{j',l',s}(t_s)   \, \mbox{d}t_s \\
&+2\sum_{j=1}^n \sum_{l \in  \mathcal{U}(j,i)}  \sum_{j'=1}^n  \sum_{l' \in  \mathcal{L}(j',i)} c_{j,l}c_{j',l'} f_{j,l,i}(0)  f_{j',l',i}(0)  \prod_{s\neq i} \mathop{\int_0^\infty   }  f'_{j,l,s}(t_s)     f'_{j',l',s}(t_s)   \, \mbox{d}t_s\\
\mathcal{K}_i&=\sum_{j=1}^n \sum_{l \in  \mathcal{U}(j,i)}  \sum_{j'=1}^n  \sum_{l' \in  \mathcal{U}(j',i)}  c_{j,l}c_{j',l'}   \prod_{s=1}^k \mathop{\int_0^\infty} f'_{j,l,s}(t_s)     f'_{j',l',s}(t_s)    \, \mbox{d}t_s
\end{align*}
satisfy the inequality 
\begin{align*}
\frac{(1-c_1) \sum_{i=1}^k  \mathcal{J}_i - c_2 \sum_{i=1}^k  \mathcal{K}_i}{ \mathcal{I}} > 1.
\end{align*}
\end{enumerate}

\end{lemma}

Our choice of sieve weights in the GPY method will then be as follows:
 $$\nu(n)=\left(\sum_{j=1}^n\sum_{l=1}^{L_j} c_{j,l} \lambda_{f_{j,l,1}}(n+h_1) \dots \lambda_{f_{j,l,k}}(n+h_k)  \right)^2$$    

\begin{proof}
Let $F:[0,\infty)^k\rightarrow \mathbb{R}$ be a symmetric, square-integrable function, supported on $T_k(\delta,\underline{A},\underline{B},\varepsilon)$, which satisfies the inequality
\begin{align} \label{inequ:intstep1}
\frac{k (1-c_1) J(F;\delta,\underline{A},\underline{B},\varepsilon) - kc_2 K(F;\delta,\underline{A},\underline{B},\varepsilon)}{I(F;\delta,\underline{A},\underline{B},\varepsilon)} > 1.
\end{align}
Let $S_j = \{(t_1, \dots, t_k) \in [0,1]^k:     \sum_{i=1}^k t_i \in \left[A_{j-1}+\varepsilon,  A_j+\varepsilon \right) \mbox{ and }  \sum_{i \in I} t_i \leq B_{j,|I|} \mbox{ for }  I=\{i: t_i > \delta\} \}$ and let $1_{S_j}(n)$ be the indicator function of set $S_j$. Since $\{S_1, \dots, S_n\}$ is a partition of $T_k(\delta,\underline{A},\underline{B},\varepsilon)$, we  then have  $F= \sum_{j=1}^n F \cdot 1_{S_j}$.  

\medskip
 
Next we choose   small $\varepsilon_1, \varepsilon_2>0$ and define $F_{2,j}:[0,\infty)^k\rightarrow \mathbb{R}$ as follows:
\begin{align*}
F_{2,j}(t_1,\dots,t_k)=\begin{cases} &F\left(\frac{t_1-\varepsilon_2}{1-\varepsilon_1},\dots,\frac{t_k-\varepsilon_2}{1-\varepsilon_1}\right) \cdot 1_{S_j}\!\left(\frac{t_1-\varepsilon_2}{1-\varepsilon_1},\dots,\frac{t_k-\varepsilon_2}{1-\varepsilon_1}\right) \quad  \mbox{ if } t_1, \dots, t_k \geq \varepsilon_2, \\
& 0 \qquad\qquad\qquad\qquad\qquad\qquad\qquad\qquad\qquad\qquad \mbox{ otherwise.} \end{cases}
\end{align*}

We also take $\varepsilon_0 = \frac{\varepsilon_1}{100}$ and define the integrals 
\begin{align*} 
&I(F,G)=\int_0^\infty \dots \int_0^\infty F(t_1, \dots, t_k)G(t_1, \dots, t_k) \,\, \mbox{d}t_1 \dots \mbox{d}t_k, \\
&J_i(F,G;B)=  \mathop{\int_0^\infty \dots \int_0^\infty}_{\substack{\sum_{s \neq i } t_s < B   }}  F(t_1,\dots, t_{i-1}, t_i, t_{i+1}, \dots, t_k)  G(t_1, \dots, t_{i-1}, t'_i, t_{i+1}, \dots, t_k) \,  \mbox{d}t'_i\mbox{d}t_1 \dots  \mbox{d}t_k, \\
&K_i(F,G;B)=  \mathop{\int_0^\infty \dots \int_0^\infty}_{\substack{\sum_{s \neq i } t_s > B   }}  F(t_1,\dots, t_k)  G(t_1, \dots, t_k) \, \mbox{d}t_1 \dots  \mbox{d}t_k, 
\end{align*} 
 Since  $F$ is square-integrable and symmetric, and satisfies inequality  (\ref{inequ:intstep1}), we can choose a very small  $\varepsilon_1>0$   and  an even smaller $\varepsilon_2>0$ so that our shifted and scaled functions $F_{2,j}$ still satisfy 
\begin{align}\label{inequ:proofstep2}
&\frac{\sum_{i=1}^k\sum_{j=1}^n \sum_{j'=1}^n ( (1-c_1) J_i(F_{2,j},F_{2,j'};B^*(j,j')) - c_2 K_i(F_{2,j},F_{2,j'};B^*(j,j'))) }{\sum_{j=1}^n \sum_{j'=1}^n I(F_{2,j},F_{2,j'})} > 1 \\
&\mbox{where } B^*(j,j')=(1-\varepsilon_0)\max\{A_j-\varepsilon,A_{j'}-\varepsilon\}. \nonumber
\end{align}
Furthermore, recall  the assumption  $B_{j,m} \leq B_{j,m+1} \leq B_{j,m}+\delta$. Hence $\sum_{i \in I} t_i \leq B_{j,|I|} \mbox{ for }  I=\{i: t_i > \delta\}$   implies $\sum_{i \in I'} t_i \leq B_{j,|I'|}$ for any $I' \subseteq I$. In particular, if $(t_1, \dots, t_k)$ is in the support of $F_{2,j}$, and $L=\{i:t_i>\delta\}$, then for any $i \in L$,  also $\frac{t_i-\varepsilon_2}{1-\varepsilon_1} > \delta$ and so $\sum_{i \in L} (\frac{t_i-\varepsilon_2}{1-\varepsilon_1}) \leq B_{j,|L|}$. Hence, for $\varepsilon_2$ much smaller than $\varepsilon_1$, we see that the support of  $F_{2,j}$ is   contained inside $R_k^+(\delta,\underline{A},\underline{B},\varepsilon,j,2\varepsilon_0)$, and it actually  stays away from all boundary faces of $R_k^+(\delta,\underline{A},\underline{B},\varepsilon,j,2\varepsilon_0)$.   Thus we can convolve each $F_{2,j}$ with a smooth approximation to the identity, supported close to the origin, to obtain   functions $F_{3,j}:[0,\infty)^k \rightarrow \mathbb{R}$ which are smooth, supported on  $R_k^+(\delta,\underline{A},\underline{B},\varepsilon,j,2\varepsilon_0)$ and  satisfy the inequality
\begin{align}\label{inequ:proofstep3}
&\frac{\sum_{i=1}^k\sum_{j=1}^n \sum_{j'=1}^n ( (1-c_1) J_i(F_{3,j},F_{3,j'};B^*(j,j')) - c_2 K_i(F_{3,j},F_{3,j'};B^*(j,j'))) }{\sum_{j=1}^n \sum_{j'=1}^n I(F_{3,j},F_{3,j'})} > 1.
\end{align}
Following the arguments of Polymath~\cite{Polymath:2014:VSS}, we next define $f_{3,j}:[0,\infty)^k \rightarrow \mathbb{R}$ by setting 
\begin{align*}
f_{3,j}(t_1, \dots, t_k) = \int_{t_1}^\infty \dots \int_{t_k}^\infty F_{3,j}(s_1, \dots, s_k) \,\mbox{d}s_1 \dots \mbox{d}s_k.
\end{align*}
Yet again, the condition  $B_{j,m} \leq B_{j,m+1} \leq B_{j,m}+\delta$ ensures that the support of $f_{3,j}$ stays contained in $R_k^+(\delta,\underline{A},\underline{B},\varepsilon,j,2\varepsilon_0)$. Furthermore, 
\begin{align*} 
&I(F_{3,j},F_{3,j'})=\int_0^\infty \dots \int_0^\infty  \frac{\partial^k}{\partial t_1 \dots \partial t_k} f_{3,j}(t_1, \dots, t_k)\frac{\partial^k}{\partial t_1 \dots \partial t_k} f_{3,j'}(t_1, \dots, t_k) \,\, \mbox{d}t_1 \dots \mbox{d}t_k, \\
&J_i(F_{3,j},F_{3,j'};B)=  \mathop{\int_0^\infty  \!\!\dots \int_0^\infty}_{\substack{\sum_{s \neq i } t_s < B   }} \frac{\partial^{k-1}}{\prod_{s \neq i} \partial t_s}   f_{3,j}(t_1, \dots, t_{i-1}, 0, t_{i+1}, \dots, t_k) \frac{\partial^{k-1}}{\prod_{s \neq i} \partial t_s}   f_{3,j'}(t_1, \dots, t_{i-1}, 0, t_{i+1}, \dots, t_k) \,   \prod_{s \neq i} \mbox{d}t_s, \\
&K_i(F_{3,j},F_{3,j'};B)=  \mathop{\int_0^\infty  \!\!\dots \int_0^\infty}_{\substack{\sum_{s \neq i } t_s > B   }} \frac{\partial^k}{\partial t_1 \dots \partial t_k} f_{3,j}(t_1, \dots, t_k)\frac{\partial^k}{\partial t_1 \dots \partial t_k} f_{3,j'}(t_1, \dots, t_k) \,\, \mbox{d}t_1 \dots \mbox{d}t_k.
\end{align*} 
We now refer to $I(F_{3,j},F_{3,j'})$ by $I^*(f_{3,j},f_{3,j'})$, we refer to $J_i(F_{3,j},F_{3,j'};B)$ by $J_i^*(f_{3,j},f_{3,j'};B)$ and we refer to $K_i(F_{3,j},F_{3,j'};B)$ by $K_i^*(f_{3,j},f_{3,j'};B)$. Hence  
\begin{align}\label{inequ:proofstep4}
&\frac{\sum_{i=1}^k\sum_{j=1}^n \sum_{j'=1}^n ( (1-c_1) J^*_i(f_{3,j},f_{3,j'};B^*(j,j')) - c_2 K^*_i(f_{3,j},f_{3,j'};B^*(j,j'))) }{\sum_{j=1}^n \sum_{j'=1}^n I^*(f_{3,j},f_{3,j'})} > 1 \\
&\mbox{where } B^*(j,j')=(1-\varepsilon_0)\max\{A_j-\varepsilon,A_{j'}-\varepsilon\}. \nonumber
\end{align}

\medskip
Next we choose another very small constant $\varepsilon_3>0$ and (like in Section~5.3 of~\cite{Polymath:2014:VSS})  apply  the Stone-Weierstrass theorem to each function $f_{3,j}$, expressing it as the smooth limit of linear combinations of the form $\sum c_l f_{4,j,l,1}(t_1) \dots  f_{4,j,l,k}(t_k)$, where each $ f_{4,j,l,i}(t_i)$ is a   smooth function    supported on an  interval of length at most $\varepsilon_3$.   Overall, we can find
 integers $L_1, \dots, L_n$, constants $c_{j,l}$ and   and smooth   functions $f_{j,l,i}: [0,\infty) \rightarrow \mathbb{R}$   with the following properties:
\begin{enumerate}[{\rm (i)}]
\item Each $f_{j,l,i}(t)$   is smooth and supported on an interval of length at most $\varepsilon_3$.
\item For every  $j \in \{1, \dots, n\}$ and $l \in \{1, \dots, L_j\}$, the support of  product $f_{j,l,1}(t_1) \dots f_{j,l,k}(t_k)$ is contained in $  R_k^+(\delta,\underline{A},\underline{B},\varepsilon,j,\varepsilon_0)$.
\item For each $j \in \{1, \dots, n\}$,  sum $$h_j(t_1, \dots, t_k)=\sum_{l=1}^{L_j} c_{j,l} f_{j,l,1}(t_1) \dots  f_{j,l,k}(t_k) $$   approximates   function $f_{3,j}(t_1, \dots, t_k)$ so well that the following is true:
\begin{align}\label{inequ:proofstep5}
&\frac{\sum_{i=1}^k\sum_{j=1}^n \sum_{j'=1}^n ( (1-c_1) J^*_i(h_{j},h_{j'};B^*(j,j')) - c_2 K^*_i(h_{j},h_{j'};B^*(j,j'))) }{\sum_{j=1}^n \sum_{j'=1}^n I^*(h_{j},h_{j'})} > 1.
\end{align}
\end{enumerate}
Now we take a closer look at $I^*$,  $J^*_i$ and $K^*_i$. The expression $\sum_{j=1}^n \sum_{j'=1}^n I^*(h_{j},h_{j'})$   rearranges to 
\begin{align*}
 \sum_{j=1}^n\sum_{l=1}^{L_j} \sum_{j'=1}^n \sum_{l'=1}^{L_{j'}} c_{j,l} c_{j',l'} 
\int_0^\infty \dots \int_0^\infty   f'_{j,l,1}(t_1) \dots  f'_{j,l,k}(t_k)  f'_{j',l',1}(t_1) \dots  f'_{j',l',k}(t_k)  \,\, \mbox{d}t_1 \dots \mbox{d}t_k.
\end{align*}
But this is simply $\mathcal{I}$. On the other hand, consider $\sum_{j=1}^n \sum_{j'=1}^n J^*_i(h_{j},h_{j'};(1-\varepsilon_0)\max\{A_j-\varepsilon,A_{j'}-\varepsilon\})$ for some $i \in \{1, \dots, k\}$. It equals the following expression: 
\begin{align*}
\mathcal{J}^*_i&=\sum_{j=1}^n \sum_{l=1}^{L_j}  \sum_{j'=1}^n  \sum_{l'=1}^{L_{j'}} c_{j,l}c_{j',l'} f_{j,l,i}(0)  f_{j',l',i}(0)   \!\!\!\!\!\!\!\!\!\!\!\!\!\!\!\!\!\!\!\! \mathop{\int_0^\infty  \!\!\dots \int_0^\infty}_{\substack{\sum_{s \neq i } t_s < (1-\varepsilon_0)\max\{A_j-\varepsilon,A_{j'}-\varepsilon\}  }} \prod_{s\neq i} f'_{j,l,s}(t_s)    \prod_{s\neq i} f'_{j',l',s}(t_s)    \prod_{s \neq i} \mbox{d}t_s. 
\end{align*}

We now partition each $\{1, \dots, L_j\}$ as follows: $\mathcal{L}(j,i)$ is the set of $ l \in  \{1, \dots, L_j\}$ for which the  support of  product $\prod_{s\neq i} f_{j,l,s}(t_s) $ is contained in $\{(t_1, \dots, t_{i-1},t_{i+1}, \dots, t_k): \sum_{s \neq i} t_i < (1-\varepsilon_0)(A_j-\varepsilon)\}$. The set $\mathcal{U}(j,i)$ is the complement $\{1, \dots, L_j\} \setminus \mathcal{L}(j,i)$. 

\medskip  The union of the supports of functions  $\prod_{s\neq i} f_{j,l,s}(t_s)\prod_{s\neq i} f_{j',l',s}(t_s) $ with $l \in \mathcal{U}(j,i)$ and $l' \in \mathcal{U}(j',i)$  only overlaps with  
$\{(t_1, \dots, t_{i-1},t_{i+1}, \dots, t_k): \sum_{s \neq i} t_i < (1-\varepsilon_0)\max\{A_j-\varepsilon,A_{j'}-\varepsilon\}\}$ on a set of measure $O(\varepsilon_3)$.  Hence for any given $\varepsilon_4>0$, we can choose the value of $\varepsilon_3$ sufficiently small that the contribution of pairs $(l,l') \in \mathcal{U}(j,i) \times \mathcal{U}(j',i)$  to $\mathcal{J}^*_i$ is bounded by $\varepsilon_4$ and 
\begin{align*}
\mathcal{J}_i&=\sum_{j=1}^n \sum_{l \in  \mathcal{L}(j,i)}  \sum_{j'=1}^n  \sum_{l' \in  \mathcal{L}(j',i)}  c_{j,l}c_{j',l'} f_{j,l,i}(0)  f_{j',l',i}(0)   \mathop{\int_0^\infty  \!\!\dots \int_0^\infty}\prod_{s\neq i} f'_{j,l,s}(t_s)    \prod_{s\neq i} f'_{j',l',s}(t_s)    \prod_{s \neq i} \mbox{d}t_s \\
&+2\sum_{j=1}^n \sum_{l \in  \mathcal{U}(j,i)}  \sum_{j'=1}^n  \sum_{l' \in  \mathcal{L}(j',i)}  c_{j,l}c_{j',l'} f_{j,l,i}(0)  f_{j',l',i}(0)   \mathop{\int_0^\infty  \!\!\dots \int_0^\infty}\prod_{s\neq i} f'_{j,l,s}(t_s)    \prod_{s\neq i} f'_{j',l',s}(t_s)    \prod_{s \neq i} \mbox{d}t_s  \\
&> \sum_{j=1}^n \sum_{j'=1}^n J^*_i(h_{j},h_{j'};(1-\varepsilon_0)\max\{A_j-\varepsilon,A_{j'}-\varepsilon\}) -\varepsilon_4.
\end{align*}

Finally, pairs  $(l,l') $ with $l \in \mathcal{L}(j,i)$ or $l' \in \mathcal{L}(j',i)$  do not contribute at all to    $\sum_{j=1}^n \sum_{j'=1}^n K^*_i(h_{j},h_{j'};B^*(j,j'))$, as their support does not overlap with the integration region. Thus the sum of $K^*_i$ equals 
\begin{align*}
\mathcal{K}^*_i&=\sum_{j=1}^n \sum_{l \in \mathcal{U}(j,i)} \sum_{j'=1}^n  \sum_{l' \in \mathcal{U}(j',i)}  c_{j,l}c_{j',l'}   \!\!\!\!\!\!\!\!\!\!\!\!\!\!\!\!\!\!\!\! \mathop{\int_0^\infty  \!\!\dots \int_0^\infty}_{\substack{\sum_{s \neq i } t_s > (1-\varepsilon_0)\max\{A_j-\varepsilon,A_{j'}-\varepsilon\}  }} \prod_{s=1}^k f'_{j,l,s}(t_s)    \prod_{s=1}^k f'_{j',l',s}(t_s)  \,  \mbox{d}t_1 \dots  \mbox{d}t_k. 
\end{align*}
Since the support of functions  $\prod_{s\neq i} f_{j,l,i}(t_i)\prod_{s\neq i} f_{j',l',i}(t_i) $ with $l \in \mathcal{U}(j,i)$ and $l' \in \mathcal{U}(j',i)$  only overlaps with  
$\{(t_1, \dots, t_{i-1},t_{i+1}, \dots, t_k): \sum_{s \neq i} t_i < (1-\varepsilon_0)\max\{A_j-\varepsilon,A_{j'}-\varepsilon\}\}$ on a set of measure $O(\varepsilon_3)$, we can drop the condition $\sum_{s \neq i } t_s > (1-\varepsilon_0)\max\{A_j-\varepsilon,A_{j'}-\varepsilon\}$ with a loss of at most $\varepsilon_4$, provided that $\varepsilon_3$ was chosen small enough. Overall we then have 
\begin{align*}
\mathcal{K}_i&=\sum_{j=1}^n \sum_{l \in \mathcal{U}(j,i)} \sum_{j'=1}^n  \sum_{l' \in \mathcal{U}(j',i)}  c_{j,l}c_{j',l'}     \mathop{\int_0^\infty  \!\!\dots \int_0^\infty}  \prod_{s=1}^k f'_{j,l,s}(t_s)    \prod_{s=1}^k f'_{j',l',s}(t_s)  \,  \mbox{d}t_1 \dots  \mbox{d}t_k \\
&< \sum_{j=1}^n \sum_{j'=1}^n K^*_i(h_{j},h_{j'};B^*(j,j')) + \varepsilon_4.
\end{align*}
Looking back at inequality (\ref{inequ:proofstep5}) we therefore have 
\begin{align*}
\frac{(1-c_1) \sum_{i=1}^k  \mathcal{J}_i - c_2 \sum_{i=1}^k  \mathcal{K}_i + ((1-c_1)k + c_2 k) \varepsilon_4}{ \mathcal{I}} > 1.
\end{align*}
Hence, taking $\varepsilon_4$ sufficiently small (and the corresponding $\varepsilon_3$ even smaller), we arrive at the desired inequality $\frac{(1-c_1) \sum_{i=1}^k  \mathcal{J}_i - c_2 \sum_{i=1}^k  \mathcal{K}_i}{ \mathcal{I}} > 1$.
 \end{proof}

 \subsection{Asymptotics} 
 Having chosen sieve weights $\nu(n)$, we now  need to prove asymptotics for sums $\sum_{x \leq n \leq 2x} \nu(n) \sum_{i=1}^k \rho(n+h_i; x)$.  We first look at $\sum_{x \leq n \leq 2x}  \rho(n+h_i; x) \prod_{s =1}^k \lambda_{F_s}(n+h_s) \lambda_{G_s}(n+h_s) $.  Here we can follow the arguments of Section~4.2 of Polymath~\cite{Polymath:2014:VSS}, with  small amendments due to use of Harman's sieve (like in Baker and Irving's work~\cite{Baker:2017:BIP}) and due to the new support. 
 
 \begin{lemma}\label{lem:asymptotics}  Suppose that function $\rho: \mathbb{N} \times (1,\infty) \rightarrow \mathbb{R}$ has the following properties:
\begin{enumerate}[{\rm(1)}]
\item Prime minorant: $-c_2 \leq \rho(n;x) \leq 1_{\mathbb{P}}(n)$ for all large $x$ and $n \in [x,2x]$.  
\item  $\rho$ has  equidistribution properties for  moduli corresponding to $T_k(\delta,\underline{A},\underline{B},\varepsilon)$. $($See Definition~\ref{def:equidist}.$)$
\item If $\rho(n;x) \neq 0$, then all prime factors of $n$ exceed $x^\beta$, for some $\beta > \max\{ B_{j,1}: 1\leq j \leq n \}$.  
\item The sum of $\rho(n;x)$ over $[x,2x]$ satisfies
$\sum_{n \in [x,2x]} \rho(n;x) = (1-c_1 +o(1)) \frac{x}{\log(x)}$.
\end{enumerate}
 Let $i_0 \in \{1, \dots, k\}$.  Consider   smooth compactly supported functions $F_{i}: [0,\infty) \rightarrow \mathbb{R}$ and $G_{i}: [0,\infty) \rightarrow \mathbb{R}$. 

Let $R_k^+(\delta,\underline{A},\underline{B},\varepsilon,j,\varepsilon_0)$ be as defined in Lemma~\ref{lem:sieveweights}. Suppose that   the support of  product $F_{1}(t_1) \dots F_{k}(t_k)$  is contained in the set $  R_k^+(\delta,\underline{A},\underline{B},\varepsilon,j,\varepsilon_0)$ and the support of   product $G_{1}(t_1) \dots G_{k}(t_k)$ is  contained in   $  R_k^+(\delta,\underline{A},\underline{B},\varepsilon,j',\varepsilon_0)$. 
Finally, suppose that the   support of  product $\prod_{i\neq i_0} F_{i}(t_i) $ is contained in the set $\{(t_1, \dots, t_{i_0-1},t_{i_0+1}, \dots, t_k): \sum_{i \neq i_0} t_i < (1-\varepsilon_0)(A_j-\varepsilon)\}$.  

\smallskip

Let $W= \prod_{p \leq \log\log\log x} p$ and   $b \in \mathbb{N}$ with  $(b+h_i,W)=1$ for all $i \in \{1,\dots,k\}$.   Then
\begin{align*}
&\sum_{\substack{x \leq n \leq 2x \\ n \equiv b (W)}}  \rho(n+h_{i_0}; x) \prod_{i =1}^k \lambda_{F_i}(n+h_i) \lambda_{G_i}(n+h_i) \\ &= \Bigg((1-c_1)  F_{k}(0) G_{k}(0) \prod_{i \neq i_0} \int_0^\infty F'_i(t_i)G'_i(t_i) \mbox{d}t_i+o(1)\Bigg)  \left( \frac{ xW^{k-1}}{\phi(W)^{k} \log(x)^{k}}\right).
\end{align*} \end{lemma}
 
 \begin{proof}
 Without loss of generality we assume that $i_0 =k$.  
 Recall that $\lambda_F(n)= \sum_{d \mid n} \mu(d) F(\log_x(d))$. We call the sum we are interested in $S$. Expanding, we have
 \begin{align*}
S&:=\sum_{\substack{x \leq n \leq 2x \\ n \equiv b (W)}}  \rho(n+h_{k}; x) \prod_{i =1}^k \lambda_{F_i}(n+h_i) \lambda_{G_i}(n+h_i) \\ &= \sum_{d_1, \dots, d_k, d'_1, \dots, d'_k }  \prod_{i=1}^k \mu(d_i) \mu(d'_i) F_i(\log_x(d_i))  G_i(\log_x(d'_i))\sum_{\substack{x \leq n \leq 2x \\ n \equiv b (W) \\ n \equiv -h_i [d_i, d'_i] \, \forall i}}  \rho(n+h_{k}; x).
\end{align*} 
 
 If $\rho(n+h_{k};x) \neq 0$, then any $d \mid n+h_{k}$ with $d \neq 1$ satisfies $\log_x(d) >   B_{j,1}$. But  by construction we have $B_{j,m} \leq B_{j,m+1} \leq B_{j,m} + \delta$, and so $\log_x(d) >   B_{j,1}$ implies $\log_x(d)+(m-1) \delta > B_{j,m}$ for all $m$. Thus any tuple $(\log_x(d_1), \dots, \log_x(d_k))$ with ($d_{k} \mid n+h_{k}$ and $d_{k} \neq 1$) cannot be in the support of product $F_1(t_1) \dots F_k(t_k)$. Only $d_{k} = 1$ contributes to $S$. Similarly, only $d'_{k}=1$ contributes to $S$ and  so we have  
  \begin{align*}
S&=  F_{k}(0) G_{k}(0) \sum_{d_1, \dots, d_k, d'_1, \dots, d'_k } \prod_{i=1}^{k-1}  \mu(d_i) \mu(d'_i) F_i(\log_x(d_i))  G_i(\log_x(d'_i))\sum_{\substack{x \leq n \leq 2x \\ n \equiv b (W) \\ n \equiv -h_i [d_i, d'_i] \, \forall i \neq k}}  \rho(n+h_{k}; x).
\end{align*} 
Now we use the $W$-trick as in Polymath~\cite{Polymath:2014:VSS}:   $b$ was chosen to satisfy $(b+h_i,W)=1$ for all $i$. Assuming   that $x$ is sufficiently large to ensure $|h_i - h_{i'}| < \log\log\log x$ for $i \neq i'$, the system of congruences given above thus can only have solutions if the lcms $[d_i,d'_i]$ are coprime to each other and to $W = \prod_{p \leq \log\log \log x} p$.  

  The conditions ($n \equiv b (W)$ and $n \equiv -h_i [d_i, d'_i] $ for $i \neq k$) can thus be replaced by a single congruence $n + h_k \equiv a_{W,d_1, \dots, d'_{k-1}} \mbox{ mod } q_{W,d_1, \dots, d'_{k-1}} $.  Here $a_{W,d_1, \dots, d'_{k-1}}$ can be chosen to be coprime to all primes $p \leq x$, and $q_{W,d_1, \dots, d'_{k-1}}$ is defined to equal
  $$q_{W,d_1, \dots, d'_{k-1}}  = W [d_1,d'_1] \dots [d_{k-1},d'_{k-1}].$$
  Hence we split $S$ into a main term $S_M$ and an error term $S_E$  as follows:
  \begin{align*}
  &S_{M} =  F_{k}(0) G_{k}(0) \sum_{d_1, \dots, d_k, d'_1, \dots, d'_k }  \frac{\prod_{i=1}^{k-1}  \mu(d_i) \mu(d'_i) F_i(\log_x(d_i))  G_i(\log_x(d'_i))}{ \phi( W[d_1,d'_1] \dots [d_{k-1},d'_{k-1}])} \sum_{\substack{x+h_k \leq m \leq 2x+h_k  }}  \rho(m; x), \\
  &S_{E} = | F_{k}(0) G_{k}(0)| \sum_{\substack{d_1, \dots, d_k, d'_1, \dots, d'_k \\ q_{W,d_1, \dots, d'_{k-1}} \mbox{\begin{scriptsize}
   squarefree
  \end{scriptsize} } }} \prod_{i=1}^{k-1}    |F_i(\log_x(d_i))  G_i(\log_x(d'_i))| \Delta(\rho;  a_{W,d_1, \dots, d'_{k-1}} (q_{W,d_1, \dots, d'_{k-1}})), \\
  &\mbox{where } \Delta(\rho;  a (q)) = \Bigg|  \sum_{\substack{x+h_k \leq m \leq 2x+h_k \\ m \equiv a (q)  }}  \rho(m; x) -  \frac{1}{\phi(q)} \sum_{\substack{x+h_k \leq m \leq 2x+h_k  \\ (m,q)=1   }}  \rho(m; x) \Bigg|.
  \end{align*}
  To deal with the main term $S_M$, we recall assumption $\sum_{n \in [x,2x]} \rho(n;x) = (1-c_1 +o(1)) \frac{x}{\log(x)}$, and we apply Lemma~4.1 of Polymath~\cite{Polymath:2014:VSS} with $[d_i,d'_i]$ replaced by $\phi([d_i,d'_i])$ (see also the discussion on page 19 of~\cite{Polymath:2014:VSS}). Overall this gives the asymptotic
  \begin{align*} 
  S_{M} =  \Bigg(  F_{k}(0) G_{k}(0) \Bigg( \prod_{i=1}^{k-1} \int_0^\infty F'_i(t_i) G'_i(t_i) \mbox{d}t_i \Bigg) (1-c_1)  +o(1) \Bigg)\left( \frac{ xW^{k-1}}{\phi(W)^{k} \log(x)^{k}}\right) .
  \end{align*}
  
  To deal with the error term $S_E$, we first recall 
 that   the support of   product $G_{1}(t_1) \dots G_{k}(t_k)$ is  contained in   $  R_k^+(\delta,\underline{A},\underline{B},\varepsilon,j',\varepsilon_0)$ and that the   support of  product $\prod_{i\neq i_0} F_{i}(t_i) $ is contained in the set $\{(t_1, \dots, t_{i_0-1},t_{i_0+1}, \dots, t_k): \sum_{i \neq i_0} t_i < (1-\varepsilon_0)(A_j-\varepsilon)\}$.  Consider  modulus $q = W [d_1,d'_1] \dots [d_{k-1},d'_{k-1}]$, and assume the modulus contributes to $S_E$. Relabel those $d_i$ which satisfy $d_i \geq x^\delta$ as $f_i$ and combine the remaining $d_i$ and $W$ to a product called $e$. Relabel $d'_i/(d_i,d'_i)$ which satisfy $d'_i/(d_i,d'_i) \geq x^\delta$ as $f'_i$ and combine the remaining $d'_i/(d_i,d'_i) $ to a product called $e'$. Then we have $ q=e e' \prod_{i=1}^m f_i \prod_{i=1}^{m'} f'_i$,    where $e$ and $e'$ are $x^\delta$-smooth and $ \log_x(f_1 \dots f_m e) \leq (1-\frac{\varepsilon_0}{2})(A_j-\varepsilon)$ and $\log_x(f'_1 \dots f'_{m'} e') \leq (1-\frac{\varepsilon_0}{2})(A_{j'}+\varepsilon)$. Using once more that $B_{j,m} \leq B_{j,m+1} \leq B_{j,m} + \delta$, we also have  $ \log_x(f_1 \dots f_m) \leq (1-\varepsilon_0)B_{j,m}$ and  $\log_x(f'_1 \dots f'_{m'}) \leq  (1-\varepsilon_0)B_{j',m'}$. In short, comparing with Definition~\ref{def:neededmoduli}, we see that any modulus  $q = W [d_1,d'_1] \dots [d_{k-1},d'_{k-1}]$ which contributes to $S_E$ is contained in the set $Q(x;\delta,\underline{A},\underline{B},\varepsilon,\frac{\varepsilon_0}{2})$.
 
\medskip We then recall our assumption (2) that $\rho$ has  equidistribution properties for  moduli corresponding to $T_k(\delta,\underline{A},\underline{B},\varepsilon)$: If $a \in \mathbb{Z}$ with $(a,p)=1$ for all $p \leq x$, and if $C>0$, we have 
\begin{align}\label{inequ:equidist}
\sum_{\substack{q \in Q^*(x;\delta,\underline{A},\underline{B},\varepsilon,\frac{\varepsilon_0}{2}) \\  q \mbox{ \scriptsize  is squarefree  }  }} \Bigg| \sum_{\substack{n \in [x,2x] \\ n \equiv a (q)}} \rho(n;x) - \dfrac{1}{\phi(q)} \sum_{\substack{n \in [x,2x] \\(n,q)=1 }}\rho(n;x) \Bigg| \ll_{C,\varepsilon_0} \dfrac{x}{\log(x)^C}.
\end{align}
However, in the statement above, the congruence class $a$ is fixed, whereas in $S_E$ we have that  $a_{W,d_1, \dots, d'_{k-1}}$ depends on $d_i$ and $d'_i$. To get around this issue, we use a trick similar to that on page~22 of~\cite{Polymath:2014:VSS}, and first define the following set: 
\begin{align*}
\mathcal{A} = \left\{ a \in \left\{1, \dots, \prod_{p \leq x} p  \right\}: a \equiv b +h_k \mbox{ mod } W,  \, \forall p \in \left(\log\log\log x, x \right] \exists i \in \{1, \dots, k-1\}:  a \equiv h_k -h_i \mbox{ mod } p\right\}.
\end{align*}
First we think about a fixed choice of $d_1, \dots, d'_{k-1}$, The set $\mathcal{A}$ is constructed so that for any $q=q_{W,d_1, \dots, d'_{k-1}}  = W [d_1,d'_1] \dots [d_{k-1},d'_{k-1}]$, the corresponding $a_{W,d_1, \dots, d'_{k-1}} (q)$ is contained in $\mathcal{A}$. More specifically, by applying the Chinese Remainder Theorem, we see that when $\widetilde{a}$ ranges over $\mathcal{A}$, then $\widetilde{a} \equiv a_{W,d_1, \dots, d'_{k-1}} (q)$ for $|\mathcal{A}|/(k-1)^{\omega(q/W)}$ choices of $\widetilde{a}$, where $\omega(q/W)$ is the number of distinct prime factors of $q/W$. But since $q/W$ is squarefree, we have $(k-1)^{\omega(q/W)} \ll \tau(q)^{O(1)}$. Hence if we sum over all $\widetilde{a} \in \mathcal{A}$ and multiply with $\frac{\tau(q)^{O(1)}}{|\mathcal{A}|}$, then $a_{W,d_1, \dots, d'_{k-1}} $ is counted at least once. Furthermore, any $q \in Q^*(x;\delta,\underline{A},\underline{B},\varepsilon,\frac{\varepsilon_0}{2})$ has $q=q_{W,d_1, \dots, d'_{k-1}} $ for at most $ \tau(q)^{O(1)} $ choices of $d_1, \dots, d'_{k-1}$. Therefore we obtain the following bound on $S_E$:
  \begin{align*}
S_{E} \ll \sum_{\substack{q \in Q^*(x;\delta,\underline{A},\underline{B},\varepsilon,\frac{\varepsilon_0}{2})  \\ q \mbox{\begin{scriptsize}
   squarefree
  \end{scriptsize} }   }}  \frac{\tau(q)^{O(1)}}{|\mathcal{A}|} \sum_{a \in \mathcal{A}} \Bigg| \sum_{\substack{n \in [x+h_k,2x+h_k] \\ n \equiv a (q)}} \rho(n;x) - \dfrac{1}{\phi(q)} \sum_{\substack{n \in [x+h_k,2x+h_k] \\(n,q)=1 }}\rho(n;x) \Bigg|.
  \end{align*}
  To finish the proof we simply apply the Cauchy-Schwarz inequality and equidistribution property (\ref{inequ:equidist}). Choosing constant $C$ sufficient large, we have:  
  \begin{align*}
S_{E} &\ll \frac{1}{|\mathcal{A|}} \sum_{a \in \mathcal{A}}  \Bigg(\sum_{\substack{q \in Q^*(x;\delta,\underline{A},\underline{B},\varepsilon,\frac{\varepsilon_0}{2})  \\ q \mbox{\begin{scriptsize}
   squarefree
  \end{scriptsize} }   }}   \tau(q)^{O(1)} \Delta(\rho,a(q)) \Bigg)^{1/2}\Bigg(\sum_{\substack{q \in Q^*(x;\delta,\underline{A},\underline{B},\varepsilon,\frac{\varepsilon_0}{2})  \\ q \mbox{\begin{scriptsize}
   squarefree
  \end{scriptsize} }   }}    \Delta(\rho,a(q)) \Bigg)^{1/2} \\
  &\ll_C \frac{1}{|\mathcal{A|}} \sum_{a \in \mathcal{A}}  (x \log(x)^{O(1)})^{1/2} (x \log(x)^{-C})^{1/2} \ll_C \frac{x}{\log(x)^{C/2}}.
  \end{align*}
  Since $W \ll (\log\log x)^{O(1)}$, summing up the asymptotic for $S_M$ and the bound on $S_E$ for a large $C$ completes the proof.
 \end{proof}

\subsection{Proof of Proposition~\ref{prop:GPYsieve}}

Proposition~\ref{prop:GPYsieve} can now be obtained as a   corollary of Lemma~\ref{lem:sieveweights}, Lemma~\ref{lem:asymptotics}, as well as Theorem~3.6 of~\cite{Polymath:2014:VSS}. We need to be slightly careful when mixing Harman's sieve ($\rho(n;x)$) and the epsilon enlargement trick.

\begin{proof}[Proof of Proposition~\ref{prop:GPYsieve}]
Let $\varepsilon_0>0$ be very small. By Lemma~\ref{lem:sieveweights}, we can then  find integers $L_1, \dots, L_n$, constants $c_{j,l}$    and   smooth compactly supported functions $f_{j,l,i}: [0,\infty) \rightarrow \mathbb{R}$, with  $f_{j,l,1}(t_1) \dots f_{j,l,k}(t_k)$ supported on $  R_k^+(\delta,\underline{A},\underline{B},\varepsilon,j,\varepsilon_0)$, so that the following is true: 
 Denoting by $\mathcal{L}(j,i)$  the set of $ l \in  \{1, \dots, L_j\}$ for which the  support of  product $\prod_{s\neq i} f_{j,l,s}(t_s) $ is contained in $\{(t_1, \dots, t_{i-1},t_{i+1}, \dots, t_k): \sum_{s \neq i} t_i < (1-\varepsilon_0)(A_j-\varepsilon)\}$, and denoting by    $\mathcal{U}(j,i)$  its   complement $\{1, \dots, L_j\} \setminus \mathcal{L}(j,i)$, 
the terms 
\begin{align*}
\mathcal{I}&=\sum_{j=1}^n\sum_{l=1}^{L_j} \sum_{j'=1}^n \sum_{l'=1}^{L_{j'}} c_{j,l} c_{j',l'}  \prod_{s=1}^k
\int_0^\infty    f'_{j,l,s}(t_s)f'_{j',l',s}(t_s) \, \mbox{d}t_s, \\
\mathcal{J}_i&=\sum_{j=1}^n \sum_{l \in  \mathcal{L}(j,i)}  \sum_{j'=1}^n  \sum_{l' \in  \mathcal{L}(j',i)}  c_{j,l}c_{j',l'} f_{j,l,i}(0)  f_{j',l',i}(0)  \prod_{s\neq i} \mathop{\int_0^\infty   }  f'_{j,l,s}(t_s)     f'_{j',l',s}(t_s)   \, \mbox{d}t_s \\
&+2\sum_{j=1}^n \sum_{l \in  \mathcal{U}(j,i)}  \sum_{j'=1}^n  \sum_{l' \in  \mathcal{L}(j',i)} c_{j,l}c_{j',l'} f_{j,l,i}(0)  f_{j',l',i}(0)  \prod_{s\neq i} \mathop{\int_0^\infty   }  f'_{j,l,s}(t_s)     f'_{j',l',s}(t_s)   \, \mbox{d}t_s\\
\mathcal{K}_i&=\sum_{j=1}^n \sum_{l \in  \mathcal{U}(j,i)}  \sum_{j'=1}^n  \sum_{l' \in  \mathcal{U}(j',i)}  c_{j,l}c_{j',l'}   \prod_{s=1}^k \mathop{\int_0^\infty} f'_{j,l,s}(t_s)     f'_{j',l',s}(t_s)    \, \mbox{d}t_s
\end{align*}
satisfy the inequality 
\begin{align*}
\frac{(1-c_1) \sum_{i=1}^k  \mathcal{J}_i - c_2 \sum_{i=1}^k  \mathcal{K}_i}{ \mathcal{I}} > 1.
\end{align*}
We now choose the following sieve weights:
 $$\nu(n)=\left(\sum_{j=1}^n\sum_{l=1}^{L_j} c_{j,l} \lambda_{f_{j,l,1}}(n+h_1) \dots \lambda_{f_{j,l,k}}(n+h_k)  \right)^2.$$  
Next we consider the ratio
\begin{align*}
&\widetilde{N}(x;\rho,\nu)=  \frac{\sum_{i=1}^k\sum_{\substack{x \leq n \leq 2x,  n \equiv b(W) }} \nu(n)   \rho(n+h_i;x) }{  \sum_{\substack{x \leq n \leq 2x,  n \equiv b(W) }} \nu(n)   }.
\end{align*} 
On the one hand, Theorem~3.6 of~\cite{Polymath:2014:VSS} tells us that
\begin{align*}
&\sum_{\substack{x \leq n \leq 2x \\ n \equiv b (W)}}  \prod_{s =1}^k \lambda_{f_{j,l,s}}(n+h_s) \lambda_{f_{j',l',s}}(n+h_s)  = \Bigg( \prod_{s =1}^k \int_0^\infty f'_{j,l,s}(t_s)f'_{j',l',s}(t_s) \mbox{d}t_s+o(1)\Bigg)  \left( \frac{ xW^{k-1}}{\phi(W)^{k} \log(x)^{k}}\right).
\end{align*}
Expanding the square in $\nu(n)$, this gives
\begin{align} \label{asym:prop11}
\sum_{\substack{x \leq n \leq 2x,  n \equiv b(W) }} \nu(n) = \left(\mathcal{I} + o(1) \right) \left( \frac{ xW^{k-1}}{\phi(W)^{k} \log(x)^{k}}\right).
\end{align}
On the other hand, a lower bound for the enumerator is obtained as follows: $\rho(n+h_i;x) + c_2 \geq 0$ and   
\begin{align*}
\nu(n) \rho(n+h_i;x) &\geq  \left(\sum_{j=1}^n\sum_{l \in \mathcal{L}(j,i)} c_{j,l} \prod_{s=1}^k \lambda_{f_{j,l,s}}(n+h_s) \right)^2 (\rho(n+h_i;x)+ c_2)   \\
&+2 \left(\sum_{j=1}^n\sum_{l \in \mathcal{L}(j,i)} c_{j,l} \prod_{s=1}^k \lambda_{f_{j,l,s}}(n+h_s)    \right) \left(\sum_{j=1}^n\sum_{l \in \mathcal{U}(j,i)} c_{j,l} \prod_{s=1}^k \lambda_{f_{j,l,s}}(n+h_s)    \right)  (\rho(n+h_i;x)+ c_2) \\
&-c_2\left(\sum_{j=1}^n\sum_{l \in \mathcal{L}(j,i)} c_{j,l} \prod_{s=1}^k \lambda_{f_{j,l,s}}(n+h_s)  +  
\sum_{j=1}^n\sum_{l \in \mathcal{U}(j,i)} c_{j,l} \prod_{s=1}^k \lambda_{f_{j,l,s}}(n+h_s)   \right)^2  \\
&=  \left(\sum_{j=1}^n\sum_{l \in \mathcal{L}(j,i)} c_{j,l} \prod_{s=1}^k \lambda_{f_{j,l,s}}(n+h_s) \right)^2 \rho(n+h_i;x)   \\
&+2 \left(\sum_{j=1}^n\sum_{l \in \mathcal{L}(j,i)} c_{j,l} \prod_{s=1}^k \lambda_{f_{j,l,s}}(n+h_s)    \right) \left(\sum_{j=1}^n\sum_{l \in \mathcal{U}(j,i)} c_{j,l} \prod_{s=1}^k \lambda_{f_{j,l,s}}(n+h_s)    \right) \rho(n+h_i;x) \\
&-c_2\left( 
\sum_{j=1}^n\sum_{l \in \mathcal{U}(j,i)} c_{j,l} \prod_{s=1}^k \lambda_{f_{j,l,s}}(n+h_s)   \right)^2 .
\end{align*}
But Lemma~\ref{lem:asymptotics} gives us asymptotic 
\begin{align*}
&\sum_{\substack{x \leq n \leq 2x \\ n \equiv b (W)}}  \rho(n+h_{i}; x) \prod_{s =1}^k \lambda_{f_{j,l,s}}(n+h_s) \lambda_{f_{j',l',s}}(n+h_s) \\ &= \Bigg((1-c_1)  f_{j,l,k}(0) f_{j',l',k}(0) \prod_{s \neq i} \int_0^\infty f'_{j,l,s}(t_s)f'_{j',l',s}(t_s) \mbox{d}t_s+o(1)\Bigg)  \left( \frac{ xW^{k-1}}{\phi(W)^{k} \log(x)^{k}}\right).
\end{align*}
Summing the lower bound for $\nu(n) \rho(n+h_i;x)$ over $[x,2x]$ and substituting the asymptotic from Lemma~\ref{lem:asymptotics} for terms involving $\rho(n+h_i;x)$, and substituting the asymptotic from Theorem~3.6 of~\cite{Polymath:2014:VSS} for terms involving $c_2$, we then get 
\begin{align} \label{asym:prop12}
\sum_{i=1}^k\sum_{\substack{x \leq n \leq 2x,  n \equiv b(W) }} \nu(n)   \rho(n+h_i;x)  =  \left((1-c_1) \sum_{i=1}^k  \mathcal{J}_i - c_2 \sum_{i=1}^k  \mathcal{K}_i + o(1) \right) \left( \frac{ xW^{k-1}}{\phi(W)^{k} \log(x)^{k}}\right).
\end{align}
Combining asymptotics (\ref{asym:prop11}) and (\ref{asym:prop12}), we finally get 
\begin{align*}
&\widetilde{N}(x;\rho,\nu)=  \frac{(1-c_1) \sum_{i=1}^k  \mathcal{J}_i - c_2 \sum_{i=1}^k  \mathcal{K}_i + o(1) }{ \mathcal{I} + o(1)  } > 1+o(1).
\end{align*} 
Hence, for sufficiently large $x$, we have the desired inequality $N(x;\rho,\nu)>0$.
\end{proof}

\section{Equidistribution estimates} \label{sec:equiest}

In this section we derive equidistribution estimates for various convolutions of coefficient sequences. These will be needed to show that some minorant for the prime indicator function (denoted by $\rho(n;x)$)  has equidistribution properties for  moduli corresponding to $T_k(\delta,\underline{A},\underline{B},\varepsilon)$, provided that the entries of $\underline{B}$ are chosen small enough:  If the entries of $\underline{B}$ are small, then the corresponding moduli, given by set $Q^*(x;\delta,\underline{A},\underline{B},\varepsilon,\varepsilon_0)$, are either smaller than $x^{1/2}$ or have a large $x^\delta$-smooth factor.  In the first case we can use the Bombieri-Vinogradov theorem, so we need to focus on moduli with a large $x^\delta$-smooth factor.

\medskip  Here the initial idea is to use the Type I/II/III estimates for $i$-tuply $x^\delta$-densely divisible integers from Polymath's paper~\cite{Polymath:2014:EDZ},  the equidistribution estimates for $x^\delta$-smooth moduli from the author's paper~\cite{Stadlmann:2023:PAP}, and well as the Type I estimate of Baker and Irving~\cite{Baker:2017:BIP}. However, there is a problem:   a big $x^\delta$-smooth factor is  not sufficient to ensure that a number is $i$-tuply $x^\delta$-densely divisible. Thus we need to prove amended versions of the equidistribution estimates listed above. Fortunately we only need to make small changes in the relevant proofs.

\subsection{Important definitions}

Since this section discusses coefficient sequences, we first define these.

\begin{defn}[Coefficient sequences]\label{def:sequence}  A  coefficient sequence  is  a  function $\alpha: \mathbb{N} \times (1, \infty) \rightarrow \mathbb{R}$ which satisfies $|\alpha(n;x)| \ll \tau(n)^{O(1)}\log(x)^{O(1)}$. 
\begin{enumerate}[{\rm (i)}]
\item $\alpha$ is said to be located at scale $N(x)$ if there are some constants $1 \ll c \ll C \ll 1$ such that   $\alpha(\,\cdot\,;x)$ is supported on $[cN(x),CN(x)]$ for every $x>1$.
\item If $\alpha$ is located at scale $N(x)$, it is said to have the Siegel-Walfisz property if 
\begin{align*}
\Bigg| \sum_{\substack{n=a (q) \\ (n,r)=1 }} \alpha(n;x)  -\dfrac{1}{\phi(q)}\sum_{\substack{ (n,qr)=1 }} \alpha(n;x) \Bigg| \ll_A \tau(qr)^{O(1)}\dfrac{ N(x) }{\log(x)^{A}}
\end{align*}
for any $x>1$, any $q, r \geq 1$, any $A >1$, and any  residue class $a$ mod $q$ with $(a,q)=1$. 
\item $\alpha$ is said to be  smooth at scale $N(x)$ if there are some constants $1\ll c \ll C \ll 1$ so that for every $x>1$, there exists  a  smooth function $\psi:\mathbb{R} \rightarrow \mathbb{C}$ supported on $[c,C]$, with $|\psi^{(j)}(t)| \ll_j \log(x)^{O_j(1)}$ for all $t \in \mathbb{R}$, such that 
$
\alpha(n;x) = \psi\left(\dfrac{n}{N}\right).
$
\end{enumerate}
\end{defn}

Further, as usually,    convolutions of sequences are defined as follows: If $\alpha: \mathbb{N} \rightarrow \mathbb{C}$ and $\beta: \mathbb{N} \rightarrow \mathbb{C}$, then
\begin{align*}
(\alpha \star \beta)(n) = \sum_{d \mid n} \alpha(d) \beta\left(\dfrac{n}{d}\right).
\end{align*}

We will now state various equidistribution estimates for Dirichlet convolutions $\alpha\star \beta$. Generalizing our earlier definition, we will use the following notion of equidistribution:

\begin{defn}[Equidistribution]\label{def:expdist}
Consider  $f: \mathbb{N} \times (1,\infty) \rightarrow \mathbb{C}$. Let $D(x;\omega,\gamma,\delta; \varepsilon)$ be a set of integers dependent on  parameters $\omega$, $\gamma$  and $\delta$, and dependent on a small parameter $\varepsilon>0$. 

For a given set of $(\omega,\gamma,\delta)$, we say $f$ has equidistribution properties for moduli in $D(x;\omega,\gamma,\delta;\varepsilon)$, provided  the following holds:  For any sufficiently small $\varepsilon >0$ and for all $x>1$  and  $A>0$,  and for all $a \in \mathbb{Z}$ with $(a,p)=1$ for $p \leq x$,
\begin{align}\label{inequ:expdist}
\sum_{\substack{ d \in D(x;\omega,\gamma,\delta;\varepsilon) \\  d \mbox{ \scriptsize  is squarefree  }  }} \Bigg| \sum_{\substack{n \in [x,2x] \\ n \equiv a (d)}} f(n;x) - \dfrac{1}{\phi(d)} \sum_{\substack{n \in [x,2x] \\(n,d)=1 }}f(n;x) \Bigg| \ll_{A,\varepsilon} \dfrac{x}{\log(x)^A}.
\end{align}
Note that the implied constant is only allowed to depend on $\omega$, $\gamma$ and $\delta$ insofar that they determine which $\varepsilon$ are sufficiently small for (\ref{inequ:expdist}) to hold.
\end{defn}

\subsection{Equidistribution ingredients} To choose a suitable $\rho(n;x)$ and $T_k(\delta,\underline{A},\underline{B},\varepsilon)$, we will use the following equidistribution estimates. Proofs follow in the next subsection.

\begin{lemma}[Polymath Type II] \label{lem:typeIIPoly} 
Suppose $f(n;x) = (\alpha \star \beta)(n;x)$, where $\alpha$ is  coefficient sequence at scale $M$ and $\beta$ is a  coefficient sequence at scale $N$ with $M(x) N(x) \asymp x$ and $N(x) \leq x^{1/2}$. Suppose also that $\beta$ has the Siegel-Walfisz property. Write $N(x) = x^\gamma$. Define the set of moduli
\begin{align*}
D_{IIa}(x;\omega,\gamma,\delta;\varepsilon) = \left\{d \in  [1, x^{1/2+2\omega}]\cap \mathbb{N}:  \mbox{There exists a divisor } r \mid d \mbox{ with } x^{\gamma -3\varepsilon-\delta} < r < x^{\gamma - 3\varepsilon} \right\}.
\end{align*}
Consider triples $(\omega,\gamma,\delta)$ which satisfy the following inequalities:
\begin{align*}
&24\omega + 7\delta -5\gamma <-2, \\
&8\omega + 3\delta -\gamma <0.
\end{align*}
For such triples  $(\omega,\gamma,\delta)$, convolution $f$ has equidistribution properties for moduli in $D_{IIa}(x;\omega,\gamma,\delta;\varepsilon)$.
\end{lemma}

This lemma is a variant  of Theorem 2.8(iv) of Polymath~\cite{Polymath:2014:EDZ}, and   serves as  replacement for Theorem 2.8(i).

\medskip Note: In Lemma~\ref{lem:typeIIPoly}, and in all subsequent lemmas, permissible $(\omega,\gamma,\delta)$ are defined via strict inequalities. If we restrict the choice of $(\omega,\gamma,\delta)$ to triples which satisfy these inequalities with room to spare (e.g. for some fixed $\varepsilon_1$, we consider triples $(\omega,\gamma,\delta)$ with $24\omega + 7\delta -5\gamma + \varepsilon_1 <-2$ and $8\omega + 3\delta -\gamma + \varepsilon_1 <0$), then $\varepsilon$ can be chosen uniformly in $(\omega,\gamma,\delta)$. This means that we can ensure that the implied constant in the equidistribution estimate does not depend on the triple.

\begin{lemma}[Polymath Type I(ii)] \label{lem:typeI(ii)Poly} 
Suppose $f(n;x) = (\alpha \star \beta)(n;x)$, where $\alpha$ is  coefficient sequence at scale $M$ and $\beta$ is a  coefficient sequence at scale $N$ with $M(x) N(x) \asymp x$ and $N(x) \leq x^{1/2}$. Suppose also that $\beta$ has the Siegel-Walfisz property. Write $N(x) = x^\gamma$. Define the set of moduli
\begin{align*}
D_{IIb}(x;\omega,\gamma,\delta;\varepsilon) = \left\{d \in  [1, x^{1/2+2\omega}]\cap \mathbb{N}:  \begin{array}{l} 
\exists r \mid d \,\, \exists u \mid \frac{d}{r}  \mbox{ with } x^{\gamma -3\varepsilon-\delta} < r < x^{\gamma - 3\varepsilon} \\ \mbox{and } x^{1/2 -\gamma-2\omega -6 \varepsilon-\delta } < u < x^{1/2-\gamma-2\omega  -6\varepsilon}
\end{array}\right\}.
\end{align*}
Consider triples $(\omega,\gamma,\delta)$ which satisfy the following inequalities: 
 \begin{align*}
  &          24\omega +7\delta  -3\gamma  <-1,  \\     
 &        8\omega+3\delta -  \gamma    <0.
 \end{align*} 
For such triples  $(\omega,\gamma,\delta)$, convolution $f$ has equidistribution properties for moduli in $D_{IIb}(x;\omega,\gamma,\delta;\varepsilon)$.
\end{lemma}

This lemma is a variant  of Theorem 2.8(ii) of Polymath~\cite{Polymath:2014:EDZ}.

\begin{lemma}[Baker-Irving Type I] \label{lem:typeIBI} 
Suppose $f(n;x) = (\alpha \star \beta)(n;x)$, where $\alpha$ is  coefficient sequence at scale $M$ and $\beta$ is a smooth coefficient sequence at scale $N$ with $M(x) N(x) \asymp x$. Write $N(x) = x^\gamma$. Define the set of moduli
\begin{align*}
D_I(x;\omega,\gamma,\delta;\varepsilon) =
\begin{cases}
 \left\{d \in  [1, x^{1/2+2\omega}]\cap \mathbb{N}:  \exists r \mid d \mbox{ with } x^{\gamma-\delta -3\varepsilon} < r < x^{\gamma - 3\varepsilon} \right\} \qquad \qquad \mbox{ if } \gamma \leq \frac{1}{2}, \\
  \left\{d \in  [1, x^{1/2+2\omega}]\cap \mathbb{N}:  \exists r \mid d \mbox{ with } x^{1-\gamma-\delta -3\varepsilon} < r < x^{1-\gamma - 3\varepsilon} \right\} \qquad \mbox{ if } \gamma \in (\frac{1}{2},\frac{1}{2}+2\omega+\varepsilon], \\
  [1, x^{1/2+2\omega}]\cap \mathbb{N}  \qquad\qquad\qquad\qquad\qquad\qquad\qquad\qquad\qquad\quad \qquad \,\,\, \mbox{ if } \gamma >\frac{1}{2}+2\omega+\varepsilon.
 \end{cases}
\end{align*}
Consider triples $(\omega,\gamma,\delta)$ which satisfy the following inequalities: 
\begin{align*}
&3\gamma-12\omega  -3\delta >1    \qquad  \,\, \mbox{ if }  \quad \gamma \leq \tfrac{1}{2}, \\
&68\omega + 14 \delta    <1  \qquad \qquad\mbox{ if } \quad   \gamma \in (\tfrac{1}{2},\tfrac{1}{2}+2\omega+\varepsilon].
\end{align*}
For such triples  $(\omega,\gamma,\delta)$, convolution $f$ has equidistribution properties for moduli in $D_I(x;\omega,\gamma,\delta;\varepsilon)$.
\end{lemma}

This lemma is a variant  of Lemma 5 of Baker and Irving~\cite{Baker:2017:BIP}.

\begin{lemma}[S. Type I] \label{lem:typeIS} 
Suppose $f(n;x) = (\alpha \star \beta)(n;x)$, where $\alpha$ is  coefficient sequence at scale $M$ and $\beta$ is a  coefficient sequence at scale $N$ with $M(x) N(x) \asymp x$ and $N(x) \leq x^{1/2}$. Suppose also that $\beta$ has the Siegel-Walfisz property. Write $N(x) = x^\gamma$. Define the set of moduli
\begin{align*}
D_{IIc}(x;\omega,\gamma,\delta;\varepsilon) = \left\{d \in  [1, x^{1/2+2\omega}]\cap \mathbb{N}:  \begin{array}{l} 
\exists r \mid d \,\, \exists u \mid \frac{d}{r} \,\, \exists d_1 \mid r \mbox{ with } x^{\gamma -3\varepsilon-\delta} < r < x^{\gamma - 3\varepsilon} \\ \mbox{and }    x^{1-\gamma-6\varepsilon-\delta}   d^{-1} < u < x^{1-\gamma-6\varepsilon}   d^{-1} \\ 
\mbox{and }  r^2 x^{2-\gamma-52\varepsilon -\delta}  d^{-4} < d_1 <  r^2 x^{2-\gamma-52\varepsilon}  d^{-4}
\end{array}\right\}.
\end{align*}
Consider triples $(\omega,\gamma,\delta)$ which satisfy the following inequalities: 
\begin{align*}
&8\omega+4\delta + 2\gamma < 1, \\
&32 \omega + 10 \delta -\gamma<0, \\
&48\omega +  16 \delta  -4\gamma  < -1.
\end{align*}
For such triples  $(\omega,\gamma,\delta)$, convolution $f$ has equidistribution properties for moduli in $D_{IIc}(x;\omega,\gamma,\delta;\varepsilon)$.
\end{lemma}

This lemma is a variant  of Theorem 1 of~\cite{Stadlmann:2023:PAP}.

\begin{lemma}[Polymath Type III] \label{lem:typeIII} 
Suppose $f(n;x) = (\alpha \star \psi_1 \star \psi_2 \star \psi_3)(n;x)$, where $\alpha$ is  coefficient sequence at scale $M$ and $\psi_1$, $\psi_2$ and $\psi_3$ are smooth coefficient sequences at scales $N_1$, $N_2$ and $N_3$ with 
\begin{align*}
&M(x) N_1(x) N_2(x) N_3(x) \asymp x, \\
&N_1(x)N_2(x), \, N_1(x)N_3(x), \, N_2(x)N_3(x) \gg x^{1-\gamma}, \\
&x^{1-2\gamma} \ll N_1(x), N_2(x), N_3(x) \ll x^{\gamma}.
\end{align*}
Define the set of moduli
\begin{align*}
D_{III}(x;\omega,\gamma,\delta;\varepsilon) = \left\{d \in  [1, x^{1/2+2\omega}]\cap \mathbb{N}:  \exists r \mid d \mbox{ with } x^{1/3+4\delta/3 -4\omega/3-\delta} < r < x^{1/3+4\delta/3 -4\omega/3} \right\}.
\end{align*}
Consider triples $(\omega,\gamma,\delta)$ which satisfy the following inequality: 
\begin{align*}
28\omega   +9\gamma   +8\delta   <4.
\end{align*}
For such triples  $(\omega,\gamma,\delta)$, convolution $f$ has equidistribution properties for moduli in $D_{III}(x;\omega,\gamma,\delta;\varepsilon)$.
\end{lemma}

This lemma is a variant  of Theorem 2.8(v) of Polymath~\cite{Polymath:2014:EDZ}.

\subsection{Proofs} The original proofs of the lemmas and theorems referenced above span a total of over 50 pages. Since these proofs only need to be amended in a few specific places, it is not practical to write down full proofs of the new versions listed above. Instead  we will just explain the particular changes we make. Our explanations are best read side-by-side with the original proofs. 

\subsubsection{Starting point} With the exception of Theorem 2.8(v) of Polymath~\cite{Polymath:2014:EDZ}, all results referenced above start their proofs in Section~5 of~\cite{Polymath:2014:EDZ}, and essentially proceed in exactly the same way from page 42 all the way until page 50. Although Polymath discusses $i$-tuply $x^\delta$-densely divisible integers, pages 42 to 50 really just use that the relevant moduli $m$  have a factorization $m=qr$ with $Nx^{-3\varepsilon-\delta_0} \ll r \ll Nx^{ -3\varepsilon}$. As such, we begin by stating Polymath's Theorem~5.8~\cite{Polymath:2014:EDZ} with a suitably relaxed condition on the moduli. Since this statement involves exponential sums, we first need to introduce some more notation:

\medskip 
 For $q \in \mathbb{N}$ and $n \in \mathbb{Z}$, we define $e_q(n) = \exp(2\pi i n/q)$. For $a, b \in \mathbb{Z}$, we set 
\begin{align*}
e_q\left(\dfrac{a}{b}\right) = \begin{cases} \,\,
e_q\big( \overline{a/b} \big) \quad \mbox{ if } a/b  \mbox{ is well-defined mod } q, \\ \,\, 0 \qquad\qquad\! \mbox{ otherwise,}
\end{cases}
\end{align*}
where $\overline{a/b}$ represents an integer $n$ which satisfies $n \equiv a/b$ mod $q$.

\medskip
We will need to work with exponential sums of the following form:
\begin{defn}\label{def:sigTheorem58}
For a given triple  $(\omega,\gamma,\delta)$ with  $
\gamma-4\omega  -\delta  >0$ and $\gamma \leq \frac{1}{2}$ and for $x>0$, let $Z(\omega,\gamma,\delta;x)$ denote the set of $(Q,R,q_0,b_1,b_2,\ell)$ with the following properties:
\begin{enumerate}[{\rm(i)}]
\item  $R \in (0,\infty)$ with $x^{\gamma-\delta -3\varepsilon} \leq R \leq x^{\gamma - 2\varepsilon} $.
\item $Q \in (0,\infty)$ with $x^{1/2-\varepsilon} \ll QR \ll x^{1/2+2\omega}$.
\item $q_0 \in \mathbb{N}$ with  $q_0 \ll Q$.
\item $b_1, b_2 \in \mathbb{Z}$ with $(b_1b_2,p)=1$ for all $p \leq x$.
\item $\ell \in \mathbb{Z}$ with  $1 \leq |\ell| \ll \frac{N}{R}$.
\end{enumerate}
  
Consider some choice of sets $\mathcal{Q}(Q;x,\omega,\gamma,\delta,\varepsilon) \subseteq [Q,2Q] \cap \mathbb{N}$ and $\mathcal{R}(R;x,\omega,\gamma,\delta,\varepsilon) \subseteq [R,2R] \cap \mathbb{N}$.

\medskip For a given $z=(Q,R,q_0,b_1,b_2,\ell) \in Z(\omega,\gamma,\delta;x)$, we then set  $C(n)  = 1_{\substack{\frac{b_1}{n}  \equiv \frac{b_2}{n+\ell r} \mbox{{\scriptsize  mod }} q_0}}
$ and   define 
\begin{align*}
&\Sigma_{\mathcal{Q}, \mathcal{R}}(z)= \!\!\!\!\! \sum_{\substack{ r \in \mathcal{R}(R;x,\omega,\gamma,\delta,\varepsilon) }}  \mathop{\sum\sum}\limits_{\substack{q_1, q_2 \\ q_0 q_1,q_0 q_2 \in \mathcal{Q}(Q;x,\omega,\gamma,\delta,\varepsilon) \\ (q_1,q_2)=1 \\  q_0q_1r, q_0q_2r \mbox{\begin{scriptsize}
squarefree
\end{scriptsize} }
}} \sum_{0<|h| \leq \frac{x^\varepsilon RQ^2}{q_0 M}}    \Bigg| \!\!\sum_{\substack{n \\ (n,rq_0q_1 )=1 \\ (n+\ell r, q_0q_2)=1 }} \!\!\!\!C(n)\beta(n)\overline{\beta(n+\ell r)}  \Phi_\ell(h,n,r,q_0,q_1,q_2) \Bigg|,\\
&\mbox{where } \Phi_\ell(h,n,r,q_0,q_1,q_2) =  e_r\left(\dfrac{ah}{nq_0q_1q_2}\right)e_{q_0q_1}\left(\dfrac{b_1h}{nrq_2}\right)e_{q_2}\left(\dfrac{b_2h}{(n+\ell r)rq_0q_1}\right).
\end{align*}

\end{defn}

\medskip

Now we are ready to summarize Theorem~5.8.

\begin{lemma}[Alternative version of Theorem~5.8 of~\cite{Polymath:2014:EDZ}] \label{lem:theorem58}
Suppose $f(n;x) = (\alpha \star \beta)(n;x)$, where $\alpha$ is  coefficient sequence at scale $M$ and $\beta$ is a  coefficient sequence at scale $N$ with $M(x) N(x) \asymp x$. Suppose  $\beta$ has the Siegel-Walfisz property. Write $N(x) = x^\gamma$.

\smallskip
Consider some sets $\mathcal{Q}(Q;x,\omega,\gamma,\delta,\varepsilon) \subseteq [Q,2Q] \cap \mathbb{N}$ and $\mathcal{R}(R;x,\omega,\gamma,\delta,\varepsilon) \subseteq [R,2R] \cap \mathbb{N}$.

Let $Z(\omega, \gamma, \delta; x)$ and $\Sigma_{\mathcal{Q},\mathcal{R}}(Q,R,q_0,b_1,b_2,\ell)$ be as defined in Definition~\ref{def:sigTheorem58}.

Let $D_0 = \exp(\log(x)^{1/3})$ and suppose that $q \in \mathcal{Q}(Q;x,\omega,\gamma,\delta,\varepsilon)$ implies that $(q,p)=1$ for $p \leq D_0$.

\medskip
 Suppose  $D(x;\omega,\gamma,\delta;\varepsilon)$ is a set of moduli with 
\begin{align*}
D(x;\omega,\gamma,\delta;\varepsilon)\subseteq \left\{d \in  [1, x^{1/2+2\omega}]\cap \mathbb{N}: \begin{array}{l} \exists R \exists Q \mbox{ with }  \log_2(Q) \in \mathbb{Z} \mbox{ and }  \log_2(R) \in \mathbb{Z} \\  \mbox{and } x^{\gamma-\delta -3\varepsilon} \leq R \leq x^{\gamma - 2\varepsilon}  \mbox{ so that } \\ d=qr \mbox{ for some }  q \in \mathcal{Q}(Q;x,\omega,\gamma,\delta,\varepsilon) \mbox{ and } r\in \mathcal{R}(R;x,\omega,\gamma,\delta,\varepsilon), \\
 \mbox{or } \prod_{\substack{ p \mid d,  p \leq D_0}} p > \exp(\log^{2/3} x)    
\end{array} \right\}.
\end{align*}
Consider triples   $(\omega,\gamma,\delta)$ with  $
\gamma-4\omega  -\delta  >0$ and $\gamma \leq \frac{1}{2}$. Suppose a collection of such triples   satisfies   
\begin{align*}
\sup_{(Q,R,q_0,b_1,b_2,\ell) \in Z(\omega,\gamma,\delta;x)}\Sigma_{\mathcal{Q},\mathcal{R}}(Q,R,q_0,b_1,b_2,\ell) \ll_\varepsilon x^{-2\varepsilon} Q^2RN (q_0, \ell) q_0^{-2}.
\end{align*}
Then convolution $f$ has equidistribution properties for moduli in $D(x;\omega,\gamma,\delta;\varepsilon)$.
\end{lemma}

\begin{proof}
Only   minor changes need to be made on pages 42 to 50 of~\cite{Polymath:2014:EDZ}: Throughout, Polymath considers congruence classes $a$ with $(a,p)=1$ for $p \in I$, where $I$ is some bounded set of integers (dependent on $x$). We simply choose $I=[1,x]$, and replace $P_I$ by $\prod_{p \leq x} p$.

\medskip

On page 43 it is observed that the contribution of $d$ with $ \prod_{\substack{ p \mid d,  p \leq D_0}} p > \exp(\log^{2/3} x)$ can be bounded trivially. Hence we only think about $d=qr$ with  $q \in \mathcal{Q}(Q;x,\omega,\gamma,\delta,\varepsilon)$ and $ r\in \mathcal{R}(R;x,\omega,\gamma,\delta,\varepsilon) $. 

\medskip
On page 44, Polymath notes that any $i$-tuply $x^\delta$-densely divisible integer has a factorization $d=qr$ with $ x^{\gamma-\delta -3\varepsilon} \leq r \leq  x^{\gamma - 3\varepsilon}  $. Next prime factors $p \leq D_0$ are removed from $q$, which lowers the size of $q$ by no more than $x^\varepsilon$. From then on, the proof works with sums over $q \in D_J^{(j)}(x^{\delta+o(1)}) \cap [Q,2Q]$ and $r \in D_I^{(k)}(x^{\delta+o(1)}) \cap [R,2R]$, which are $x^{\delta+o(1)}$-densely divisible integers. We replace these by sums over $q \in \mathcal{Q}(Q;x,\omega,\gamma,\delta,\varepsilon)$ and $ r\in \mathcal{R}(R;x,\omega,\gamma,\delta,\varepsilon) $, where   $ x^{\gamma-\delta -3\varepsilon} \leq R \leq  x^{\gamma - 2\varepsilon}  $. We already assumed that elements of $\mathcal{Q}(Q;x,\omega,\gamma,\delta,\varepsilon)$ have no prime factors $p \leq D_0$, so we do not need to remove any factors. The condition $\log_2(Q), \log_2(R) \in \mathbb{Z}$ was included to highlight that we use dyadic decomposition and only need to consider $\ll \log(x)^2$ choices of $Q$ and $R$. Finally, we may continue to assume that $qr$ is squarefree. 

\medskip On page 45, bound $RQ^2 \ll x$ is observed. Since $RQ \ll x^{1/2+2\omega}$ and $RQ^2 \ll x^{1+4\omega}/R \ll x^{1+4\omega -\gamma +\delta +3\varepsilon}$, this gives us the condition $\gamma-4\omega  -\delta  >0$. (Recall that we choose $\varepsilon$ very small depending on $(\omega,\gamma,\delta)$. So we do not keep track of $\varepsilon$ in the inequalities we record.) 

\medskip On page 46, the Cauchy-Schwarz inequality is applied, giving us two copies of $q$, denoted by $q_1$ and $q_2$. On page 48, the gcd $q_0 = (q_1,q_2)$ is introduced. However, later on page 49, relabeling occurs: $q_1$ is replaced by $q_0 q_1$ and $q_2$ is replaced by $q_0q_2$ with $(q_1,q_2)=1$. Thus from now on we   sum over  $q_0q_1 \in \mathcal{Q}(Q;x,\omega,\gamma,\delta,\varepsilon)$ and $q_0q_2 \in \mathcal{Q}(Q;x,\omega,\gamma,\delta,\varepsilon)$. Having made these small changes, inequality (5.31) in Theorem~5.8 of~\cite{Polymath:2014:EDZ} now instead reads $\Sigma_{\mathcal{Q}, \mathcal{R}}(Q,R,q_0,b_1,b_2,\ell) \ll_\varepsilon x^{-2\varepsilon} Q^2RN (q_0, \ell) q_0^{-2}.$
\end{proof}

\subsubsection{Proof of Lemma~\ref{lem:typeIIPoly}} The proof of Theorem 2.8(iv) of~\cite{Polymath:2014:EDZ} continues on pages 51 to 54 and is very straightforward: we do not need to factorize moduli any further.  However, it only applies to $\gamma \geq \frac{1}{2} -2\omega$. Making a very minor change, we extend the proof to smaller values of $\gamma$, which serves as a replacement for Theorem 2.8(i). (Theorem 2.8(i) is a bit stronger, but this strength is not needed for our proofs, since $\gamma$ close to $\frac{1}{2}$ behave well and do not constitute a critical parameter range.)

\begin{proof}[Proof of Lemma~\ref{lem:typeIIPoly}]
We first apply Lemma~\ref{lem:theorem58} with a specific choice of $\mathcal{Q}$ and $\mathcal{R}$: 
Recall first that in Lemma~\ref{lem:typeIIPoly} we work with moduli $$D_{IIa}(x;\omega,\gamma,\delta;\varepsilon) = \left\{d \in  [1, x^{1/2+2\omega}]\cap \mathbb{N}:  \exists r \mid d \mbox{ with } x^{\gamma-\delta -3\varepsilon} < r < x^{\gamma - 3\varepsilon} \right\}.$$
So $d \in D_{IIa}(x;\omega,\gamma,\delta;\varepsilon) $ implies $d=qr$  with $x^{\gamma-\delta -3\varepsilon} < r < x^{\gamma - 3\varepsilon}$. Suppose now that $\prod_{p \mid d, p \leq D_0} p \leq \exp(\log^{2/3} x)$. Then $q=q^* s$ for some $q^*$ with $(q^*,p)=1$ for $p \leq D_0$ and some $s$ with $s \ll x^{\varepsilon}$. Thus we can factorise $d$ as $d=q^* (rs)$ where $x^{\gamma-\delta -3\varepsilon} < rs < x^{\gamma - 2\varepsilon}$. 
Hence    we can simply take
\begin{align*}
&\mathcal{Q}(Q;x,\omega,\gamma,\delta,\varepsilon)=\{q\in [Q,2Q] \cap \mathbb{N}: (q,p)=1 \mbox{ for } p \leq D_0\},\\
&\mathcal{R}(R;x,\omega,\gamma,\delta,\varepsilon)=[R,2R] \cap \mathbb{N}.
\end{align*}
Applying Lemma~\ref{lem:theorem58}, we must now show that 
\begin{align}\label{inequ:basic.poly(iv)}
\sup_{(Q,R,q_0,b_1,b_2,\ell) \in Z(\omega,\gamma,\delta;x)}\Sigma_{\mathcal{Q},\mathcal{R}}(Q,R,q_0,b_1,b_2,\ell) \ll_\varepsilon x^{-2\varepsilon} Q^2RN (q_0, \ell) q_0^{-2}.
\end{align}
Fixing a value of $r$, replacing   condition  $q_0q_1, q_0 q_2 \in \mathcal{Q}(Q;x,\omega,\gamma,\delta,\varepsilon)$  by $q_0q_1, q_0 q_2 \in [Q,2Q]$ (which increases the LHS), and replacing the absolute value signs in $\Sigma_{\mathcal{Q},\mathcal{R}}(Q,R,q_0,b_1,b_2,\ell)$ by constants $c_{h,q_1,q_2}$, we end up exactly with condition (5.33) on page 51 of~\cite{Polymath:2014:EDZ}. The remaining proof of Theorem 2.8(iv) does not use any further factorization of moduli. Hence  we can simply copy the arguments on pages 51 to 54 to show that (\ref{inequ:basic.poly(iv)}) holds provided that
\begin{align*}
&\frac{x^{1+12\omega+7\delta/2}}{N^{5/2}} \ll_\varepsilon x^{-100\varepsilon}, \\
&\frac{x^{8\omega + 3\delta}}{N}\ll_\varepsilon x^{-100\varepsilon}.
\end{align*}
(These inequalities can be found just below (5.36) on page 53.) At this point Polymath uses the assumption $\gamma > \frac{1}{2}-2\omega -c$. However, we do not want to impose such a restriction and instead directly substitute $N=x^\gamma$. The inequalities listed above hold for sufficiently small $\varepsilon$, provided the following is true:
\begin{align*}
&24\omega + 7\delta -5\gamma <-2, \\
&8\omega + 3\delta -\gamma <0.
\end{align*}
Finally, recall that Lemma~\ref{lem:theorem58} had condition $ 4\omega +\delta  -\gamma<0$. However, this condition is already implied by $8\omega + 3\delta -\gamma <0$, so we can discard it.
\end{proof}

\subsubsection{Exponential sum bounds}

After arriving at Theorem~5.8,   both Polymath~\cite{Polymath:2014:EDZ} and Baker and Irving~\cite{Baker:2017:BIP} apply exponential sum bounds which use the $q$-van der Corput method. (Theorem~2.8(iv), discussed above, is the one exception.) More specifically, they use the first inequality in Corollary~4.16 of~\cite{Polymath:2014:EDZ}.   However, using the $q$-van der Corput method necessitates the presence of another factor of convenient size in $d$. To keep conditions on the moduli simple, we would rather avoid this. Hence we will use a different exponential sum bound: the second inequality in Corollary~4.16 of~\cite{Polymath:2014:EDZ}, recorded below. This simpler (but potentially weaker) bound does not require us to consider further factorizations of $d$. 

\begin{lemma}[Corollary~4.16 of~\cite{Polymath:2014:EDZ}, inequality 2]\label{lem:cor416}
Let $N \geq 1$ and suppose $\psi_N(n)$ is a smooth coefficient sequence at scale $N$. Let $d_1, d_2 \in \mathbb{Z}$ be squarefree, and let $a, c_1, c_2, l_1, l_2 \in \mathbb{Z}$. Suppose $d \mid [d_1, d_2]$. Write $\delta_i = d_i/(d_1,d_2)$ and  $\delta'_i = \delta_i/(d,\delta_i)$. Then for any $\varepsilon>0$, 
\begin{align*}
\left|\sum_{n \equiv a (d)} \psi_N(n) e_{d_1}\!\left(\frac{c_1}{n+l_1} \right) e_{d_2}\!\left(\frac{c_2}{n+l_2} \right)\right| \ll_\varepsilon [d_1,d_2]^\varepsilon \left( \left(\frac{[d_1,d_2]}{d}\right)^{1/2} + \frac{1}{d} \frac{(c_1, \delta'_1)}{\delta'_1} \frac{(c_2, \delta'_2)}{\delta'_2} N \right).
\end{align*}
\end{lemma}

Since this exponential sum bound uses gcds, the following simple lemma will also be useful: 
\begin{lemma}[Lemma 1.4 of~\cite{Polymath:2014:EDZ}] Let $q \in \mathbb{N}$ and $K \geq 1$. Then 
\begin{align*}
\sum_{1 \leq k \leq K} (k,q) \leq K \tau(q).
\end{align*}
\end{lemma}

\subsubsection{Proof of Lemma~\ref{lem:typeI(ii)Poly}} The proof of Theorem 2.8(ii) of~\cite{Polymath:2014:EDZ} involves a further factorization of $d=qr$ into $d=q_0uvr$, which needs to be reflected in our choice of $\mathcal{D}$, $\mathcal{Q}$ and $\mathcal{R}$. In~\cite{Polymath:2014:EDZ}, the required size of $u$ changes depending on $q_0$, which is rather inconvenient. We thus carefully replace the  range for $u$ by a fixed interval, independent of $q_0$. Additionally (as mentioned) above, we need to replace inequality 1 of Corollary 4.16 by inequality 2, so that we can simplify the factorization conditions further.

\begin{proof}[Proof of Lemma~\ref{lem:typeI(ii)Poly}] 
Recall first that in Lemma~\ref{lem:typeI(ii)Poly} we work with moduli \begin{align*}
D_{IIb}(x;\omega,\gamma,\delta;\varepsilon) = \left\{d \in  [1, x^{1/2+2\omega}]\cap \mathbb{N}:  \begin{array}{l} 
\exists r \mid d \,\, \exists u \mid \frac{d}{r} \mbox{ with } x^{\gamma -3\varepsilon-\delta} < r < x^{\gamma - 3\varepsilon} \\ \mbox{and } x^{1/2 -\gamma-2\omega -6 \varepsilon-\delta } < u < x^{1/2-\gamma-2\omega  -6\varepsilon}
\end{array}\right\}.
\end{align*}
Hence    we can  take
\begin{align*}
&\mathcal{Q}(Q;x,\omega,\gamma,\delta,\varepsilon)=\left\{q\in [Q,2Q] \cap \mathbb{N}: 
 \begin{array}{l}  (q,p)=1 \mbox{ for } p \leq D_0 \\ \mbox{and }
\exists u \mid q \mbox{ with }x^{1/2 -\gamma-2\omega -7 \varepsilon-\delta } < u < x^{1/2-\gamma-2\omega  -6\varepsilon} \\
\end{array}\right\},\\
&\mathcal{R}(R;x,\omega,\gamma,\delta,\varepsilon)=[R,2R] \cap \mathbb{N}.
\end{align*}
Applying Lemma~\ref{lem:theorem58}, we must now show that 
\begin{align}\label{inequ:basic.poly(ii)}
\sup_{(Q,R,q_0,b_1,b_2,\ell) \in Z(\omega,\gamma,\delta;x)}\Sigma_{\mathcal{Q},\mathcal{R}}(Q,R,q_0,b_1,b_2,\ell) \ll_\varepsilon x^{-2\varepsilon} Q^2RN (q_0, \ell) q_0^{-2}.
\end{align}
We follow the steps on pages 57 to 59 of~\cite{Polymath:2014:EDZ}, making changes as follows: On page 57 Polymath uses that $q_0q_1$  is $x^{\delta+o(1)}$-densely divisible to deduce that $q_1$ has a factor $u_1$ between $q_0^{-1} x^{-\delta-2\varepsilon} Q/H$ and $ x^{-2\varepsilon} Q/H$. We instead have $q_0q_1 \in \mathcal{Q}(Q;x,\omega,\gamma,\delta,\varepsilon)$ and thus get $q_0q_1=u v$ with
$$x^{1/2 -\gamma-2\omega -7 \varepsilon-\delta } < u < x^{1/2-\gamma-2\omega  -6\varepsilon}.$$
Recalling that $q_0q_1r$ is squarefree and writing $u_1=u/(u,q_0)$ and $v_1 = v/(v,q_0)$, we then have $q_1 =u_1 v_1$ with 
$q_0^{-1} x^{1/2 -\gamma-2\omega -7 \varepsilon-\delta } < u_1 < x^{1/2-\gamma-2\omega  -6\varepsilon}.$ 
Thus we replace bound (5.39)  on page 57 by
\begin{align*}
&q_0^{-1} x^{1/2 -\gamma-2\omega -7 \varepsilon-\delta }  \ll U \ll x^{1/2-\gamma-2\omega  -6\varepsilon}.
\end{align*} 
Of course we still have $UV \asymp Q/q_0$. 
(Note: The condition $\gamma < 1/2 - 2\omega $ ensures that the upper bound on $U$ is greater than $1$, provided $\varepsilon$ is sufficiently small.) \medskip

Continuing on, we can follow the exact same steps as Polymath to arrive at (5.42) on page 58. Summarizing their results up to this point, we  now need to show the following:
\begin{align}\label{inequ:poly(ii)aim}
  \Upsilon_2&:=\sum_{u_1 \asymp U} \sum_{q_2 \asymp Q/q_0}  
 \mathop{ \sum\sum}_{\substack{ v_1,v_2\asymp V \\ (u_1v_1v_2,q_0q_2)= 1}} \mathop{\sum\sum}_{1\leq |h_1|,|h_2|\leq \frac{x^\varepsilon RQ^2}{q_0 M}}
 \sup_{t (q_0)} \left|
  \sum_{n \equiv t (q_0)} \psi_N(n) \Phi_{\ell}(h_1,n, r,q_0,u_1v_1, q_2)
\overline{ \Phi_{\ell}(h_2,n, r,q_0,u_1v_2,q_2)} \right|\\
&\ll_\varepsilon  x^{-10\varepsilon} q_0^{-1}Q^2 V N.
\nonumber
\end{align}
At this point Polymath applies inequality 1 of Corollary~4.16. We instead apply inequality 2 (Lemma~\ref{lem:cor416}). Like Polymath, we   use  $q_0$ in the place of $d$,  
$r q_0u_1[v_1,v_2]$ in the place of $d_1$ and $q_2$ in the place of $d_2$, and have  $(c_2, \delta'_2)/\delta'_2 \leq 1$ and $(c_1, \delta'_1)/\delta'_1 \leq (c_1, r)/r = (h_1 v_2-h_2v_1,r)/r$. Hence 
 \begin{align*}
 \Upsilon_2 &\ll_\varepsilon x^\varepsilon \sum_{u_1 \asymp U} \sum_{q_2 \asymp Q/q_0}  
 \mathop{ \sum\sum}_{\substack{ v_1,v_2\asymp V \\ (u_1v_1v_2,q_0q_2)= 1}} \mathop{\sum\sum}_{1\leq |h_1|,|h_2|\leq \frac{x^\varepsilon RQ^2}{q_0 M}}   
  \left( \left(r u_1[v_1,v_2]q_2\right)^{1/2} + \frac{N}{q_0} \frac{ (h_1 v_2-h_2v_1,r)}{r}   \right)\\ 
  &\ll_\varepsilon x^{2\varepsilon}  q_0^{-3}  U^{3/2}  
V^3     R^{5/2}Q^{11/2}  M^{-2}   
  + x^{4\varepsilon}\sum_{u_1 \asymp U} \sum_{q_2 \asymp Q/q_0}  \sum_{|\Delta| \ll \frac{x^\varepsilon RQ^2V}{q_0 M}}    \frac{N(\Delta,r)}{q_0 r}  
 \mathop{ \sum }_{\substack{ v \asymp V  }} \mathop{\sum }_{1\leq |h  |\leq \frac{x^\varepsilon RQ^2}{q_0 M}}    \tau(\Delta-hv) \\
   &\ll_\varepsilon x^{2\varepsilon}  q_0^{-3}  U^{3/2}  
V^3     R^{5/2}Q^{11/2}  M^{-2}   
  + x^{10\varepsilon} q_0^{-4} UV^2 R Q^5 N M^{-2} + x^{10\varepsilon} q_0^{-3}UV RQ^3 NM^{-1}.
 \end{align*}
 The bound (\ref{inequ:poly(ii)aim}) is therefore satisfied if the following inequalities are true: 
 \begin{align*}
  & q_0^{-2}  U^{3/2}  V^2     R^{5/2}Q^{7/2}  M^{-2} N^{-1}   \ll_\varepsilon x^{-100\varepsilon},  \\     
 &  q_0^{-3} UV  R Q^3   M^{-2} \ll_\varepsilon x^{-100\varepsilon}, \\
 &  q_0^{-2}U  RQ  M^{-1}\ll_\varepsilon x^{-100\varepsilon}.
 \end{align*}
Recall now that $RQ \ll x^{1/2+2\omega}$ and $R \gg x^{\gamma -3\varepsilon-\delta}$  and $q_0^{-1} x^{1/2 -\gamma-2\omega -7 \varepsilon-\delta }  \ll U \ll x^{1/2-\gamma-2\omega  -6\varepsilon}$ and $UV \asymp Q/q_0$.  Hence it suffices to have (for  sufficiently small $\varepsilon$):
  \begin{align*}
   & - \tfrac{1}{2}(\tfrac{1}{2} -\gamma -2 \omega -\delta)   +\tfrac{11}{2} (\tfrac{1}{2}+2\omega)-3(\gamma-\delta) + (\gamma -2)  <0,  \\     
 &    4 (\tfrac{1}{2} + 2\omega) -3(\gamma-\delta)-2(1- \gamma) <0, \\
 & (\tfrac{1}{2}-\gamma -2\omega - 6\varepsilon)+  (\tfrac{1}{2}+2\omega)- (1-\gamma) <0.
 \end{align*}
 Rearranging, we obtain our proposed conditions
 \begin{align*}
  &          24\omega +7\delta  -3\gamma  <-1,  \\     
 &        8\omega+3\delta -  \gamma    <0.
 \end{align*} 
  The  condition $ 4\omega +\delta  -\gamma<0$ of Lemma~\ref{lem:theorem58} is again already implied by $8\omega + 3\delta -\gamma <0$.
\end{proof}

\subsubsection{Proof of Lemma~\ref{lem:typeIBI}}  Baker and Irving~\cite{Baker:2017:BIP} use the assumption that coefficient sequence $\beta$ is smooth and  apply inequality 1 of Corollary~4.16 of~\cite{Polymath:2014:EDZ} directly to $\Sigma_{\mathcal{Q},\mathcal{R}}$. We instead use inequality 2. 

\begin{proof}[Proof of Lemma~\ref{lem:typeIBI}] We first consider $\gamma \leq \frac{1}{2}$. 
Recall that we work with moduli $$D_I(x;\omega,\gamma,\delta;\varepsilon) = \left\{d \in  [1, x^{1/2+2\omega}]\cap \mathbb{N}:  \exists r \mid d \mbox{ with } x^{\gamma-\delta -3\varepsilon} < r < x^{\gamma - 3\varepsilon} \right\}.$$ 
Hence we can simply take
\begin{align*}
&\mathcal{Q}(Q;x,\omega,\gamma,\delta,\varepsilon)=\{q\in [Q,2Q] \cap \mathbb{N}: (q,p)=1 \mbox{ for } p \leq D_0\},\\
&\mathcal{R}(R;x,\omega,\gamma,\delta,\varepsilon)=[R,2R] \cap \mathbb{N}.
\end{align*}
Applying Lemma~\ref{lem:theorem58} with this choice of $\mathcal{Q}$ and $\mathcal{R}$, we must now show that 
\begin{align*}
\sup_{(Q,R,q_0,b_1,b_2,\ell) \in Z(\omega,\gamma,\delta;x)}\Sigma_{\mathcal{Q},\mathcal{R}}(Q,R,q_0,b_1,b_2,\ell) \ll_\varepsilon x^{-2\varepsilon} Q^2RN (q_0, \ell) q_0^{-2}.
\end{align*}
Like in the proof of Lemma~3 of~\cite{Baker:2017:BIP}, we observe that by fixing $r \in [R,2R]$, this problem can be reduced to showing that
\begin{align*}
&  \mathop{\sum\sum}\limits_{\substack{q_1, q_2 \\ q_0 q_1,q_0 q_2 \in \mathcal{Q}(Q;x,\omega,\gamma,\delta,\varepsilon) \\ (q_1,q_2)=1 \\  q_0q_1r, q_0q_2r \mbox{\begin{scriptsize}
squarefree
\end{scriptsize} }
}} \sum_{0<|h| \leq \frac{x^\varepsilon RQ^2}{q_0 M}}  (q_0,\ell) \sup_{ t(q_0)}   \Bigg| \sum_{\substack{n  \equiv t (q_0) }} \!\!\!\! \beta(n)\overline{\beta(n+\ell r)}   e_{rq_0q_1}\left(\dfrac{z h}{nq_2}\right) e_{q_2}\left(\dfrac{b_2h}{(n+\ell r)rq_0q_1}\right) \Bigg|  \\ &\ll_\varepsilon x^{-2\varepsilon} Q^2N (q_0, \ell) q_0^{-2},
\end{align*} 
for a suitable choice of $z=z(a,b_1,r,q_0,q_1,q_2)$ with $(z,rq_0q_1)=1$. 
\medskip 

We now apply Lemma~\ref{lem:cor416} with $\beta(n)\overline{\beta(n+\ell r)}$ in the place of $\psi_N(n)$, $t$ in the place of $a$, $q_0$ in the place of $d$, $r q_0q_1$ in the place of $d_1$ and $q_2$ in the place of $d_2$. Further, we take $zh/q_2$ in the place of $c_1$, $b_2h/(rq_0q_1)$ in the place of $c_2$, $0$ in the place of $l_1$ and $\ell r$ in the place of $l_2$. This are the same choices as in~\cite{Baker:2017:BIP}, and thus we also get $[d_1,d_2]= rq_0q_1q_2$ and $\delta'_1 = rq_1$ and $\delta'_2=q_2$. Further, $(c_1, \delta'_1)= (h,rq_1)$ and $(c_2, \delta'_2)=(h,q_2)$.  Hence we find that
\begin{align*}
 &\mathop{\sum\sum}\limits_{\substack{q_1, q_2 \\ q_0 q_1,q_0 q_2 \in \mathcal{Q}(Q;x,\omega,\gamma,\delta,\varepsilon) \\ (q_1,q_2)=1 \\  q_0q_1r, q_0q_2r \mbox{\begin{scriptsize}
squarefree
\end{scriptsize} }
}} \sum_{0<|h| \leq \frac{x^\varepsilon RQ^2}{q_0 M}}  (q_0,\ell) \sup_{ t(q_0)}   \Bigg| \sum_{\substack{n  \equiv t (q_0) }} \!\!\!\! \beta(n)\overline{\beta(n+\ell r)}   e_{rq_0q_1}\left(\dfrac{z h}{nq_2}\right) e_{q_2}\left(\dfrac{b_2h}{(n+\ell r)rq_0q_1}\right) \Bigg|  \\ 
&\ll_\varepsilon 
\mathop{\sum\sum}\limits_{\substack{q_1, q_2 \\  q_0 q_1,q_0 q_2 \in \mathcal{Q}(Q;x,\omega,\gamma,\delta,\varepsilon) \\ (q_1,q_2)=1 \\  q_0q_1r, q_0q_2r \mbox{\begin{scriptsize}
squarefree
\end{scriptsize} }
}} \sum_{0<|h| \leq \frac{x^\varepsilon RQ^2}{q_0 M}}  (q_0,\ell) (rq_0q_1q_2)^\varepsilon \left( ( rq_1q_2)^{1/2} +   \frac{(h,rq_1q_2)}{rq_0q_1q_2} N\right) \\
&\ll_\varepsilon  x^{3\varepsilon}
\mathop{\sum\sum}\limits_{\substack{q_1, q_2 \\ q_0 q_1,q_0 q_2 \in \mathcal{Q}(Q;x,\omega,\gamma,\delta,\varepsilon) \\ (q_1,q_2)=1 \\  q_0q_1r, q_0q_2r \mbox{\begin{scriptsize}
squarefree
\end{scriptsize} }
}}  (q_0,\ell)  \left( \frac{R^{1/2}Q}{q_0}+   \frac{q_0 N}{RQ^2}  \right)  \left( \frac{  RQ^2}{q_0 M}\right)  \ll_\varepsilon 
x^{3\varepsilon} \frac{(q_0,\ell)}{q_0^2}  \left( \frac{R^{3/2}Q^5}{ M}+   \frac{ Q^2 N}{M}  \right).  
\end{align*}
To get the desired upper bound $\ll_\varepsilon x^{-2\varepsilon} Q^2N (q_0, \ell) q_0^{-2}$, we thus  simply need
\begin{align*}
    \left( \frac{R^{3/2}Q^3}{ M N}+   \frac{  1}{M}  \right)
\ll_\varepsilon x^{-5\varepsilon}   .  
\end{align*}
Recalling $NM \asymp x$ and $RQ \ll x^{1/2+2\omega}$ and $R \gg x^{\gamma-\delta-3\varepsilon}$, and choosing $\varepsilon$ sufficiently small (depending on $(\omega,\gamma,\delta)$), the inequality holds provided that
\begin{align*}
3(1/2+2\omega) - 3/2(\gamma-\delta ) <1.
\end{align*}
This rearranges to condition $3\gamma-12\omega  -3\delta >1 $ and concludes the treatment of $\gamma \leq \frac{1}{2}$.

\medskip Now we consider $\gamma \in [\frac{1}{2},\frac{1}{2} + 2\omega +\varepsilon]$.    If $24\omega + 7 \delta - 5(1-\gamma) < -2$ and $8 \omega + 3\delta -(1-\gamma)<0$, we can  apply our Lemma~\ref{lem:typeIIPoly} (with the roles of $\alpha$ and $\beta$ swapped and $\gamma$ replaced by $1-\gamma$). Since $\gamma \leq \frac{1}{2} + 2\omega +\varepsilon$, these inequalities are satisfied provided that 
\begin{align*}
\min\{\tfrac{3}{5} - \tfrac{7}{5} \delta  -\tfrac{24}{5}\omega, 1 -8\omega - 3\delta \} > \tfrac{1}{2}+2\omega.
\end{align*}
Rearranging, the condition becomes
\begin{align*}
&68\omega + 14 \delta    <1.
\end{align*}

\medskip Finally, if $\gamma > \frac{1}{2} + 2\omega +\varepsilon$, we can (like Baker and Irving) apply the results of pages 21 and 22 of Polymath~\cite{Polymath:2014:EDZ}:  for such large $\gamma$ we have the desired equidistribution estimates  for any $d \leq x^{1/2+2\omega}$, we do not need any factorization conditions.

\medskip
Overall we thus have the desired equidistribution properties if $68\omega + 14 \delta    <1$ and $3\gamma-12\omega  -3\delta >1 $.
\end{proof}

\subsubsection{Proof of Lemma~\ref{lem:typeIS}}   Theorem~1 of~\cite{Stadlmann:2023:PAP} only applies to $x^\delta$-smooth numbers. Again we extract much weaker factorization requirements from its proof to obtain Lemma~\ref{lem:typeIS}. Fortunately, the original proof  does not involve  inequality 1 of Corollary~4.16 of~\cite{Polymath:2014:EDZ}, so although we do again change some factor sizes a bit (depending on $q_0$), calculations only change minimally.

\begin{proof}[Proof of Lemma~\ref{lem:typeIS}]
Recall that we work with moduli
\begin{align*}
D_{IIc}(x;\omega,\gamma,\delta;\varepsilon) = \left\{d \in  [1, x^{1/2+2\omega}]\cap \mathbb{N}:  \begin{array}{l} 
\exists r \mid d \,\, \exists u \mid \frac{d}{r} \,\, \exists d_1 \mid r \mbox{ with } x^{\gamma -3\varepsilon-\delta} < r < x^{\gamma - 3\varepsilon} \\ \mbox{and }    x^{1-\gamma-6\varepsilon-\delta}   d^{-1} < u < x^{1-\gamma-6\varepsilon}   d^{-1} \\ 
\mbox{and }  r^2 x^{2-\gamma-100\varepsilon -\delta}  d^{-4} < d_1 <  r^2 x^{2-\gamma-100\varepsilon}  d^{-4}
\end{array}\right\}.
\end{align*}
We thus choose the following sets $\mathcal{Q}(Q,R;x,\omega,\gamma,\delta,\varepsilon)$ and $\mathcal{R}(Q,R;x,\omega,\gamma,\delta,\varepsilon)$:
\begin{align*}
&\mathcal{Q}(Q,R;x,\omega,\gamma,\delta,\varepsilon)=\left\{q\in [Q,2Q] \cap \mathbb{N}: 
 \begin{array}{l}  (q,p)=1 \mbox{ for } p \leq D_0 \\ \mbox{and }
\exists u \mid q \mbox{ with }  x^{1-\gamma-6\varepsilon-\delta}   (QR)^{-1} \ll u \ll x^{1-\gamma-6\varepsilon}   (QR)^{-1}\\
\end{array}\right\},\\
&\mathcal{R}(Q,R;x,\omega,\gamma,\delta,\varepsilon)=\left\{r\in [R,2R] \cap \mathbb{N}: 
 \begin{array}{l}   
\exists d_1 \mid r \mbox{ with }   x^{2-\gamma-52\varepsilon -\delta}  (Q^4R^2)^{-1} \ll d_1 \ll    x^{2-\gamma-52\varepsilon}  (Q^4R^2)^{-1}
\end{array}\right\}.
\end{align*}
For any $d \in D_{IIc}(x;\omega,\gamma,\delta;\varepsilon)$ with 
$\prod_{p \mid d, p \leq D_0} p \leq \exp(\log^{2/3} x)$ there exist  $Q$ and $R$ with $\log_2(Q) \in \mathbb{Z}$ and $\log_2(R) \in \mathbb{Z}$    and $ x^{\gamma -3\varepsilon-\delta} \ll R \ll x^{\gamma - 2\varepsilon} $ so that $d=qr$ for some $q \in  \mathcal{Q}(Q,R;x,\omega,\gamma,\delta,\varepsilon)$ and some $r \in \mathcal{R}(Q,R;x,\omega,\gamma,\delta,\varepsilon)$.   
\medskip

We now follow the proof of Theorem~1 of~\cite{Stadlmann:2023:PAP} and make   small alterations. The original proof starts on  pages~40 to~43 in the appendix of~\cite{Stadlmann:2023:PAP}. These initial steps follow Polymath's proof of Theorem~2.8(ii) rather closely, so we are already familiar with the necessary changes: In~\cite{Stadlmann:2023:PAP},  $d \mid P(x^\delta)$ is used to deduce that $d=qr$ for some $r \in [R,2R]$ with $x^{-4\varepsilon-\delta}N \ll R \ll x^{-2\varepsilon}N$. Although we might not have   $d \mid P(x^\delta)$,   the properties of $D(x;\omega,\gamma,\delta;\varepsilon)$  ensure that we can also factorize like this. In particular, we  can  replace summation conditions like $(q_1 \asymp Q$, $r \asymp R$ and $q_1r \mid P(x^\delta))$ by alternative conditions ($q_1 \in  \mathcal{Q}(Q,R;x,\omega,\gamma,\delta,\varepsilon)$ and $r \in \mathcal{R}(Q,R;x,\omega,\gamma,\delta,\varepsilon)$ and $q_1r$ squarefree).  We also replace condition $(ab_1b_2,P(x^\delta))=1$ by   assumption $(ab_1b_2, P(x))=1$.

\medskip On page~42 the original proof then uses $q_1 \mid P(x^\delta)$ to extract a factor of size $U$ from $q_1$. At this point, we instead have  $q_0q_1 \in  \mathcal{Q}(Q,R;x,\omega,\gamma,\delta,\varepsilon)$ and thus know that $q_0q_1 = u  v $ for some $u$ with  $ x^{1-\gamma-6\varepsilon-\delta}   (QR)^{-1} \ll u \ll x^{1-\gamma-6\varepsilon}   (QR)^{-1}$. We then write $u_1 =u/(u,q_0)$ and $v_1=v/(v,q_0)$ and get $q_1 = u_1 v_2$    for some $u_1 \asymp U$  with  $ q_0^{-1}x^{1-\gamma-6\varepsilon-\delta}   (QR)^{-1} \ll U \ll x^{1-\gamma-6\varepsilon}   (QR)^{-1}$. Recalling the notation $H=x^\varepsilon RQ^2/(q_0 M)$, this condition can be rewritten as 
 $$ q_0^{-2}x^{-\delta - 5\varepsilon} Q/H \ll U \ll q_0^{-1}x^{ - 5\varepsilon} Q/H$$
 Hence we also get  $x^{5\varepsilon} H \ll V \ll q_0 x^{\delta+5\varepsilon} H$. 
We then go to the main text. Our first change is to look at Lemma~3 (the summary of Polymath) and change the ranges of $U$, $V$ and $R$ in Definition~3 on page 10 accordingly.  

\medskip On page 12 of~\cite{Stadlmann:2023:PAP}, the   proof  of Theorem 1 then uses $r \mid P(x^\delta)$ to extract a factor $d_1$ from $r$, this time of size $D = N/(x^{50\varepsilon} H^2)$. We instead recall that $r \in \mathcal{R}(R;x,\omega,\gamma,\delta,\varepsilon)$ and hence  get $r=d_1 r_1 $ for some $d_1 \asymp \Delta$ with  $x^{2-\gamma-52\varepsilon -\delta}  (Q^4R^2)^{-1} \ll \Delta \ll    x^{2-\gamma-52\varepsilon}  (Q^4R^2)^{-1}$. This can be rewritten as
$$ 
\frac{ x^{-\delta} N}{q_0^2 x^{ 50\varepsilon  } H^2}   \ll \Delta \ll   \frac{   N}{q_0^2 x^{ 50\varepsilon  } H^2}.  
$$
In particular, we have   replaced $D\approx N/H^2$ by $D \approx q_0^{-2} N/H^2$, so that  the value of $D$ no longer depends on $q_0$. 
On pages $12$ to $17$ of~\cite{Stadlmann:2023:PAP}, the $q$-van der Corput process is carried out. Although we have slightly changed the ranges of  $R$, $U$ and $\Delta$, the arguments on   pages $12$ to $17$ can be copied with minimal alterations. We do need to be a bit careful when treating   diagonal terms, which are covered by Lemma~7  on page 17 of~\cite{Stadlmann:2023:PAP}: because we increased $V$ by a factor $q_0$, various bounds now include an extra factor $q_0^2$, and so we need to replace $\max\{\frac{x^{\delta+10\varepsilon }H^3 }{(w_2,m)}, \frac{H^4}{(w_2,m)}\}$ by $\max\{\frac{x^{\delta+10\varepsilon }H^3 q_0^2}{(w_2,m)}, \frac{H^4}{(w_2,m)}\}$ in the upper bound on $\Sigma_3$. This extra factor $q_0^2$ is carried through to later lemmas in the paper. 

\medskip Definition 7 and Lemma 8 on page 20 of~\cite{Stadlmann:2023:PAP} summarize results up to this point. Due to extra factors $q_0$, we need to make various amendments in this summary: $v_1, v_2 \ll x^{\delta + 5\varepsilon} H$ becomes $v_1, v_2 \ll q_0 x^{\delta + 5\varepsilon} H$, while $\frac{N}{x^{\delta+55\varepsilon} H^2} \ll \Delta_1 \ll \frac{N}{x^{55\varepsilon} H^2}$ becomes 
 $\frac{N}{q_0^2x^{\delta+55\varepsilon} H^2} \ll \Delta_1 \ll \frac{N}{q_0^2x^{55\varepsilon} H^2}$. Further , $\frac{RQ^2H}{q_0 (v_1,v_2)\Delta_1} \ll m \ll \frac{x^\delta RQ^2H}{q_0 (v_1,v_2)\Delta_1}$ becomes $\frac{RQ^2H}{q_0 (v_1,v_2)\Delta_1} \ll m \ll \frac{x^\delta RQ^2H}{(v_1,v_2)\Delta_1}$ and finally $|\Lambda| \ll \frac{1}{w_1(v_1,v_2)} x^{\delta+5\varepsilon} H^2$ becomes $|\Lambda| \ll q_0 \frac{1}{w_1(v_1,v_2)} x^{\delta+5\varepsilon} H^2$. Fortunately, most of the remaining proof does not refer to specific parameter ranges for $\Delta_1$, $\Lambda$ and $m$, so from page 21 all the way until page 32 we can now proceed with minimal changes. However, it is convenient to notice that in Lemma~10 on page~23, we have the definition $\Delta^* = \min\{\frac{N}{\Lambda x^{5\varepsilon}}, \Delta_1\}$. For our new choice of $\Delta_1$ and for $\varepsilon$ sufficiently small compared to $\delta$, we always have $\frac{N}{\Lambda x^{5\varepsilon}} \gg \frac{w_1 (v_1,v_2) N}{q_0 x^{\delta +10\varepsilon} H^2  } \gg \frac{ N}{q_0 x^{55\varepsilon} H^2  }  \ll \Delta_1$ and can simply choose $\Delta^* = \Delta_1$. 
 
 \medskip To finish the proof, we now just need to amend the computations on page~33 of~\cite{Stadlmann:2023:PAP} with the extra factors $q_0$ we introduced. The equations (3.30), (3.31) and (3.32) which need to be verified now look as follows:
\begin{align} \label{wts1}
&\dfrac{N\Delta_1}{ m}  \ll   \min\left\{\frac{1}{x^{\delta+10\varepsilon}H^3 q_0^2},\frac{1}{H^4}\right\}\frac{(v_1,v_2)q_0^2(q_0,\ell)  N \Delta_1}{ x^{\delta+131\varepsilon}H^2}, \\ \label{wts2}
 &N \ll  \min\left\{\frac{1}{x^{\delta+10\varepsilon}H^3 q_0^2},\frac{1}{H^4}\right\}\frac{(v_1,v_2)q_0^2(q_0,\ell)  N \Delta_1}{ x^{\delta+131\varepsilon}H^2 }, \\ \label{wts3}
&m \ll  \min\left\{\frac{1}{x^{\delta+10\varepsilon}H^3 q_0^2},\frac{1}{H^4}\right\}\frac{(v_1,v_2)q_0^2(q_0,\ell)  N \Delta_1}{ x^{\delta+131\varepsilon}H^2}.
\end{align} 
We begin with (\ref{wts1}). This condition rearranges to
\begin{align} \label{wts11}
\left( \frac{ x^{\delta+131\varepsilon}}{(v_1,v_2)q_0^2(q_0,\ell)  }\right) \dfrac{\max\left\{ q_0^2x^{\delta+10\varepsilon}H^5,H^6\right\}}{ m}  \ll  1. 
\end{align}
Recalling     $m \gg \tfrac{RQ^2H}{q_0 (v_1,v_2)\Delta_1} \gg \tfrac{ q_0 x^{55\varepsilon}RQ^2H^3}{ (v_1,v_2)N} \gg \tfrac{q_0 x^{54\varepsilon}  M H^4}{(v_1,v_2)N}$, we  get that the LHS of (\ref{wts11}) is bounded  by 
\begin{align*} 
\left( \frac{ x^{\delta+131\varepsilon}}{(v_1,v_2)q_0^2(q_0,\ell)  }\right) \dfrac{\max\left\{q_0^2 x^{\delta+10\varepsilon}H^5,H^6\right\}}{ m}  \ll   x^{-1+\delta+77\varepsilon}N^2 \max\left\{x^{\delta+10\varepsilon}H,H^2\right\} \ll  x^{-1+8\omega+3\delta+100\varepsilon}  N^2. 
\end{align*}
This gives us the condition $8\omega+4\delta + 2\gamma < 1$.

\medskip
Next we consider (\ref{wts2}). Rearranging, we find that this condition holds if
\begin{align*}
\frac{x^{\delta+131\varepsilon}}{q_0^2\Delta_1}\max\left\{q_0^2x^{\delta+10\varepsilon}H^5,H^6\right\} \ll  1.
\end{align*}
Recalling that $\Delta_1  \gg \frac{N}{q_0^2x^{\delta+55\varepsilon} H^2}   $ and $H \ll \tfrac{x^{4\omega+\delta+7\varepsilon}}{q_0}$, we get that  (\ref{wts2}) holds if $32 \omega + 10 \delta -\gamma<0 $.
 
\medskip 

 Finally, we consider (\ref{wts3}). Recall that  $m \ll  \tfrac{x^\delta RQ^2H}{  (v_1,v_2)\Delta_1}  \ll \frac{q_0 x^\delta M H^2}{(v_1,v_2)\Delta_1}$. So (\ref{wts3}) is true if 
$$\frac{ x^{1+2\delta+131\varepsilon } }{N^2 \Delta_1^2} \max\left\{q_0^2x^{\delta+10\varepsilon}H^7,H^8\right\}\ll  1.$$
 Using $
 H \ll \tfrac{x^{4\omega+\delta+7\varepsilon}}{q_0}$, we see that (\ref{wts3}) holds if the following two inequalities are satisfied: 
 \begin{align} \label{twofinalinequalities}
 \frac{x^{1+32\omega+  10\delta+200\varepsilon}   }{N^2} \left(\frac{ x^{5\varepsilon}\Lambda}{\Delta_1N}\right)  \ll 1 \quad \mbox{ and } \quad  \frac{x^{1+32\omega+  10\delta+200\varepsilon}   }{N^2} \left(\frac{1}{\Delta_1^2}\right) \ll  1.
\end{align}
Using  $\Lambda \ll x^{\delta+5\varepsilon} H^2$ and  $\Delta_1 \gg \frac{N}{x^{\delta+55\varepsilon}H^2}$, it suffices to have
$ 48\omega +  16 \delta  -4\gamma  < -1 $.
\end{proof}

\subsubsection{Proof of Lemma~\ref{lem:typeIII}}

Theorem~2.8(v) of~\cite{Polymath:2014:EDZ} requires that moduli $d$ can be factorized as $d=qr$ with $r \asymp R$ for many different sizes of $R$. This is a bit inconvenient, and so we change the proof to  work with a fixed value of $R$.  Regarding notation, $\gamma$ in this paper corresponds to $\frac{1}{2}-\sigma$ in~\cite{Polymath:2014:EDZ}. 

\begin{proof}[Proof of Lemma~\ref{lem:typeIII}]
The proof of Theorem~2.8(v)  begins on page~85 of~\cite{Polymath:2014:EDZ}, but its reduction to exponential sums does not need any factorization of moduli, and so we will mainly focus on the end of the proof, pages 90 to 93. 
Prior to page 90, factorization is mentioned once, below (7.17). There we have $b= \prod_{p \mid b_1b_2b_3} p$ and it is observed that $q=bd$ implies $d \in \mathcal{D}_I(bx^\delta)$. We instead recall that in Lemma~\ref{lem:typeIIPoly} we work with  $$D_{III}(x;\omega,\gamma,\delta;\varepsilon) = \left\{q \in  [1, x^{1/2+2\omega}]\cap \mathbb{N}:  \exists r \mid q \mbox{ with } x^{1/3+4\delta/3 -4\omega/3-\delta} < r < x^{1/3+4\delta/3 -4\omega/3} \right\}.$$
We thus define the set
\begin{align*}
\mathcal{D}'(b;x;\omega,\gamma,\delta) =\left\{d \in  \mathbb{N}:  \exists r \mid d \mbox{ with }  b^{-1} x^{1/3+4\delta/3 -4\omega/3-\delta} < r < x^{1/3+4\delta/3 -4\omega/3} \right\}
\end{align*}
Consider a squarefree $q \in D(x;\omega,\gamma,\delta;\varepsilon)$. We   have $q=rs$  for some $r$ with $ x^{1/3+4\delta/3 -4\omega/3-\delta} < r < x^{1/3+4\delta/3 -4\omega/3}$. If  also $q=bd$, we   take $r'=r/(r,b)$ and $s'=s/(s,b)$, so that $d= r' s'$ with $ b^{-1}x^{1/3+4\delta/3 -4\omega/3-\delta} < r' < x^{1/3+4\delta/3 -4\omega/3}$. Hence we replace  condition $d \in \mathcal{D}_I(bx^\delta)$ by condition $d \in \mathcal{D}'(b;x;\omega,\gamma,\delta)$.

\medskip Summarizing results up to page 89,  the goal of pages 90 to 93  is as follows: We must show that
\begin{align}\label{inequ:thm28aim}
\Sigma'_2&=\frac{N_1N_2N_3}{Q^2} \sum_{(b_1,b_2,b_3)} \frac{b_1^\flat b_2^\flat b_3^\flat}{b^2} T(b_1,b_2,b_3) \ll_\varepsilon x^{1-3\varepsilon}   \\
&\mbox{where } T(b_1,b_2,b_3) = \sum_{0 < |\ell|\leq  \frac{x^{3\varepsilon/2} Q^3}{b_1b_2b_3 N}}  
 \Bigg|\sum_{\substack{d \asymp Q/b \\ d\in \mathcal{D}'(b;x;\omega,\gamma,\delta) \\ d \mbox{\begin{scriptsize}
 squarefree \end{scriptsize}   }   \\(b\ell,d)=1}} \eta'_{bd} \sum_{(m,bd)=1} \alpha(m)  \mbox{Kl}_3\left(\frac{a\ell b_1b_2b_3}{b^3m};d \right) \Bigg|
\nonumber
\end{align}
Here $x^{1/2-\varepsilon} \ll Q \ll x^{1/2+2\omega}$,  $\eta'_{bd}$ is some sequence of coefficients supported on short interval $\mathcal{I}(Q)$ and $b= \prod_{p \mid b_1b_2b_3} p$ for some $b_i \in \mathbb{N}$ with $b \ll Q$ and $b_1b_2b_3 \ll x^{3\varepsilon/2} Q^3/N$. Further, $b_i^\flat$ is the largest squarefree divisor of $b_i$ and $b= (b_1b_2b_3)^\flat$. 

\medskip On page 90, Polymath begins by using $d \in \mathcal{D}_I(bx^\delta)$ to deduce that $d=rs$ with   $(bx^\delta)^{-1} S \leq s \leq S$. Of course, we can use $d \in \mathcal{D}'(b;x;\omega,\gamma,\delta)$ to deduce the same with $S= x^{1/3+4\delta/3 -4\omega/3}$. This choice of $S$ satisfies the requirement $1 \leq S \leq x^\delta Q/2$ provided that $4\omega -4\delta <1$ and $2\delta  -8\omega  <   1$.

\medskip Fortunately, Polymath initially proceeds without specifying their value of $S$. As such, we can simply copy their arguments on pages 90 and 91. Noting that $H_b= \frac{x^{3\varepsilon/2}H}{b_1b_2b_3}=  \frac{x^{3\varepsilon/2}Q^3}{b_1b_2b_3 N}$ where $N=N_1N_2N_3$, we get $| T(b_1,b_2,b_3)|^2\leq T_1T_2$ with $T_1 \ll_\varepsilon \frac{x^{3\varepsilon }Q^3 }{b_1b_2b_3 N}$ and
\begin{align*}
T_2=T'_2+T''_2 \ll_\varepsilon \frac{x^{3\varepsilon}  MQ^4 S}{b b_1b_2b_3 N}+
 \frac{x^{3\varepsilon} M^2Q^4S^{1/2} }{b b_1b_2b_3 N}+ \frac{x^{3\varepsilon}   x^{\delta/2}M^2Q^3}{b^{5/2}S^{1/2}}. 
\end{align*}
Multiplying $T_1$ and $T_2$ together, substituting $S= x^{1/3+4\delta/3 -4\omega/3}$ and $Q \ll x^{1/2+2\omega}$ and recalling that $x^{3/2-3\gamma/2 } \ll N \ll x^{3 \gamma }$, we now have 
\begin{align*}
| T(b_1,b_2,b_3)|  &\ll_\varepsilon    \frac{x^{3\varepsilon}  M^{1/2}Q^{7/2} S^{1/2}}{b^{1/2}  b_1b_2b_3  N }+
 \frac{x^{3\varepsilon} M Q^{7/2}S^{1/4} }{b^{1/2}  b_1b_2b_3  N }+ \frac{x^{3\varepsilon}    x^{\delta/4}M Q^3}{b^{5/4} (b_1b_2b_3)^{1/2}S^{1/4}N^{1/2}}  \\
 &\ll_\varepsilon  \left( \frac{x^{3\varepsilon}Q^2 }{N_1N_2N_3 }  \right)
  \left( \frac{   x^{1/2 +3\omega  +3\gamma/4 } S^{1/2}}{b^{1/2}  b_1b_2b_3     }+
 \frac{ x^{1/4 +3\omega  +3\gamma/2}  S^{1/4} }{b^{1/2}  b_1b_2b_3     }+ \frac{    x^{\delta/4 + 3/4+2\omega  +3\gamma/4 }    }{b^{5/4} (b_1b_2b_3)^{1/2}S^{1/4} }  \right)\\
  &\ll_\varepsilon  \left( \frac{x^{3\varepsilon}Q^2 }{N_1N_2N_3 }  \right)
  \left( \frac{   x^{2/3 +7\omega/3  +3\gamma/4  +2\delta/3 }  }{b   (b_1b_2b_3)^{1/2}     }+
 \frac{ x^{1/3 +8\omega/3  +3\gamma/2   +\delta/3 }   }{b   (b_1b_2b_3)^{1/2}   }+ \frac{    x^{7\omega/3  +3\gamma/4 +2/3 -\delta/12    }    }{b  (b_1b_2b_3)^{1/2}  }  \right).
\end{align*}
Substituting this upper bound on $| T(b_1,b_2,b_3)|$ back into (\ref{inequ:thm28aim}), we get 
\begin{align*}
\Sigma'_2 \ll_\varepsilon x^{3\varepsilon} \left(   x^{2/3 +7\omega/3  +3\gamma/4  +2\delta/3 }        +
   x^{1/3 +8\omega/3  +3\gamma/2   +\delta/3 }    +   x^{7\omega/3  +3\gamma/4 - 5/6 -\delta/12    }        \right) \sum_{(b_1,b_2,b_3)} \frac{b_1^\flat b_2^\flat b_3^\flat}{b^3 (b_1b_2b_3)^{1/2} }.
\end{align*}
But by Lemma~7.5 of~\cite{Polymath:2014:EDZ}, the inner sum over $(b_1,b_2,b_3)$ is convergent. Therefore we have $\Sigma'_2 \ll_\varepsilon x^{1-3\varepsilon}$ provided the following three inequalities hold:
\begin{align*}
&2/3 +7\omega/3  +3\gamma/4  +2\delta/3  <1, \\
&1/3 +8\omega/3  +3\gamma/2   +\delta/3 <1, \\
&7\omega/3  +3\gamma/4 +2/3 -\delta/12  <1.
\end{align*}
The second and third inequality are already implied by the first, and so this leaves us with the condition $28\omega   +9\gamma   +8\delta   <4$.
\end{proof}

\section{Finding permissible parameters}\label{sec:domain}

In the previous sections we first derived a version of the GPY sieve with coefficients supported on $T_k(\delta,\underline{A},\underline{B},\varepsilon)$, conditional on there existing a minorant $\rho(n;x)$ for the prime indicator function which has suitable equidistribution properties for moduli produced by $T_k(\delta,\underline{A},\underline{B},\varepsilon)$. We then derived a number of equidistribution estimates which work well for moduli with large $x^\delta$-smooth factors. Now we need to put these ingredients together to choose good values for $\delta$, $\underline{A}$ and $ \underline{B}$ and construct a suitable $\rho(n;x)$. 

\subsection{Harman's sieve} We begin with a Harman's sieve construction.

\begin{defn}[Decomposition] \label{def:decharman}
Let $\epsilon = 10^{-10}$ and let $\xi_1, \xi_2, \xi_3 \in (0,1)$. Then $\mathcal{H}(\xi_1,\xi_2,\xi_3)$ denotes the set of functions $f: \mathbb{N} \times (1,\infty) \rightarrow \mathbb{C}$ which satisfy   one of the following three conditions: 
\begin{enumerate} 
\item[{\rm (I)}]  $f(n;x) = (\alpha \star \beta)(n;x)$, where $\alpha$ is  a coefficient sequence at scale $M$ and $\beta$ is a  smooth  coefficient sequence at scale $N$ with $M(x) N(x) \asymp x$. Additionally, $N(x) = x^{\gamma(x)}$ with $
\gamma(x) \geq  \xi_1-\epsilon.
$
\item[{\rm (II)}]  $f(n;x) = (\alpha \star \beta)(n;x)$, where $\alpha$ is a coefficient sequence at scale $M$, $\beta$ is a  coefficient sequence at scale $N$ with $M(x) N(x) \asymp x$, and $\alpha$ and $\beta$ both have the Siegel--Walfisz property.   Additionally, $N = x^{\gamma(x)}$ with 
$\xi_2 -\epsilon \leq \gamma(x) \leq 1- \xi_2+\epsilon$. 
\item[{\rm (III)}] $f(n;x) = (\alpha \star \psi_1 \star \psi_2 \star \psi_3)(n;x)$, where $\alpha$ is a coefficient sequence at scale $M$ and $\psi_1$, $\psi_2$ and $\psi_3$ are   smooth  coefficient sequences at scales $N_1$, $N_2$ and $N_3$ with $M(x) N_1(x) N_2(x) N_3(x) \asymp x$.   Additionally, $x^{1-2\xi_3-\epsilon} \leq N_1(x), N_2(x), N_3(x) \leq x^{\xi_3+\epsilon}$ and $N_i(x)N_j(x) \geq x^{1-\xi_3-\epsilon}$ for $i \neq j$.
\end{enumerate} 
\end{defn}
We will now use    $\mathcal{H}(\xi_1,\xi_2,\xi_3)$   to construct a minorant for the primes.

\begin{prop} \label{prop:harman}
Let  $T_k(\delta,\underline{A},\underline{B},\varepsilon)$ be as given in Definition~\ref{def:intregion}, let 
$Q^*(x;\delta,\underline{A},\underline{B},\varepsilon,\varepsilon_0)$ be as given in Definition~\ref{def:neededmoduli} and let $\mathcal{H}(\xi_1,\xi_2,\xi_3)$ be as given in Definition~\ref{def:decharman}.

Suppose $\xi_1, \xi_2, \xi_3 \in (0,1)$ satisfy all of the following inequalities:
\begin{align*}
& 2\xi_1 + 3\xi_2   <2, \\
& \xi_2 \leq \xi_3,\\
&\xi_1+ 9\xi_2 < 4, \\
&2\xi_1+ \xi_2>1,\\
&17 \xi_2 < 7.
\end{align*}

Suppose  further that for every $f \in \mathcal{H}(\xi_1,\xi_2,\xi_3)$,  $x>1$, $\varepsilon_0 >0$, $A>0$ and    $a\in \mathbb{Z}$  with $(a,P(x))=1$,   
\begin{align}\label{inequ:harmrequ}
\sum_{\substack{q \in Q^*(x;\delta,\underline{A},\underline{B},\varepsilon,\varepsilon_0) \\  q \mbox{ \scriptsize  is squarefree  }   }}  \Bigg| \sum_{\substack{n \in [x,2x] \\ n \equiv a (q)}} f(n;x) - \dfrac{1}{\phi(q)} \sum_{\substack{ n \in [x,2x] \\ (n,q)=1}}f(n;x) \Bigg| \ll_{A,\varepsilon_0} \dfrac{x}{\log(x)^A}.
\end{align}

We then  write $p_j = x^{\alpha_j}$ and define the minorant 
\begin{align*}
\rho(n;x) = 1_{\mathbb{P}}(n) - \sum_{\substack{ n = p_1p_2p_3p_4p_5 \\ 1 -2\xi_2 < \alpha_4 < \alpha_3 < \alpha_2 < \alpha_1 < \xi_2 \\ \alpha_1+\alpha_2 < \xi_2 \\ \alpha_2+\alpha_3+\alpha_4 >1-\xi_2 \\ \alpha_5>\alpha_4}} 1- \sum_{\substack{ n = p_2p_3p_4p_5p_6 \\ 1-2\xi_2 < \alpha_2, \alpha_3, \alpha_4, \alpha_5, \alpha_6 < 8\xi_2-3    \\ \alpha_2, \alpha_4 > \alpha_3  \\ \alpha_2+\alpha_4 < \xi_2\\ \alpha_3 + \alpha_4 + \alpha_5 >1-\xi_2 \\   \alpha_6 > \alpha_5 }} 1.
\end{align*}
Furthermore, we choose the constants
\begin{align*}
 c_1&=  \int^{\xi_2}_{1 -2\xi_2}\int^{\min\{\alpha_1, \xi_2-\alpha_1\}}_{1 -2\xi_2}\int^{\alpha_2}_{1 -2\xi_2}\int^{\min\{\alpha_3,(1-\alpha_1-\alpha_2-\alpha_3)/2\}}_{\max\{1 -2\xi_2, 1 - \xi_2-\alpha_2-\alpha_3\}} \frac{1}{\alpha_1\alpha_2\alpha_3\alpha_4(1-\sum_{i=1}^4\alpha_i)} \, \mbox{d}\alpha_1 \mbox{d}\alpha_2\mbox{d}\alpha_3\mbox{d}\alpha_4 \\
&+
 \int^{8\xi_2-3 }_{1 -2\xi_2}\int^{\min\{8\xi_2-3 , \alpha_2\}}_{1 -2\xi_2}\!\!\!\!\int_{\max\{\alpha_3, 1 -2\xi_2\}}^{\min\{8\xi_2-3 , \xi_2-\alpha_2\}}\!\!\!\!\int^{\min\{8\xi_2-3 ,(1-\sum_{i=2}^4\alpha_i)/2\}}_{\max\{1 -2\xi_2,1 - \xi_2-\alpha_3-\alpha_4\}} \frac{1}{\prod_{i=2}^5\alpha_i(1-\sum_{i=2}^5\alpha_i)} \, \mbox{d}\alpha_2 \mbox{d}\alpha_3\mbox{d}\alpha_4\mbox{d}\alpha_5,\\
 c_2 &=  \begin{cases} 
 24 \quad \mbox{ if } \xi_2 > \frac{4}{10}, \\ 
 0  \,\,\, \quad \mbox{ if } \xi_2 \leq  \frac{4}{10}.
 \end{cases}
\end{align*} 
 
Then $\rho(n;x)$ has the following properties:
\begin{enumerate}[{\rm(1)}]
\item  Prime minorant: $-c_2 \leq \rho(n;x) \leq 1_{\mathbb{P}}(n)$ for all large $x$ and $n \in [x,2x]$.  
\item  $\rho$ has  equidistribution properties for  moduli corresponding to $T_k(\delta,\underline{A},\underline{B},\varepsilon)$. $($See Definition~\ref{def:equidist}.$)$
\item If $\rho(n;x) \neq 0$, then all prime factors of $n$ exceed $x^{1-2\xi_2}$.  
\item The sum of $\rho(n;x)$ over $[x,2x]$ satisfies
$\sum_{n \in [x,2x]} \rho(n;x) = (1-c_1 +o(1)) \frac{x}{\log(x)}$.
\end{enumerate}
\end{prop}

(Note: If $\xi_2 \leq  \frac{4}{10}$, then $\rho(n;x)$ is simply $1_{\mathbb{P}}(n)$.)

\begin{proof}
The construction of $\rho(n;x)$ is essentially the same as in Baker and Irving's work~\cite{Baker:2017:BIP} and in the author's paper~\cite{Stadlmann:2023:PAP}.  We have simply replaced specific numeric values by variables. However, along the way numerous inequalities need to be satisfied, so to justify this choice of $\rho(n;x)$, we look a little more closely at the proofs in Section~4 of~\cite{Stadlmann:2023:PAP}: The construction begins on page~35 with the statement of various equidistribution estimates. We replace the values $0.33856$, $0.40481$, $0.59519$, $0.19$, $0.405$ and $0.595$, which are used in~\cite{Stadlmann:2023:PAP}, by $\xi_1$, $\xi_2$, $1-\xi_2$, $1-2\xi_3$, $\xi_3$ and $1-\xi_3$, respectively.  In Lemma~19 on page~35 we then instead have $\alpha = \xi_2$, $\beta = 1-2\xi_2 =\lambda$ and $M= x^{1-\xi_1}$ and $\zeta = 1-\xi_1-\xi_2$. In particular, $0.19038$ is replaced by $1-2\xi_2$ and $0.25663$ is replaced by $1-\xi_1-\xi_2$.  Using Buchstab decompositions as on page 36 of~\cite{Stadlmann:2023:PAP}, and making the   assumption
$$2\zeta = 2( 1-\xi_1-\xi_2 )> \xi_2,$$
we then arrive at  
\begin{align} \label{equ:decstart}
1_{\mathbb{P}}(n) = \theta_0(n;x) + \!\!\!\!\!\!\!\!\!\!\sum_{\substack{ n = p_1p_2n_3 \\ 1-2\xi_2 \leq \alpha_2 < \alpha_1 < \xi_2 \\ \alpha_1+\alpha_2 > 1-\xi_2 \\ \alpha_2 > 1-\xi_1-\xi_2 }} \!\!\!\!\!\!\!\!\!\!\psi(n_3, p_2) + \!\!\!\!\!\!\!\!\!\!\sum_{\substack{ n = p_1p_2p_3p_4n_5 \\ 1-2\xi_2\leq \alpha_4 < \alpha_3 < \alpha_2 < \alpha_1 < \xi_2 \\ \alpha_1+\alpha_2 < \xi_2 \\ \alpha_3 < 1-\xi_1-\xi_2}} \!\!\!\!\!\!\!\!\!\!\!\!\!\! \psi(n_5, p_4)- \!\!\!\!\!\!\!\!\!\!\sum_{\substack{ n = p_1p_2p_3n_4 \\ 1-2\xi_2 \leq \alpha_3 <\alpha_2<\alpha_1 < \xi_2 \\ \alpha_1+\alpha_2 > 1-\xi_2 \\ \alpha_2 < 1-\xi_1-\xi_2 }}\!\!\!\!\!\!\!\!\!\!\!\!\!\!\!\! \psi(n_4, p_3),
\end{align}
where $\theta_0(n;x)$ has  equidistribution properties for  moduli corresponding to $T_k(\delta,\underline{A},\underline{B},\varepsilon)$. 

\medskip For $\xi_1 =0.33856$ and $\xi_2=0.40481$, the first sum in  (\ref{equ:decstart}) is treated in Lemma~20 on page 36 of~\cite{Stadlmann:2023:PAP}. There it is shown that this sum also has the desired equidistribution properties. For this proof to work for our more general parameters $\xi_1$, $\xi_2$ and $\xi_3$, we need the following: At the start of the proof of Lemma~20 we use that $n_3$ must be prime. Since $\alpha_1+\alpha_2 > 1-\xi_2$ and $n_3$ not prime would imply $n_3 > x^{2(1-\xi_1-x_2)}$  this is ensured by assumption $2(1-\xi_1-x_2)  >\xi_2$. Next, in the fourth paragraph on page 37, Polymath's Type III equidistribution estimate is implied. To make sure this is possible, we have the important assumption
$$  \xi_2 \leq \xi_3.$$ 
A paragraph further,  $\mu_2 + \mu_3 < \xi_2$ is treated. There $1-\xi_1-\xi_2/2 > \xi_2$ is used to get $\mu_4, \mu_5>1-2\xi_2$. (For the meaning of $\mu_i$, see~\cite{Stadlmann:2023:PAP}. We skip this discussion here, as it is not relevant for the inequality check.) $\mu_1 + (\mu_6 + \dots + \mu_j)<1-\xi_1-\xi_2$ is then deduced from  $1-4(1-2\xi_2) < 1-\xi_1-\xi_2$, which rearranges to
\begin{align*}
\xi_1+ 9\xi_2 < 4.
\end{align*}
The inequalities $2\xi_1 + 3\xi_2 <2$ and  $\xi_2 < \xi_3$ and $\xi_1+9\xi_2 <4$ are all we need for the proof of Lemma 20 to work.  Moving on, the second sum in (\ref{equ:decstart}) is treated by Lemma 21 on page 37. 

\medskip For the proof of Lemma~21 to work with general $\xi_1$, $\xi_2$ and $\xi_3$, we need the following: First, the second sum is split up depending on whether $\alpha_2+\alpha_3+\alpha_4 \leq 1-\xi_2$ or $\alpha_2+\alpha_3+\alpha_4 > 1-\xi_2$. In the first case, we   need $3 (1-2\xi_2) > \xi_2$ to deduce that the sum has the desired equidistribution properties. Here we get the additional assumption
$ 7 \xi_2 < 3.$ 

\medskip
In the second case, we use $\xi_2/2 < 1-\xi_1-\xi_2$ to drop the assumption $\alpha_3 < 1-\xi_1-\xi_2$, and    $\alpha_2+\alpha_3+\alpha_4 > 1-\xi_2$ means that we need $(4/3)(1-\xi_2)+2(1-2\xi_2) >1$
to deduce that $n_5$ is prime. The condition $(4/3)(1-\xi_2)+2(1-2\xi_2) >1$ rearranges to
$  16\xi_2 < 7
$, which is already implied by $ 7 \xi_2 < 3.$ Hence we are left with the sum 
\begin{align}\label{equ:decpart1}
\Gamma_2(n;x)= \sum_{\substack{ n = p_1p_2p_3p_4p_5 \\ 1-2\xi_2 < \alpha_4 < \alpha_3 < \alpha_2 < \alpha_1 < \xi_2 \\ \alpha_1+\alpha_2 <  \xi_2 \\ \alpha_2+\alpha_3 +\alpha_4 > 1-\xi_2 \\ \alpha_5 > \alpha_4}} \!\!\!\!\!\! 1,
\end{align}
for which we do not have good equidistribution estimates. It thus needs to be subtracted from $1_{\mathbb{P}}(n)$. Using the prime number theorem for short intervals, we find that $\sum_{n \in [x,2x]} \Gamma_2(n;x)$ equals 
\begin{align*}
\dfrac{x(1+o(1))}{\log(x)}\int^{\xi_2}_{1 -2\xi_2}\int^{\min\{\alpha_1, \xi_2-\alpha_1\}}_{1 -2\xi_2}\int^{\alpha_2}_{1 -2\xi_2}\int^{\min\{\alpha_3,(1-\alpha_1-\alpha_2-\alpha_3)/2\}}_{\max\{1 -2\xi_2, 1 - \xi_2-\alpha_2-\alpha_3\}} \frac{1}{\alpha_1\alpha_2\alpha_3\alpha_4(1-\sum_{i=1}^4\alpha_i)} \, \mbox{d}\alpha_1 \mbox{d}\alpha_2\mbox{d}\alpha_3\mbox{d}\alpha_4. 
\end{align*}  

Finally, we need to look at the third sum in (\ref{equ:decstart}). It  is treated by Lemma 22 on page 38 of~\cite{Stadlmann:2023:PAP}.  The proof first needs $(1-\xi_2) + 3(1-2\xi_2) >1$  to deduce that $n_4$ is prime. This is implied by $7\xi_2 <3$. Next, $(1-\xi_2)/2> 1-\xi_1-\xi_2$ is needed to drop condition $\alpha_2 <\alpha_1$. This condition rearranges to 
$$   2\xi_1+ \xi_2  > 1. $$
Further, we have $1-3(1-2\xi_2) < 1-\xi_2$ (since $  7\xi_2  <  3$), and this inequality  is used to drop condition $\alpha_1 < 1-\xi_2$.  At that point the proof of Lemma~22  splits the sum into two parts, $\Upsilon_1$ and $\Upsilon_2$. $\Upsilon_2$ has good equidistribution properties, while reversal of roles is applied to $\Upsilon_1$. We end up with 
 \begin{align*}
\Upsilon_3(n;x) = \sum_{\substack{ n = p_2p_3p_4p_5p_6 \\ 1-2\xi_2 \leq \alpha_3 <\alpha_2 < \xi_2 \\ \alpha_3+\alpha_4 < \xi_2 \\ \alpha_2 < 1-\xi_1-\xi_2 \\ \alpha_4 \geq \alpha_3 \\ 1-2\xi_2 \leq \alpha_5 \leq \alpha_6}} 1, 
\end{align*} 
where $6(1-2\xi_2)>1$  was needed to get the prime $p_6$. This rearranges to condition
$
12\xi_2<5.
$
Note that $
12\xi_2<5 
$ implies $
7\xi_2<3.
$
Next we observe that $\alpha_i > 1-2\xi_2$ for all $i$ implies that also $\alpha_i < 1-4(1-2\xi_2)= 8\xi_2-3   $ for all $i$. Furthermore, $8\xi_2-3 < \xi_2$ since $7\xi_2 <3$. So condition $\alpha_i < \xi_2$ can be replaced by $\alpha_i < 8\xi_2-3  $.  Finally, the proof of Lemma~22  also requires $2( 8\xi_2-3) < 1- \xi_2$ . We need
  $$17\xi_2  < 7 $$
  to get that $\Upsilon_3$ is the sum of a function with the desired equidistribution properties and of $\Gamma^*_2(n;x)$, where
  \begin{align}\label{equ:decpart2}
\Gamma^*_2(n;x)=\sum_{\substack{ n = p_2p_3p_4p_5p_6 \\1-2\xi_2 \leq \alpha_2, \alpha_3, \alpha_4, \alpha_5, \alpha_6 \leq 8\xi_2-3   \\ \alpha_2, \alpha_4 > \alpha_3  \\ \alpha_2+\alpha_4 <\xi_2\\ \alpha_3 + \alpha_4 + \alpha_5 >1- \xi_2 \\   \alpha_6 > \alpha_5 }} 1.
\end{align}
   
Using standard techniques, we find that $\sum_{ n \in[x,2x]}\Gamma^*_2(n;x)$ is asymptotically equal to 
\begin{align*} 
\dfrac{x (1+o(1))}{\log(x)}\!\! \int^{8\xi_2-3 }_{1 -2\xi_2}\int^{\min\{8\xi_2-3 , \alpha_2\}}_{1 -2\xi_2}\!\!\!\!\int_{\max\{\alpha_3, 1 -2\xi_2\}}^{\min\{8\xi_2-3 , \xi_2-\alpha_2\}}\!\!\!\!\int^{\min\{8\xi_2-3 ,(1-\sum_{i=2}^4\alpha_i)/2\}}_{\max\{1 -2\xi_2,1 - \xi_2-\alpha_3-\alpha_4\}} \frac{1}{\prod_{i=2}^5\alpha_i(1-\sum_{i=2}^5\alpha_i)} \, \mbox{d}\alpha_2 \mbox{d}\alpha_3\mbox{d}\alpha_4\mbox{d}\alpha_5.
\end{align*} 
Overall, this tells us that $\rho(n;x)$ has  equidistribution properties for  moduli corresponding to $T_k(\delta,\underline{A},\underline{B},\varepsilon)$, and that $\sum_{n \in [x,2x]} \rho(n;x) = (1-c_1 +o(1)) \frac{x}{\log(x)}$ for the given value of $c_1$. Furthermore, clearly $\rho(n;x) \leq 1_{\mathbb{P}}(n)$ and  $\rho(n;x) \neq 0$   implies that all prime factors of $n$ exceed $x^{1-2\xi_2}$.  Now we   need to think about $c_2$. 

\medskip If $\xi_2 \leq \frac{4}{10}$, we note that $1-2\xi_2 \geq \frac{1}{5}$ and hence  $\alpha_i > \frac{1}{5}$ for all $i$ in both $\Gamma_2(n;x)$ and $\Gamma^*_2(n;x)$. This means that $\Gamma_2(n;x)$ and $\Gamma^*_2(n;x)$ are simply empty sums and $\rho(n;x)=1_{\mathbb{P}}(n)$. We can take $c_2=0$. If $\xi_2 >\frac{4}{10}$ we need to be more careful. Looking at  $\Gamma_2(n;x)$, defined in (\ref{equ:decpart1}), we only get a negative contribution if $n=p'_1p'_2p'_3p'_4p'_5$ for some primes $p'_i$ with $p'_1 > \dots > p'_5$. Further, $p_4$ needs to be the smallest prime, so there are only $4$ choices for $p_5$. Since $p_1 > p_2 > p_3 > p_4$, the choice of $p_5$ immediately   also determines the values of $p_1$, $p_2$ and $p_3$. No $n$ can be counted more than $4$ times. Thus  $\Gamma_2(n;x) \geq -4$. Similarly, looking at $\Gamma^*_2(n;x)$, defined in (\ref{equ:decpart2}), we only get a negative contribution if  $n=p'_1p'_2p'_3p'_4p'_5$ for some primes $p'_i$ with $p'_1 > \dots > p'_5$. There are at most $\binom{5}{3}=10$ choices for $\{p_2,p_3, p_4\}$, and since we need $p_2>p_3$ and $p_4 > p_3$ and $p_6>p_5$, any $n$ can be counted at most $20$ times. We get $\Gamma_2(n;x) \geq -20$ and overall have $\rho(n;x) \geq -24$. This gives us choice $c_2=24$.
\end{proof}

\subsection{Partition lemmas}

In Proposition~\ref{prop:harman} we chose parameters $(\xi_1, \xi_2,\xi_3)$ and constructed a prime minorant $\rho(n;x)$ which has  equidistribution properties for  moduli corresponding to $T_k(\delta,\underline{A},\underline{B},\varepsilon)$ provided that every $f \in \mathcal{H}(\xi_1, \xi_2,\xi_3)$  has equidistribution properties for moduli in $Q^*(x;\delta,\underline{A},\underline{B},\varepsilon,\varepsilon_0)$. We now need to translate the equidistribution estimates proved in Section~\ref{sec:equiest} into conditions on  $(\xi_1, \xi_2,\xi_3)$  and $(\delta, \underline{A}, \underline{B})$ which ensure that $f \in \mathcal{H}(\xi_1, \xi_2,\xi_3)$ equidistributes for moduli in $Q^*(x;\delta,\underline{A},\underline{B},\varepsilon,\varepsilon_0)$. We will do this by thinking  about the following sets:

\begin{defn}\label{def:checkset} For given $B_1, B_2, \delta \in (0,\infty)$ and given $m_1, m_2 \in \mathbb{N}\cup\{0\}$, we set 
$$\Xi(B_1,B_2,m_1,m_2,\delta)=\left\{(y_1, \dots, y_{m_1+m_2}) \in [\delta,1]^{m_1+m_2}: \sum_{i=1}^{m_1} y_i \leq  B_{1}  \mbox{ and }   \sum_{i=m_1+1}^{m_1+m_2} y_i \leq B_{2} \right\}.$$ 
\end{defn}
Properties of $\Xi(B_{j,m},B_{j',m'},m,m',\delta)$ will be used to factorize moduli in $Q(x;\delta,\underline{A},\underline{B},\varepsilon,j,j',m,m',\varepsilon_0)$.

\subsubsection{2 factors}

\medskip We begin by asking what conditions on $T_k(\delta,\underline{A},\underline{B},\varepsilon)$ ensure that   $q \in Q^*(x;\delta,\underline{A},\underline{B},\varepsilon,\varepsilon_0)$ has a factor in some fixed interval $[x^a, x^b]$. 
We make the following simple observation:
\begin{lemma}\label{lem:partition1} Let $Q(x;\delta,\underline{A},\underline{B},\varepsilon,j,j',m,m',\varepsilon_0)$ be as given in Definition~\ref{def:neededmoduli} and let $\Xi(B_1,B_2,m_1,m_2,\delta)$ be as given in Definition~\ref{def:checkset}. 
 Consider $a, b \in (0,\frac{1}{2})$   with $b-a \geq \delta$.

 \smallskip 
  Suppose that for every  $(y_1, \dots, y_{m+m'}) \in \Xi(B_{j,m},B_{j',m'},m,m',\delta)$  there exists a partition $I_1 \cup I_2$ of $\{1, \dots, m+m'\}$ with 
$$ \sum_{i \in I_1} y_i \leq b \quad \mbox{ and } \quad  \sum_{i \in I_2} y_i \leq  \frac{1}{2}-a. $$
Then there exists a constant $\varepsilon_1>0$ (dependent only on $\varepsilon_0$ and $\delta$), so that the following is true for large $x$:  If $q \in Q(x;\delta,\underline{A},\underline{B},\varepsilon,j,j',m,m',\varepsilon_0)$ and $q > x^{1/2-\varepsilon_1}$, then $q$ has some $r \mid q$ with $r \in  [x^{a}, x^b]$. 
\end{lemma} 

\begin{proof} Recall that we defined $Q(x;\delta,\underline{A},\underline{B},\varepsilon,j,j',m,m',\varepsilon_0)$ to be the following set:
\begin{align*}\left\{ e e' \prod_{i=1}^m f_i \prod_{i=1}^{m'} f'_i    \in [1,x]: 
\begin{array}{ll}
  \log_x(f_1 \dots f_m) \leq (1-\varepsilon_0)B_{j,m}, & \log_x(f'_1 \dots f'_{m'}) \leq  (1-\varepsilon_0)B_{j',m'}, \\
  \log_x(f_1 \dots f_m e) \leq (1-\varepsilon_0)(A_j-\varepsilon), &  \log_x(f'_1 \dots f'_{m'} e') \leq (1-\varepsilon_0)(A_{j'}+\varepsilon), \\
ee' \mbox{ is } x^\delta\mbox{-smooth}, & \log_x(f_i), \log_x(f'_i) \geq  \delta \mbox{ for all } i
\end{array}\right\}
\end{align*} 
Fix some $q=e e' \prod_{i=1}^m f_i \prod_{i=1}^{m'} f'_i    \in  Q(x;\delta,\underline{A},\underline{B},\varepsilon,j,j',m,m',\varepsilon_0)$ with $q > x^{1/2-\epsilon}$, where $\epsilon$ is very small.  Let $p_1 \dots p_s$ be the prime factorization of $e e'$. We then know that $p_i <x^\delta$ for all $i$. 
We   set
\begin{align*}
&y_i = \frac{1}{(1-\varepsilon_0)}  \log_x(f_i)  \qquad\,\,\, \mbox{ for } 1 \leq i \leq m, \\
 &y_{m+i} = \frac{1}{(1-\varepsilon_0)}   \log_x(f'_i)  \quad \mbox{ for } 1 \leq i \leq m', \\
 &z_i = \log_x(p_i)   \qquad \qquad \qquad \,\, \mbox{ for } 1 \leq i \leq s.
\end{align*} 
Since $y_i \geq \delta$ for all $i$ and $\sum_{i=1}^m y_i \leq B_{j,m}$ and $\sum_{i=m+1}^{m+m'} y_i \leq B_{j',m'}$, there exists a partition 
$I_1 \cup I_2$ of $\{1, \dots, m+m'\}$ with 
$$ \sum_{i \in I_1} (1-\varepsilon_0) y_i \leq  (1-\varepsilon_0) b < b \quad \mbox{ and } \quad  \sum_{i \in I_2} (1-\varepsilon_0) y_i \leq  (1-\varepsilon_0) \left(\tfrac{1}{2} -a\right).$$
Since $q > x^{1/2-\epsilon}$, we also know that $$ \sum_{i \in I_1} (1-\varepsilon_0) y_i +  \sum_{i=1}^s z_i \geq \tfrac{1}{2}-\epsilon-  (1-\varepsilon_0)\left(\tfrac{1}{2} -a \right)   \geq a,$$
provided $\epsilon$ is sufficiently small compared to $\varepsilon_0$.  But $z_i < \delta$   and so there exists $J \subseteq \{1, \dots, s\}$ with 
\begin{align*}
 \sum_{i \in I_1} (1-\varepsilon_0) y_i +  \sum_{i \in J}^s z_i  \in [a,b].
\end{align*}
Multiplying the corresponding $f_i$, $f'_i$ and $p_i$ together, we then get a factor $r \mid q$ with $r \in [x^a,x^b]$.
\end{proof}

\subsubsection{3 factors}

Similarly, we next ask what condition ensures that $q \in Q^*(x;\delta,\underline{A},\underline{B},\varepsilon,\varepsilon_0)$ has a factorization $q=uvr$ with $r \in [x^{a_1},x^{b_1}]$ and $u \in [x^{a_2},x^{b_2}]$.

\begin{lemma}\label{lem:partition2} Let $Q(x;\delta,\underline{A},\underline{B},\varepsilon,j,j',m,m',\varepsilon_0)$ be as given in Definition~\ref{def:neededmoduli} and let $\Xi(B_1,B_2,m_1,m_2,\delta)$ be as given in Definition~\ref{def:checkset}. 
 Consider $a_1, a_2, b_1, b_2 \in (0,\frac{1}{2})$   with $b_1 -b_2 \geq a_1-a_2$ and  $b_1 + b_2 < \frac{1}{2}$ and $b_1-a_1 \geq \delta$ and $b_2-a_2 \geq \delta$.
 
 \smallskip
 
  Suppose that for every  $(y_1, \dots, y_{m+m'}) \in \Xi(B_{j,m},B_{j',m'},m,m',\delta)$  there exists a partition  $I_1 \cup I_2 \cup I_3$ of $\{1, \dots, m+m'\}$ with 
$$ \sum_{i \in I_1} y_i \leq b_1 \quad \mbox{ and } \quad  \sum_{i \in I_2} y_i \leq  b_2  
\quad \mbox{ and } \quad  \sum_{i \in I_3} y_i \leq  \frac{1}{2} - b_1 -a_2. $$
Then there exists a constant $\varepsilon_1>0$ (dependent only on $\varepsilon_0$ and $\delta$), so that the following is true for large $x$:  If $q \in Q(x;\delta,\underline{A},\underline{B},\varepsilon,j,j',m,m',\varepsilon_0)$ and $q > x^{1/2-\varepsilon_1}$, then $q$ has a factorization $q=uvr$ with with $r \in  [x^{a_1}, x^{b_1}]$ and $u \in [x^{a_2},x^{b_2}]$. 
\end{lemma} 

\begin{proof} The proof of Lemma~\ref{lem:partition2} is very similar to the proof of Lemma~\ref{lem:partition1}:

 We first fix some $q=e e' \prod_{i=1}^m f_i \prod_{i=1}^{m'} f'_i    \in  Q(x;\delta,\underline{A},\underline{B},\varepsilon,j,j',m,m',\varepsilon_0)$ with $q > x^{1/2-\epsilon}$, where $\epsilon$ is very small.  Denoting by $p_1 \dots p_s$   the prime factorization of $e e'$, we know that  $p_i <x^\delta$ for all $i$ and set
\begin{align*}
&y_i = \frac{1}{(1-\varepsilon_0)}  \log_x(f_i)  \qquad\,\,\, \mbox{ for } 1 \leq i \leq m, \\
 &y_{m+i} = \frac{1}{(1-\varepsilon_0)}   \log_x(f'_i)  \quad \mbox{ for } 1 \leq i \leq m', \\
 &z_i = \log_x(p_i)   \qquad \qquad \qquad \,\, \mbox{ for } 1 \leq i \leq s.
\end{align*} 
Then there exists a partition 
$I_1 \cup I_2 \cup I_3$ of $\{1, \dots, m+m'\}$ with 
$$ \sum_{i \in I_1} (1-\varepsilon_0) y_i   < b_1 \quad \mbox{ and } \quad  \sum_{i \in I_2} (1-\varepsilon_0) y_i   < b_2 \quad \mbox{ and } \quad \sum_{i \in I_3} (1-\varepsilon_0) y_i \leq  (1-\varepsilon_0) \left(\tfrac{1}{2} -b_1-a_2\right).$$
Since $q > x^{1/2-\epsilon}$, we also know that 
\begin{align*}
&\sum_{i \in I_2} (1-\varepsilon_0) y_i +  \sum_{i=1}^s z_i > \tfrac{1}{2}-\epsilon- b_1- (1-\varepsilon_0)\left(\tfrac{1}{2} -b_1 -a_2 \right) \geq    a_2.
 \end{align*} 
 But $z_i < \delta$ for all $i$  and so there exists $J_2 \subseteq \{1, \dots, s\}$ with  
 \begin{align*}
&\sum_{i \in I_2} (1-\varepsilon_0) y_i +  \sum_{i \in J_2} z_i \in [a_2,b_2].
 \end{align*} 
 Further,  we then also have  
 \begin{align*}
&\sum_{i \in I_1} (1-\varepsilon_0) y_i +  \sum_{i \not\in J_2} z_i > \tfrac{1}{2}-\epsilon- b_2-  (1-\varepsilon_0)\left(\tfrac{1}{2} -b_1 -a_2 \right) \geq  b_1  - b_2 +a_2 \geq  a_1.
\end{align*} 
Hence there also exists $J_1 \subseteq \{1, \dots, s\} \setminus J_2$ with 
\begin{align*}
 \sum_{i \in I_1} (1-\varepsilon_0) y_i +  \sum_{i \in J_1}^s z_i  \in [a_1,b_1].
\end{align*}
Multiplying the corresponding $f_i$, $f'_i$ and $p_i$ together, we then get a factorization $q=uvr$ with   $r \in  [x^{a_1}, x^{b_1}]$ and $u \in [x^{a_2},x^{b_2}]$. 
\end{proof}

\subsubsection{4 factors}

Finally, we need a  condition which ensures that $q \in Q^*(x;\delta,\underline{A},\underline{B},\varepsilon,\varepsilon_0)$ has a factorization $q=uvr$ with the following properties: $r \in [x^{a_1},x^{b_1}]$ and $u \in [x^{a_2},x^{b_2}]$ and    there exists $d_1 \mid r$ with $d_1 \in [r^2x^{a_3},r^2x^{b_3}]$ . This is the condition which is typically the hardest to satisfy, and so we work more carefully and consider moduli $q$ of a specific size, given by $q \asymp x^{1/2+2\omega_0}$.

\begin{lemma}\label{lem:partition3} Let $Q(x;\delta,\underline{A},\underline{B},\varepsilon,j,j',m,m',\varepsilon_0)$ be as given in Definition~\ref{def:neededmoduli} and let $\Xi(B_1,B_2,m_1,m_2,\delta)$ be as given in Definition~\ref{def:checkset}. 
 Consider $a_1, a_2, a_3, b_1, b_2, b_3$   with   $3(b_1 - a_1) + (a_3 -b_3)\geq 0$ and $b_1 -b_2 \geq a_1-a_2$ and $b_1-a_1 \geq \delta$ and $b_2-a_2 \geq \delta$ and $b_3-a_3 \geq \delta$.
 
 \smallskip
 Let $\omega_0 \geq -\varepsilon_0$ and suppose  that for every  $(y_1, \dots, y_{m+m'}) \in \Xi(B_{j,m},B_{j',m'},m,m',\delta)$  there exists a   partition $I_1 \cup I_2 \cup I_3 \cup I_4 $ of $\{1, \dots, m+m'\}$ with 
$$ \sum_{i \in I_1} y_i \leq 2a_1+b_3 \quad \mbox{ and } \quad  \sum_{i \in I_2} y_i \leq  b_2  
\quad \mbox{ and } \quad  \sum_{i \in I_3} y_i \leq  \frac{1}{2} +2\omega_0 - b_1 -a_2  \quad \mbox{ and } \quad  \sum_{i \in I_4} y_i \leq a_1 -2b_1-a_3. $$

Then for   large $x$, every $q \in Q(x;\delta,\underline{A},\underline{B},\varepsilon,j,j',m,m',\varepsilon_0)$ with $q \asymp x^{1/2+2\omega_0}$ has a factorization  $q=uvr$ with the following properties: $r \in [x^{a_1},x^{b_1}]$ and $u \in [x^{a_2},x^{b_2}]$ and    there exists $d_1 \mid r$ with $d_1 \in [r^2x^{a_3},r^2x^{b_3}]$. 
\end{lemma} 

\begin{proof} We first fix some $q=e e' \prod_{i=1}^m f_i \prod_{i=1}^{m'} f'_i    \in  Q(x;\delta,\underline{A},\underline{B},\varepsilon,j,j',m,m',\varepsilon_0)$ with $q \asymp x^{1/2+2\omega_0}$.  Denoting by $p_1 \dots p_s$   the prime factorization of $e e'$, we know that  $p_i <x^\delta$  and set
\begin{align*}
&y_i =   \log_x(f_i)  \qquad\,\,\, \mbox{ for } 1 \leq i \leq m, \\
 &y_{m+i} =    \log_x(f'_i)  \quad \mbox{ for } 1 \leq i \leq m', \\
 &z_i = \log_x(p_i)     \qquad \, \,\, \mbox{ for } 1 \leq i \leq s.
\end{align*} 
Then there exists a partition 
$I_1 \cup I_2 \cup I_3 \cup I_4$ of $\{1, \dots, m+m'\}$ with 
$$ \sum_{i \in I_1} y_i \leq 2a_1+b_3 \quad \mbox{ and } \quad  \sum_{i \in I_2} y_i \leq  b_2  
\quad \mbox{ and } \quad  \sum_{i \in I_3} y_i \leq  \frac{1}{2} +2\omega_0 - b_1 -a_2  \quad \mbox{ and } \quad  \sum_{i \in I_4} y_i \leq  a_1 -2b_1-a_3. $$ 
Since $q \asymp x^{1/2+2\omega_0}$, we also know that 
\begin{align*}
 \sum_{i \in I_2}   y_i +  \sum_{i=1}^s z_i &> \tfrac{1}{2}+2 \omega_0-  (2a_1+b_3)- ( \tfrac{1}{2} +2\omega_0 - b_1 -a_2) - (a_1 -2b_1-a_3 ) \\
&\geq a_2  +(3b_1 - 3a_1+a_3 -b_3  )   \geq a_2.
 \end{align*} 
 But $z_i < \delta$ for all $i$  and so there exists $J_2 \subseteq \{1, \dots, s\}$ with  
 \begin{align*}
&\sum_{i \in I_2}   y_i +  \sum_{i \in J_2} z_i \in [a_2,b_2].
 \end{align*} 
This sum corresponds to factor $u$. Next, we observe that 
 \begin{align*}
 &\sum_{i \in I_1}   y_i + \sum_{i \in I_4} y_i \leq (2a_1+b_3 ) +  (  a_1 -2b_1-a_3) = b_1+ (3a_1 -3b_1 + b_3-a_3 )  \leq b_1,
 \\
&\sum_{i \in I_1}   y_i + \sum_{i \in I_4} y_i+ \sum_{i \not\in J_2} z_i > \tfrac{1}{2}+2\omega_0- b_2-   \left(\tfrac{1}{2}+2\omega_0 -b_1 -a_2 \right) \geq  b_1  - b_2 +a_2 \geq  a_1.
\end{align*} 
Hence there also exists $J_1 \subseteq \{1, \dots, s\} \setminus J_2$ with 
\begin{align*}
 \sum_{i \in I_1}   y_i + \sum_{i \in I_4} y_i+ \sum_{i  \in J_1} z_i \in [a_1,b_1].
\end{align*} 
This sum corresponds to factor $r$. 
Finally, consider $c \in [a_1,b_1]$. Notice that 
 \begin{align*}
 &\sum_{i \in I_1}   y_i  \leq 2a_1+b_3  \leq  2c +b_3 ,
 \\
&\sum_{i \in I_1}   y_i +  \sum_{i \in J_1} z_i \geq a_1 - (a_1 -2b_1-a_3) \geq 2b_1 +a_3 \geq 2c+a_3.
\end{align*}

Hence for any $c \in [a_1,b_1]$ we can find $J' \subseteq J_1$ with 
\begin{align*}
\sum_{i \in I_1}   y_i +  \sum_{i \in J'} z_i  \in [2c+a_3,2c+b_3].
\end{align*}
This sum corresponds to factor $d_1$ of $r$. Overall, multiplying the corresponding $f_i$, $f'_i$ and $p_i$ together, we then get that $q \asymp x^{1/2+2\omega_0}$ has a factorization  $q=uvr$ with   $r \in [x^{a_1},x^{b_1}]$ and $u \in [x^{a_2},x^{b_2}]$ so that  for every $R \in   [x^{a_1},x^{b_1}]$, there exists $d_1 \mid r$ with $d_1 \in [R^2x^{a_3},R^2x^{b_3}]$. 
\end{proof}

\subsection{An equidistribution criterion} 

Recall now that our goal is to find conditions on $\delta$, $\underline{A}$ and $\underline{B}$ which ensure that every $f \in \mathcal{H}(\xi_1, \xi_2,\xi_3)$  has equidistribution properties for moduli in $Q^*(x;\delta,\underline{A},\underline{B},\varepsilon,\varepsilon_0)$.  Combining the partition lemmas we just proved with the equidistribution estimates of Section~\ref{sec:equiest}, we can rephrase these conditions in terms of properties of $\Xi(B_1,B_2,m_1,m_2,\delta)$, as seen in Proposition~\ref{prop:tupleconditions}.

\begin{prop}\label{prop:tupleconditions}
Let  $T_k(\delta,\underline{A},\underline{B},\varepsilon)$ be as given in Definition~\ref{def:intregion},  let $Q^*(x;\delta,\underline{A},\underline{B},\varepsilon,\varepsilon_0)$ be as given in Definition~\ref{def:neededmoduli} and  let $\Xi(B_1,B_2,m_1,m_2,\delta)$ be as given in Definition~\ref{def:checkset}.  \smallskip

Let $\epsilon = 10^{-10}$. Suppose $\xi_1, \xi_2, \xi_3 \in (0,1)$ satisfy the following conditions: 
\begin{enumerate}[{\rm (I)}]
\item Type I condition: $\min\left\{  \xi_1 - 4 A_n + \tfrac{2}{3}   , \tfrac{9}{7}-\tfrac{34}{7} A_n  \right\}-2\epsilon  > \delta$. $($Recall that $\underline{A} =( A_1, \dots, A_n).)$
\item Type II condition: $
\tfrac{19}{2}-36 A_n   - 13\delta   +100\epsilon  \geq 0 $ and 
$$\min\left\{  \, \tfrac{\xi_2}{10} -\tfrac{32A_n}{10} + \tfrac{8}{10}  , \,
\tfrac{\xi_2}{4} +\tfrac{11}{16} -3A_n  \right\} -2\epsilon \geq \delta.$$
\item Type III condition:  $ \frac{11}{8}-\frac{7}{2}
 A_n -\frac{9}{8}\xi_3   -2\epsilon > \delta$.
\end{enumerate}

\smallskip 
Suppose further that for all $j, j'\in \{1, \dots, n\}$ and for all $m, m' \in \{1, \dots, \lfloor \frac{1}{\delta} \rfloor\}$ with $m+m'>0$, every   tuple  $(y_1, \dots, y_{m+m'}) \in \Xi(B_{j,m},B_{j',m'},m,m',\delta)$ has all of the following properties: 
\begin{enumerate}[{\rm (A)}]
\item Type I condition: There exists a partition $I_1 \cup I_2$ of $\{1, \dots, m+m'\}$ with  
\begin{align*}
\sum_{i \in I_1} y_i \leq \xi_1-2\epsilon \quad \mbox{ and } \quad  \sum_{i \in I_2} y_i \leq  \frac{1}{6}-   4 \omega(j,j')  -2\epsilon \quad \mbox{ where } \quad  \omega(j,j')=\left( \frac{A_j + A_{j'}}{2}  - \frac{1}{4} \right)  . 
\end{align*}
\item Type IIa condition: 
 There exists a partition $I_1 \cup I_2$ of $\{1, \dots, m+m'\}$ with  
\begin{align*}
\sum_{i \in I_1} y_i \leq \frac{2}{5}+\frac{24\omega(j,j')}{5} +\frac{7\delta}{5}   -2 \epsilon \quad \mbox{ and } \quad  \sum_{i \in I_2} y_i \leq \frac{1}{14}     -\frac{24\omega(j,j')}{7} -2\epsilon.
\end{align*}
\item Type IIb condition:
There also exists a partition $I_1 \cup I_2 \cup I_3$ of $\{1, \dots, m+m'\}$ with 
 \begin{align*}
 &\sum_{i \in I_1} y_i \leq  \frac{1}{3}+\frac{24\omega(j,j')}{3} +\frac{7\delta}{3}   -4\epsilon, \\  &\sum_{i \in I_2} y_i \leq    
  \frac{1}{10}-\frac{34\omega(j,j')}{5} -\frac{7\delta}{5}    
 -4 \epsilon,  \nonumber  \\
&\sum_{i \in I_3} y_i \leq            \frac{1}{35}+\frac{22\omega(j,j')}{35} +\frac{21\delta}{35}     - 4\epsilon . \nonumber
\end{align*}
\item Type IIc condition: For any $\omega_0 \in [-\epsilon, \omega(j,j')]$ and any $\gamma$ with $\xi_2 -\epsilon \leq \gamma \leq  \tfrac{1}{3}+\tfrac{24\omega(j,j')}{3} +\tfrac{7\delta}{3}  +3\epsilon$, 
there   exists a partition $I_1 \cup I_2 \cup I_3 \cup I_4$   with 
\begin{align*}
&\sum_{i \in I_1} y_i \leq   \gamma-2\delta       -    8\omega_0    -\epsilon, \\
  &\sum_{i \in I_2} y_i \leq  \frac{1}{2}-\gamma- 2\omega_0 -\epsilon, \\
   &\sum_{i \in I_3} y_i \leq  4\omega_0   +\delta    -\epsilon ,\\
    &\sum_{i \in I_4} y_i \leq   8\omega_0 .
\end{align*} 
\item Type III condition: There also exists a partition $I_1 \cup I_2$ of $\{1, \dots, m+m'\}$ with  
\begin{align*}
\sum_{i \in I_1} y_i \leq  1- 6\omega(j,j') -\frac{3\xi_3}{2}-2\epsilon \quad \mbox{ and } \quad  \sum_{i \in I_2} y_i \leq    \frac{5\omega(j,j')}{2} +\frac{3 \xi_3}{8}  -2\epsilon. 
\end{align*} 
\end{enumerate}
Then for every $f \in \mathcal{H}(\xi_1,\xi_2,\xi_3)$,  $x>1$, $\varepsilon_0 >0$, $A>0$ and    $a\in \mathbb{Z}$  with $(a,P(x))=1$,   
\begin{align*} 
\sum_{\substack{q \in Q^*(x;\delta,\underline{A},\underline{B},\varepsilon,\varepsilon_0) \\  q \mbox{ \scriptsize  is squarefree  }   }}  \Bigg| \sum_{\substack{n \in [x,2x] \\ n \equiv a (q)}} f(n;x) - \dfrac{1}{\phi(q)} \sum_{\substack{ n \in [x,2x] \\ (n,q)=1}}f(n;x) \Bigg| \ll_{A,\varepsilon_0} \dfrac{x}{\log(x)^A}.
\end{align*}
\end{prop}

\begin{proof}
Recall that $\mathcal{H}(\xi_1,\xi_2,\xi_3)$ consists of convolutions of three different types, I, II and III.   Our goal is to show that for each type of $f$ and for   all $j, j'\in \{1, \dots, n\}$ and for all $m, m' \in \{1, \dots, \lfloor \frac{1}{\delta} \rfloor\}$,
\begin{align*} 
\sum_{\substack{q \in Q(x;\delta,\underline{A},\underline{B},\varepsilon,j,j',m,m',\varepsilon_0) \\  q \mbox{ \scriptsize  is squarefree  }   }}  \Bigg| \sum_{\substack{n \in [x,2x] \\ n \equiv a (q)}} f(n;x) - \dfrac{1}{\phi(q)} \sum_{\substack{ n \in [x,2x] \\ (n,q)=1}}f(n;x) \Bigg| \ll_{A,\varepsilon_0} \dfrac{x}{\log(x)^A}.
\end{align*}

We fix some choice of $j, j'\in \{1, \dots, n\}$ and   $m, m' \in \{1, \dots, \lfloor \frac{1}{\delta} \rfloor\}$ and may assume $m+m' >0$ (as the $m+m'=0$ case follows trivially from the other cases). For $q  \in Q(x;\delta,\underline{A},\underline{B},\varepsilon,j,j',m,m',\varepsilon_0)$ we have 
$ q \leq  x^{(1-\varepsilon_0)(A_j + A_{j'})} $ and so we will work with 
$$ \omega_{\begin{scriptsize}
\mbox{max}
\end{scriptsize}} =   \frac{A_j + A_{j'}}{2}  - \frac{1}{4}.$$

\smallskip
\underline{Case I}: First consider  $f(n;x) = (\alpha \star \beta)(n;x)$, where $\alpha$ is a coefficient sequence at scale $M$ and $\beta$ is a  smooth  coefficient sequence at scale $N$ with $M(x) N(x) \asymp x$  and  $
N(x) \geq x^{ \xi_1-\epsilon} 
$.

\medskip   Initially we assume $\gamma \leq \frac{1}{2}$ and define $
\delta^*(\gamma) =     \gamma - 4\omega_{\begin{scriptsize}
\mbox{max}
\end{scriptsize}} -\tfrac{1}{3}  -\epsilon 
$. For $\gamma \geq \xi_1-\epsilon$ we now have 
\begin{align*} 
 &3\gamma-12\omega_{\begin{scriptsize}
\mbox{max}
\end{scriptsize}}  -3\delta^*(\gamma) >1. 
\end{align*}
Also, the assumption $  \xi_1 - 4 A_n + \tfrac{2}{3}    -2\epsilon  > \delta$ ensures that $\delta^*(\gamma) \geq \delta$.

\medskip Now we   apply Lemma~\ref{lem:typeIBI} with $(\omega_{\begin{scriptsize}
\mbox{max}
\end{scriptsize}}, \gamma, \delta^*(\gamma))$ in the place of $(\omega, \gamma, \delta)$. For sufficiently small $\varepsilon'$,  any $f(n;x)$ of Type I has good equidistribution properties for moduli in  $D_I(x;\omega_{\begin{scriptsize}
\mbox{max}
\end{scriptsize}},\gamma,\delta^*(\gamma);\varepsilon')$. 
But this means $ Q(x;\delta,\underline{A},\underline{B},\varepsilon,j,j',m,m',\varepsilon_0)$ also has the desired equidistribution properties, provided it is a subset of $D_I(x;\omega_{\begin{scriptsize}
\mbox{max}
\end{scriptsize}},\gamma,\delta^*(\gamma);\varepsilon')$.  Recall that for $\gamma \leq \frac{1}{2}$, 
\begin{align*}
D_I(x;\omega,\gamma,\delta;\varepsilon) =
 \left\{d \in  [1, x^{1/2+2\omega}]\cap \mathbb{N}:  \exists r \mid d \mbox{ with } x^{\gamma-\delta -3\varepsilon} < r < x^{\gamma - 3\varepsilon} \right\}.
\end{align*}
 To show $ Q(x;\delta,\underline{A},\underline{B},\varepsilon,j,j',m,m',\varepsilon_0) \subseteq D_I(x;\omega_{\begin{scriptsize}
\mbox{max}
\end{scriptsize}},\gamma,\delta^*(\gamma);\varepsilon')$,  we now apply Lemma~\ref{lem:partition1}. We need that the relevant moduli $q$ have a factor $r$ with $x^{\gamma-\delta^*(\gamma) -3\varepsilon'} < r < x^{\gamma - 3\varepsilon'}$ when $\xi_1 -\epsilon \leq \gamma \leq \frac{1}{2}$. But 
$$   \gamma-\delta^*(\gamma) -3\varepsilon' =  
  4\omega_{\begin{scriptsize}
\mbox{max}
\end{scriptsize}} +\tfrac{1}{3}  +\epsilon -3\varepsilon'.$$

By Lemma~\ref{lem:partition1} elements of  $ Q(x;\delta,\underline{A},\underline{B},\varepsilon,j,j',m,m',\varepsilon_0) $ thus  have the desired factors provided that for  every  $(y_1, \dots, y_{m+m'}) \in \Xi(B_{j,m},B_{j',m'},m,m',\delta)$  there exists a partition $I_1 \cup I_2$ of $\{1, \dots, m+m'\}$ with  
\begin{align*}
\sum_{i \in I_1} y_i \leq \xi_1-2\epsilon \quad \mbox{ and } \quad  \sum_{i \in I_2} y_i \leq  \frac{1}{6}-   4\omega_{\begin{scriptsize}
\mbox{max}
\end{scriptsize}}   -2\epsilon. 
\end{align*}

\medskip On the other hand, if $\gamma \in (\frac{1}{2},\frac{1}{2} + 2\omega_{\begin{scriptsize}\mbox{max}\end{scriptsize}} + \varepsilon']$, we use  $
\delta^*(\gamma) =   \tfrac{1}{14}-\tfrac{68}{14}\omega_{\begin{scriptsize}
\mbox{max}
\end{scriptsize}} -\epsilon 
$, so that 
\begin{align}
 &68\omega_{\begin{scriptsize}
\mbox{max}
\end{scriptsize}} + 14 \delta^*(\gamma)    <1.
\end{align}  
Then $f(n;x)$ equidistributes for moduli in $D_I(x;\omega_{\begin{scriptsize}
\mbox{max}
\end{scriptsize}},\gamma,\delta^*(\gamma);\varepsilon')$. This time, $D_I$ consists of $q$ which have a factor 
  $r$ with $x^{1-\gamma-\delta^*(\gamma) -3\varepsilon'} < r < x^{1-\gamma - 3\varepsilon'} $. Thus Lemma~\ref{lem:partition1} applies to 
$ Q(x;\delta,\underline{A},\underline{B},\varepsilon,j,j',m,m',\varepsilon_0) $  provided that for  every  $(y_1, \dots, y_{m+m'}) \in \Xi(B_{j,m},B_{j',m'},m,m',\delta)$  there exists a partition $I_1 \cup I_2$ with  
\begin{align}
\sum_{i \in I_1} y_i \leq \tfrac{1}{2} -2 \omega_{\begin{scriptsize}
\mbox{max}\end{scriptsize}} -2\epsilon \quad \mbox{ and } \quad  \sum_{i \in I_2} y_i \leq \tfrac{1}{14}-\tfrac{68}{14}\omega_{\begin{scriptsize}
\mbox{max}
\end{scriptsize}}     -2\epsilon. 
\end{align}

\medskip
\underline{Case II}: Next consider $f(n;x) = (\alpha \star \beta)(n;x)$, where $\alpha$ is  coefficient sequence at scale $M$, $\beta$ is a  coefficient sequence at scale $N$ with $M(x) N(x) \asymp x$, and $x^{ \xi_2 -\epsilon  } \leq N(x)=x^\gamma \leq x^{ 1- \xi_2+\epsilon}$. Assume also that $\alpha$ and $\beta$ both have the Siegel--Walfisz property.    By symmetry, we only need to consider $N(x) \leq x^{1/2}$. 

\medskip

We will now apply the equidistribution estimates given in Lemma~\ref{lem:typeIIPoly}, Lemma~\ref{lem:typeI(ii)Poly} and Lemma~\ref{lem:typeIS}.

\medskip
\underline{Case IIa}:  We begin with Lemma~\ref{lem:typeIIPoly} and set 
$\delta^*(\gamma) =  \frac{5\gamma}{7}-\frac{2}{7}  -\frac{24\omega_{\begin{scriptsize}
\mbox{max}\end{scriptsize}}}{7} -\epsilon$. This choice ensures that 
 \begin{align*}
&24\omega_{\begin{scriptsize}
\mbox{max}\end{scriptsize}} + 7\delta^*(\gamma) -5\gamma <-2, \\
&8\omega_{\begin{scriptsize}
\mbox{max}\end{scriptsize}} + 3\delta^*(\gamma) -\gamma <0,
\end{align*}
so that we can apply Lemma~\ref{lem:typeIIPoly} with $(\omega_{\begin{scriptsize}
\mbox{max}
\end{scriptsize}}, \gamma, \delta^*(\gamma))$ in the place of $(\omega, \gamma, \delta)$. We get that $f$ has good equidistribution properties for moduli in $D_{IIa}(x;\omega_{\begin{scriptsize}
\mbox{max}
\end{scriptsize}},\gamma,\delta^*(\gamma);\varepsilon')$, provided $\varepsilon'$ is very small. Now we   need to show that   $ Q(x;\delta,\underline{A},\underline{B},\varepsilon,j,j',m,m',\varepsilon_0) \subseteq D_{IIa}(x;\omega_{\begin{scriptsize}
\mbox{max}
\end{scriptsize}},\gamma,\delta^*(\gamma);\varepsilon')$. To apply Lemma~\ref{lem:partition1}, we need to restrict to $\gamma$ with 
$
\delta^*(\gamma) =  \frac{5\gamma}{7}-\frac{2}{7}  -\frac{24\omega_{\begin{scriptsize}
\mbox{max}\end{scriptsize}}}{7} -\epsilon \geq \delta.
$
This means for now we only consider $$\gamma \geq  \tfrac{2}{5}+\tfrac{24\omega_{\begin{scriptsize}
\mbox{max}\end{scriptsize}}}{5} +\tfrac{7\delta}{5}  +2\epsilon.$$

\bigskip
We then get that any $q \in Q(x;\delta,\underline{A},\underline{B},\varepsilon,j,j',m,m',\varepsilon_0)$ has a divisor $r \mid q$ with  $   x^{\gamma -3\varepsilon-\delta^*(\gamma)} < r < x^{\gamma - 3\varepsilon}$ provided that  for  every  $(y_1, \dots, y_{m+m'}) \in \Xi(B_{j,m},B_{j',m'},m,m',\delta)$  there exists a partition $I_1 \cup I_2$ with  
\begin{align}
\sum_{i \in I_1} y_i \leq \tfrac{2}{5}+\tfrac{24\omega_{\begin{scriptsize}
\mbox{max}\end{scriptsize}}}{5} +\tfrac{7\delta}{5}   -2 \epsilon \quad \mbox{ and } \quad  \sum_{i \in I_2} y_i \leq \tfrac{1}{14}     -\tfrac{24\omega_{\begin{scriptsize}
\mbox{max}\end{scriptsize}}}{7} -2\epsilon.
\end{align}
 
\medskip
\underline{Case IIb}: 
Now we consider $\gamma <  \tfrac{2}{5}+\tfrac{24\omega_{\begin{scriptsize}
\mbox{max}\end{scriptsize}}}{5} +\tfrac{7\delta}{5}  +2\epsilon$ and  use  Lemma~\ref{lem:typeI(ii)Poly}. Here we instead choose  
$\delta^*(\gamma) =  \frac{3\gamma}{7}-\frac{1}{7}  -\frac{24\omega_{\begin{scriptsize}
\mbox{max}\end{scriptsize}}}{7} -\epsilon$, which ensures that 
 \begin{align*}
&24\omega_{\begin{scriptsize}
\mbox{max}\end{scriptsize}} + 7\delta^*(\gamma) -3\gamma <-1, \\
&8\omega_{\begin{scriptsize}
\mbox{max}\end{scriptsize}} + 3\delta^*(\gamma) -\gamma <0.
\end{align*}
Hence we can apply Lemma~\ref{lem:typeI(ii)Poly} with $(\omega_{\begin{scriptsize}
\mbox{max}
\end{scriptsize}}, \gamma, \delta^*(\gamma))$ and get that $f$ has good equidistribution properties for moduli in $D_{IIb}(x;\omega_{\begin{scriptsize}
\mbox{max}
\end{scriptsize}},\gamma,\delta^*(\gamma);\varepsilon')$. To show    $ Q(x;\delta,\underline{A},\underline{B},\varepsilon,j,j',m,m',\varepsilon_0) \subseteq D_{IIb}(x;\omega_{\begin{scriptsize}
\mbox{max}
\end{scriptsize}},\gamma,\delta^*(\gamma);\varepsilon')$, we  apply Lemma~\ref{lem:partition2}. However, we need to restrict to $\gamma$ with 
$
\delta^*(\gamma) =   \frac{3\gamma}{7}-\frac{1}{7}  -\frac{24\omega_{\begin{scriptsize}
\mbox{max}\end{scriptsize}}}{7}  - \epsilon \geq \delta$, so   have $$   \tfrac{1}{3}+\tfrac{24\omega_{\begin{scriptsize}
\mbox{max}\end{scriptsize}}}{3} +\tfrac{7\delta}{3}  +3\epsilon< \gamma <   \tfrac{2}{5}+\tfrac{24\omega_{\begin{scriptsize}
\mbox{max}\end{scriptsize}}}{5} +\tfrac{7\delta}{5}  +2\epsilon $$
Recall now that
\begin{align*}
D_{IIb}(x;\omega,\gamma,\delta;\varepsilon) = \left\{d \in  [1, x^{1/2+2\omega}]\cap \mathbb{N}:  \begin{array}{l} 
\exists r \mid d \,\, \exists u \mid \frac{d}{r}  \mbox{ with } x^{\gamma -3\varepsilon-\delta} < r < x^{\gamma - 3\varepsilon} \\ \mbox{and } x^{1/2 -\gamma-2\omega -6 \varepsilon-\delta } < u < x^{1/2-\gamma-2\omega  -6\varepsilon}
\end{array}\right\}.
\end{align*}
Any $q \in Q(x;\delta,\underline{A},\underline{B},\varepsilon,j,j',m,m',\varepsilon_0)$ has a divisor $r \mid q$  and a divisor $u \mid \frac{q}{r}$ with  $   x^{\gamma -3\varepsilon-\delta^*(\gamma)} < r < x^{\gamma - 3\varepsilon}$ and $x^{1/2 -\gamma-2\omega_{\begin{scriptsize}
\mbox{max}\end{scriptsize}} -6 \varepsilon-\delta^*(\gamma) } < u < x^{1/2-\gamma-2\omega_{\begin{scriptsize}
\mbox{max}\end{scriptsize}} -6\varepsilon}$
provided that the following is true: For  every  $(y_1, \dots, y_{m+m'}) \in \Xi(B_{j,m},B_{j',m'},m,m',\delta)$  there exists a partition $I_1 \cup I_2 \cup I_3$ with  
\begin{align*}
 \sum_{i \in I_1} y_i \leq \gamma-2\epsilon, \quad  \sum_{i \in I_2} y_i \leq  1/2 -\gamma-2\omega_{\begin{scriptsize}
\mbox{max}\end{scriptsize}} -2 \epsilon 
\quad \mbox{ and } \quad  \sum_{i \in I_3} y_i \leq         \frac{3\gamma}{7}-\frac{1}{7}  -\frac{10\omega_{\begin{scriptsize}
\mbox{max}\end{scriptsize}}}{7}  - 2\epsilon .
\end{align*}
 Of course we could simply check this condition for every $\gamma$ in the relevant range, but Case IIb generally does not correspond to the most critical condition (Case IIc  does), and so we  simplify the above condition to the following criterion: There exists a partition $I_1 \cup I_2 \cup I_3$ with  
 \begin{align}
 &\sum_{i \in I_1} y_i \leq  \tfrac{1}{3}+\tfrac{24\omega_{\begin{scriptsize}
\mbox{max}\end{scriptsize}}}{3} +\tfrac{7\delta}{3}   -4\epsilon, \\  &\sum_{i \in I_2} y_i \leq    
  \tfrac{1}{10}-\tfrac{34\omega_{\begin{scriptsize}
\mbox{max}\end{scriptsize}}}{5} -\tfrac{7\delta}{5}    
 -4 \epsilon,  \nonumber  \\
&\sum_{i \in I_3} y_i \leq            \tfrac{1}{35}+\tfrac{22\omega_{\begin{scriptsize}
\mbox{max}\end{scriptsize}}}{35} +\tfrac{21\delta}{35}     - 4\epsilon . \nonumber
\end{align}

\underline{Case IIc}:  Finally, we use  Lemma~\ref{lem:typeIS} to treat the case $\xi_2 -\epsilon \leq \gamma \leq  \tfrac{1}{3}+\tfrac{24\omega_{\begin{scriptsize}
\mbox{max}\end{scriptsize}}}{3} +\tfrac{7\delta}{3}  +3\epsilon$.  
Here we will work with a function $\delta^*(\gamma)$ which satisfies
$$\delta^*(\gamma) \leq  \min\left\{\frac{1}{4}-2\omega_{\begin{scriptsize}
\mbox{max}\end{scriptsize}} -\frac{\gamma}{2} , \, \frac{\gamma}{10} -\frac{32\omega_{\begin{scriptsize}
\mbox{max}\end{scriptsize}}}{10}  , \,
\frac{\gamma}{4} -\frac{1}{16} -3\omega_{\begin{scriptsize}
\mbox{max}\end{scriptsize}} \right\} -\epsilon.$$
This ensures that we have the following inequalities:
\begin{align*}
&8\omega_{\begin{scriptsize}
\mbox{max}\end{scriptsize}} +4\delta^*(\gamma) + 2\gamma < 1, \\
&32\omega_{\begin{scriptsize}
\mbox{max}\end{scriptsize}}  + 10 \delta^*(\gamma) -\gamma<0, \\
&48\omega_{\begin{scriptsize}
\mbox{max}\end{scriptsize}}  +  16 \delta^*(\gamma)  -4\gamma  < -1.
\end{align*}
Hence we can apply Lemma~\ref{lem:typeIS} with $(\omega_{\begin{scriptsize}
\mbox{max}
\end{scriptsize}}, \gamma, \delta^*(\gamma))$ and get that $f$ has good equidistribution properties for moduli in $D_{IIc}(x;\omega_{\begin{scriptsize}
\mbox{max}
\end{scriptsize}},\gamma,\delta^*(\gamma);\varepsilon')$. Recall here that 
 \begin{align*}
D_{IIc}(x;\omega,\gamma,\delta;\varepsilon) = \left\{d \in  [1, x^{1/2+2\omega}]\cap \mathbb{N}:  \begin{array}{l} 
\exists r \mid d \,\, \exists u \mid \frac{d}{r} \,\, \exists d_1 \mid r \mbox{ with } x^{\gamma -3\varepsilon-\delta} < r < x^{\gamma - 3\varepsilon} \\ \mbox{and }    x^{1-\gamma-6\varepsilon-\delta}   d^{-1} < u < x^{1-\gamma-6\varepsilon}   d^{-1} \\ 
\mbox{and }  r^2 x^{2-\gamma-52\varepsilon -\delta}  d^{-4} < d_1 <  r^2 x^{2-\gamma-52\varepsilon}  d^{-4}
\end{array}\right\}.
\end{align*}

We now  ask if the given upper bound on $\delta^*(\gamma)$ is sufficiently large to allow $\delta^*(\gamma) \geq \delta$ for $\xi_2 -\epsilon \leq \gamma \leq  \tfrac{1}{3}+\tfrac{24\omega_{\begin{scriptsize}
\mbox{max}\end{scriptsize}}}{3} +\tfrac{7\delta}{3}  +3\epsilon$. On the one hand, we need
\begin{align*}
\tfrac{1}{4}-2\omega_{\begin{scriptsize}
\mbox{max}\end{scriptsize}} -\tfrac{1}{2} \left( \tfrac{1}{3}+\tfrac{24\omega_{\begin{scriptsize}
\mbox{max}\end{scriptsize}}}{3} +\tfrac{7\delta}{3}  +3\epsilon\right)  \geq \delta,
\end{align*}
which is implied by  
\begin{align*}
\tfrac{19}{2}-36 A_n   - 13\delta   +100\epsilon  \geq 0.
\end{align*}
 On the other hand, we also need that 
 $$  \min\left\{  \, \tfrac{\xi_2}{10} -\tfrac{32\omega_{\begin{scriptsize}
\mbox{max}\end{scriptsize}}}{10}  , \,
\tfrac{\xi_2}{4} -\tfrac{1}{16} -3\omega_{\begin{scriptsize}
\mbox{max}\end{scriptsize}} \right\} -2\epsilon \geq \delta.$$ Recalling that $\omega_{\begin{scriptsize}
\mbox{max}\end{scriptsize}} \leq A_n -\frac{1}{4}$, it suffices to have 
 $$  \min\left\{  \, \tfrac{\xi_2}{10} -\tfrac{32A_n}{10} + \tfrac{8}{10}  , \,
\tfrac{\xi_2}{4} +\tfrac{11}{16} -3A_n  \right\} -2\epsilon \geq \delta.$$
 These two criteria form condition (II) of Proposition~\ref{prop:tupleconditions}.  Hence we can work with any $\delta^*(\gamma)$ which satisfies
 \begin{align*}
\delta \leq \delta^*(\gamma) \leq  \min\left\{\frac{1}{4}-2\omega_{\begin{scriptsize}
\mbox{max}\end{scriptsize}} -\frac{\gamma}{2} , \, \frac{\gamma}{10} -\frac{32\omega_{\begin{scriptsize}
\mbox{max}\end{scriptsize}}}{10}  , \,
\frac{\gamma}{4} -\frac{1}{16} -3\omega_{\begin{scriptsize}
\mbox{max}\end{scriptsize}} \right\} -\epsilon.
\end{align*}  
 
  Now we can apply Lemma~\ref{lem:partition3} to show that $  Q^*(x;\delta,\underline{A},\underline{B},\varepsilon,\varepsilon_0) \subseteq  D_{IIc}(x;\omega_{\begin{scriptsize}
\mbox{max}
\end{scriptsize}},\gamma,\delta^*(\gamma);\varepsilon')$.   According to Lemma~\ref{lem:partition3}, every $q \in Q^*(x;\delta,\underline{A},\underline{B},\varepsilon,\varepsilon_0)$ has a suitable factorization provided the following is true:
For every  $\omega_0 \in [ -\varepsilon_0, \omega_{\begin{scriptsize}
\mbox{max}\end{scriptsize}}]$  and every  tuple $(y_1, \dots, y_{m+m'}) \in \Xi(B_{j,m},B_{j',m'},m,m',\delta)$  there exists a   partition $I_1 \cup I_2 \cup I_3 \cup I_4 $ of $\{1, \dots, m+m'\}$ which satisfies inequalities
\begin{align*}
&\sum_{i \in I_1} y_i \leq   \gamma-2\delta^*(\gamma)      -    8\omega_0    -\epsilon, \\
  &\sum_{i \in I_2} y_i \leq  \frac{1}{2}-\gamma- 2\omega_0 -\epsilon, \\
   &\sum_{i \in I_3} y_i \leq  4\omega_0   +\delta^*(\gamma)   -\epsilon ,\\
    &\sum_{i \in I_4} y_i \leq   8\omega_0.
\end{align*} 
Looking at these conditions, it is actually often better to work with a smaller $\delta$. Since $\delta^*(\gamma) \geq \delta$, we hence settle on the following criterion:
\begin{align*}
&\sum_{i \in I_1} y_i \leq   \gamma-2\delta       -    8\omega_0    -\epsilon, \\
  &\sum_{i \in I_2} y_i \leq  \frac{1}{2}-\gamma- 2\omega_0 -\epsilon, \\
   &\sum_{i \in I_3} y_i \leq  4\omega_0   +\delta    -\epsilon ,\\
    &\sum_{i \in I_4} y_i \leq   8\omega_0.
\end{align*} 

\medskip
\underline{Case III}: Finally, consider $f(n;x) = (\alpha \star \psi_1 \star \psi_2 \star \psi_3)(n;x)$, where $\alpha$ is a coefficient sequence at scale $M$ and $\psi_1$, $\psi_2$ and $\psi_3$ are   smooth  coefficient sequences at scales $N_1$, $N_2$ and $N_3$ with $M(x) N_1(x) N_2(x) N_3(x) \asymp x$.   Additionally, $x^{1-2\xi_3-\epsilon} \leq N_1(x), N_2(x), N_3(x) \leq x^{\xi_3+\epsilon}$ and $N_i(x)N_j(x) \geq x^{1-\xi_3-\epsilon}$ for $i \neq j$.

We define $\delta^* =  \frac{1}{2}-\frac{7}{2}\omega_{\begin{scriptsize}
\mbox{max}
\end{scriptsize}} -\frac{9}{8}\xi_3   -\epsilon$, so that 
\begin{align*}
28\omega_{\begin{scriptsize}
\mbox{max}
\end{scriptsize}}   +9\xi_3   +8\delta^*   <4. 
\end{align*}
The condition $ \frac{11}{8}-\frac{7}{2}
 A_n -\frac{9}{8}\xi_3   -2\epsilon > \delta$ ensures  $\delta^* >\delta$.  
  Now we   apply Lemma~\ref{lem:typeIII} with $(\omega_{\begin{scriptsize}
\mbox{max}
\end{scriptsize}}, \gamma, \delta^*)$ in the place of $(\omega, \gamma, \delta)$. We get that any $f(n;x)$ of Type III has good equidistribution properties for moduli in  $D_{III}(x;\omega_{\begin{scriptsize}
\mbox{max}
\end{scriptsize}},\gamma,\delta^* ;\varepsilon')$. 
To obtain $$ Q(x;\delta,\underline{A},\underline{B},\varepsilon,j,j',m,m',\varepsilon_0) \subseteq   D_{III}(x;\omega_{\begin{scriptsize}
\mbox{max}
\end{scriptsize}},\gamma,\delta^* ;\varepsilon')$$ we use once more Lemma~\ref{lem:partition1}.  We need to show that elements $q$ of $Q$  have a factor $r \mid q$ with 
$$x^{1/3+4\delta^*/3 -4\omega_{\begin{scriptsize}
\mbox{max}
\end{scriptsize}}/3-\delta^*} < r < x^{1/3+4\delta^*/3 -4\omega_{\begin{scriptsize}
\mbox{max}
\end{scriptsize}}/3}.$$ 
This is guaranteed if  every  $(y_1, \dots, y_{m+m'}) \in \Xi(B_{j,m},B_{j',m'},m,m',\delta)$ has a partition $I_1 \cup I_2$  with  
\begin{align*}
\sum_{i \in I_1} y_i \leq  \frac{1}{3}+\frac{4\delta^*}{3} -\frac{4\omega_{\begin{scriptsize}
\mbox{max}
\end{scriptsize}}}{3} \quad \mbox{ and } \quad  \sum_{i \in I_2} y_i \leq  \frac{1}{6}-\frac{\delta^*}{3} +\frac{4\omega_{\begin{scriptsize}
\mbox{max}
\end{scriptsize}}}{3}  . 
\end{align*} 
Plugging in $\delta^* =  \frac{1}{2}-\frac{7}{2}\omega_{\begin{scriptsize}
\mbox{max}
\end{scriptsize}} -\frac{9}{8}\xi_3   -\epsilon$, this is implied by
\begin{align}
\sum_{i \in I_1} y_i \leq  1- 6\omega_{\begin{scriptsize}
\mbox{max}
\end{scriptsize}} -\frac{3\xi_3}{2}-2\epsilon \quad \mbox{ and } \quad  \sum_{i \in I_2} y_i \leq    \frac{5\omega_{\begin{scriptsize}
\mbox{max}
\end{scriptsize}}}{2} +\frac{3 \xi_3}{8}  -2\epsilon. 
\end{align} 
This concludes the proof.
\end{proof}

The statement of Proposition~\ref{prop:tupleconditions} may look a bit complicated at first, but it is extremely useful for identifying suitable choices of $B_{j,m}$ for given $\underline{A}$ and $\delta$. For instance, one may choose some small constant $c$ and sample tuples in $\Xi(B_1,B_2,m_1,m_2,\delta)$ by considering $(y_1, \dots, y_{m_1+m_2})$ with $y_i \in c\mathbb{N}$. Using a computer, it is then easy to check if there is a partition $I_1 \cup I_2 \cup I_3 \cup I_4$ so that $(y_1, \dots, y_{m_1+m_2})$ satisfies $\sum_{i \in I_j} y_i < C_j$ for $j=1,2,3,4$. Choosing constants $C_1$, $C_2$, $C_3$ and $C_4$ according to the partitions listed in Proposition~\ref{prop:tupleconditions}, but leaving room to spare to account for an error   $(m_1+m_2) c$, one can then find   $B_1$ and $B_2$ with the desired partition properties. Of course, the smaller the choice of $c$, the longer this computation takes, but at least for small $m$ and $m'$ it is not difficult to find near optimal values for $B_{j,m}$ by using this approach. However, to prove $H_1 \leq 240$, such a brute force computation is actually not necessary and we will instead pick $B_{j,m}$ for which the conditions of  Proposition~\ref{prop:tupleconditions} can easily be verified by hand.

\section{Integrals via matrix multiplication}\label{sec:code}

In the previous sections we discussed the theoretical framework of our proof: we proved a variant of the GPY sieve for the support $T_k(\delta,\underline{A},\underline{B},\varepsilon)$, derived useful equidistribution lemmas and constructed a  minorant for the prime indicator function. However, once we have chosen $\delta$, $\underline{A}$, $\underline{B}$ and $\varepsilon$, one major obstacle remains: we need to solve the optimization problem described in Proposition~\ref{prop:GPYsieve} and find a function $F$ which makes the following quantity large:
\begin{align*} 
\frac{k(1-c_1)  J(F;\delta,\underline{A},\underline{B},\varepsilon)-k c_2   K(F;\delta,\underline{A},\underline{B},\varepsilon)}{I(F;\delta,\underline{A},\underline{B},\varepsilon)} > 1.
\end{align*}

\subsection{A basis}  To find a good function $F$, we     first  choose some symmetric functions $G_1, \dots, G_l$ as a basis and consider linear combinations $F=c_1 G_1 + \dots + c_l G_l$. We set $\underline{c}=(c_1,\dots,c_l)$ and define matrices $\mathbf{M_1}$ and $\mathbf{M_2}$ as follows: 
\begin{align*}
(M_1)_{i,j}&=\mathop{\int \dots \int}_{T_k(\delta,\underline{A},\underline{B},\varepsilon)} G_i(t_1, \dots, t_k)G_j(t_1, \dots, t_k) \mbox{d}t_1 \dots \mbox{d}t_k, \\
(M_2)_{i,j}&= k (1-c_1) \sum_{m=1}^n\sum_{m'=1}^n\mathop{\int \dots \int}_{\substack{t_1+\dots+t_{i-1}+t_{i+1}+\dots+t_k \leq \max\{A_m-\varepsilon,A_{m'}-\varepsilon\} \\ t_1+\dots+t_{i-1}+t_i+t_{i+1}+\dots+t_k \in [A_{m-1}+\varepsilon,A_m+\varepsilon] \\  t_1+\dots+t_{i-1}+t'_i+t_{i+1}+\dots+t_k \in [A_{m'-1}+\varepsilon,A_{m'}+\varepsilon]  }} \!\!\!\!\!\!\!\!\!\!\!\!\!\!\!\!\!\!\!\!\!\!\!\!\! G_i(t_1,\dots, t_{k-1}, t_k)  G_j(t_1, \dots, t_{k-1}, t'_k)\mbox{d}t_k\mbox{d}t'_k \mbox{d}t_1 \dots \mbox{d}t_{k-1} \\
&-kc_2 \sum_{m=1}^n\sum_{m'=1}^n\mathop{\int \dots \int}_{\substack{ t_1 + \dots + t_{k-1} > \max\{A_m-\varepsilon,A_{m'}-\varepsilon\}\\ t_1+\dots+t_{k-1}+t_k \in [A_{m-1}+\varepsilon,A_m+\varepsilon] \\  t_1+\dots + t_{k-1}+t'_k \in [A_{m'-1}+\varepsilon,A_{m'}+\varepsilon]  }} G_i(t_1, \dots, t_k)G_j(t_1, \dots, t_k) \mbox{d}t_1 \dots \mbox{d}t_k
\end{align*}
This gives us the following identity:
\begin{align}\label{equ:matrixratio}
\frac{\underline{c} \, \mathbf{M_2} \,  \underline{c}^T}{\underline{c} \, \mathbf{M_1} \, \underline{c}^T} = \frac{k(1-c_1)  J(F;\delta,\underline{A},\underline{B},\varepsilon)-k c_2   K(F;\delta,\underline{A},\underline{B},\varepsilon)}{I(F;\delta,\underline{A},\underline{B},\varepsilon)} .
\end{align}
To find the optimal choice for $\underline{c}$, Maynard~\cite{Maynard:2015:SGP} and  Polymath~\cite{Polymath:2014:VSS} computed the largest eigenvalue of $\mathbf{M_2} \mathbf{M_1}^{-1}$, and chose  $\underline{c}$ to be a corresponding eigenvector. While our new choice of $\mathbf{M_1}$ and $\mathbf{M_2}$ may not be positive definite, we nonetheless follow the same approach: We compute exact values for the entries of $\mathbf{M_1}$ and $\mathbf{M_2}$, and then approximate numerically to obtain eigenvalues and eigenvectors. Choosing an eigenvector, we then approximate it via rational numbers, obtaining a good choice for $\underline{c}$. We then compute the ratio~(\ref{equ:matrixratio}) exactly, and if the result is greater than $1$, we have shown that $H_1 \leq H(k)$, where $H(k)$ is the length of the narrowest admissible $k$-tuple.

\medskip 

Regarding the choice of basis $G_1, \dots, G_l$, Polymath considered the set of symmetric polynomials   $$ \mathcal{B}_D=\left\{   p(t_1, \dots, t_k)^2 (1-t_1-\dots-t_k)^b :  p(t_1, \dots, t_k) \mbox{ symmetric polynomial}, \mbox{deg}(p)=a,   2a+b \leq D \right\}.$$ In particular, $H_1 \leq 246$ was obtained by optimization over  basis $\mathcal{B}_{27}$. To obtain $H_1 \leq 240$, we  work with the smaller basis $\mathcal{B}_{19}$, which indicates that our method gives quite a large improvement.
   
\subsection{Integration via recursion and matrix multiplication} 
  
To find eigenvectors of  $\mathbf{M_2} \mathbf{M_1}^{-1}$, one  must of course first   calculate the entries of matrices $\mathbf{M_1}$ and $\mathbf{M_2}$. We need to obtain exact integrals as even minor errors in the matrix entries quickly lead to incorrect results. However, trying to use Mathematica to integrate a polynomial in $49$ variables over a region like $T_k(\delta,\underline{A},\underline{B},\varepsilon)$ seems pretty hopeless: even for   polynomials  in 10 variables a computation can already take 10 minutes or more, and the duration appears to increase exponentially, so that we would never get any result for $49$ variables. Therefore we need to find a better way to integrate polynomials over $T_k(\delta,\underline{A},\underline{B},\varepsilon)$. We have carefully chosen the region to make it possible to completely rewrite integration   in terms of recursion and matrix multiplication. Below we will briefly explain the key ideas behind this process.  

\subsubsection{Building blocks}
   We begin by  considering much simpler sets 
 \begin{align*} 
 &T_s(k)=\{(t_1, \dots, t_k) \in [0,\delta]^k: t_1+\dots+t_k \leq 1\}, \\ 
 &T_b(k)=\{(t_1, \dots, t_k) \in [\delta,1]^k: t_1+\dots+t_k \leq 1\} 
  \end{align*} 
 and polynomials $ t_1^{a_1} \dots t_k^{a_k} (1-t_1-\dots-t_k)^b$.  If one naively  expands $(1-t_1-\dots-t_k)^b$ and integrates $ t_1^{a_1} \dots t_k^{a_k} (1-t_1-\dots-t_k)^b$ by hand, then the condition $t_1+\dots+t_k \leq 1$ has the effect that each evaluation of $\int\mbox{d}t_i$ increases the polynomial's degree by $1$. So we would quickly get an unmanageable number of different polynomials that need to be considered. To avoid such issues,   we use the following two tricks: 

\medskip  
 First, we note that the integral of $ t_1^{a_1} \dots t_k^{a_k} (1-t_1-\dots-t_k)^b$ over $T_s(k)$  or $T_b(k)$ can itself be expressed (piecewise) as a polynomial in $\delta$. Its degree is at most $k+a_1+\dots+a_k+b$. The coefficients of this polynomial only depend on the value of $\lfloor \frac{1}{\delta} \rfloor$. We may write
 \begin{align*}
 &\mathop{\int \dots \int }_{T_s (k)}    t_1^{a_1} \dots t_k^{a_k} (1-t_1-\dots-t_k)^b \mbox{d}t_1 \dots  \mbox{d}t_k = \sum_{i=0}^{k+a_1+\dots+a_k+b} C_{m,i}(k,\underline{a},b) \,\delta^i \quad \mbox{ for }  \left\lfloor \frac{1}{\delta} \right\rfloor \in [m,m+1], \\ 
  &\mathop{\int \dots \int }_{T_b (k)}    t_1^{a_1} \dots t_k^{a_k} (1-t_1-\dots-t_k)^b \mbox{d}t_1 \dots  \mbox{d}t_k = \sum_{i=0}^{k+a_1+\dots+a_k+b} D_{m,i}(k,\underline{a},b) \,\delta^i \quad \mbox{ for }  \left\lfloor \frac{1}{\delta} \right\rfloor \in [m,m+1].
 \end{align*}
 The goal  now is to find the coefficients $C_{m,i}(k,\underline{a},b)$ and $D_{m,i}(k,\underline{a},b)$. 
 
\medskip  
  Secondly, when $a_k \neq 0$ or $b \neq 0$  
 we introduce additional integrals to replace $t_k^{a_k}$ and $(1-t_1-\dots-t_k)^b$. This trick gives us  a nice recursive definition for the coefficients $C_{m,i}(k,\underline{a},b)$ and $D_{m,i}(k,\underline{a},b)$. Here is a brief example:
 \begin{align*}
 &\sum_{i=0}^{k+a_1+\dots+a_k+b} C_{m,i}(k,\underline{a},b) \, \delta^i  =\mathop{\int_0^\delta \dots \int_0^\delta}_{t_1 + \dots + t_k \leq 1}   t_1^{a_1} \dots t_k^{a_k} (1-t_1-\dots-t_k)^b \mbox{d}t_1 \dots  \mbox{d}t_k \\
 &= 
 \mathop{\int_0^\delta \dots \int_0^\delta}_{t_1 + \dots + t_k \leq 1} \int_{0}^{(1-t_1-\dots-t_k)^b}   t_1^{a_1} \dots t_{k}^{a_{k}}  \, \mbox{d}s \, \mbox{d}t_1 \dots  \mbox{d}t_k\\
 &= 
\int_{0}^{1}  (1 -s^{1/b} )^{k+ a_1+\dots + a_k} \mathop{\int_0^{\frac{\delta}{1 -s^{1/b} }}  \dots \int_0^{\frac{\delta}{1 -s^{1/b} }}}_{\substack{ u_1 + \dots + u_k \leq 1    }}   u_1^{a_1} \dots u_{k}^{a_{k}} \, \mbox{d}u_1 \dots  \mbox{d}u_k \, \mbox{d}s \\
&= \sum_{l=0}^\infty 
\int_{\max\{0,(1-\delta (l+1))^b\}}^{(1-\delta l)^b}  (1 -s^{1/b} )^{k+ a_1+\dots + a_k}   \sum_{i=0}^{k+a_1+\dots+a_k+b} C_{l,i}(k,\underline{a},0)\left(\frac{\delta}{1-s^{1/b}}\right)^i\, \mbox{d}s.
 \end{align*}
 Evaluating the integral over $s$, we get another polynomial in $\delta$. Carefully comparing coefficients on both sides, we can then express $C_{m,i}(k,\underline{a},b)$ in terms of $C_{l,i}(k,\underline{a},0)$ for various $l$. Similarly, $C_{m,i}(k,\underline{a},0)$ can be expressed in terms of $C_{m,i}(k,(a_1, \dots, a_{k-1}),0)$, $D_{m,i}(k,(a_1, \dots, a_{k-1}),0)$, $C_{m-1,i}(k,(a_1, \dots, a_{k-1}),0)$ and $D_{m-1,i}(k,(a_1, \dots, a_{k-1}),0)$. This algorithm drastically speeds up integration over $T_s(k)$ for polynomials of the form   $t_1^{a_1} \dots t_k^{a_k} (1-t_1-\dots-t_k)^b $.

\subsubsection{More complicated integrals}

The integrals discussed in the previous section work as building blocks for the calculation of more complicated integrals.  For instance, we may now consider the region
\begin{align*}
T_{b,s}(k)=\{(u_1,\dots, u_r, v_1, \dots, v_s) \in [\delta,1]^r\times [0,\delta]^s: u_1 + \dots + u_r+ v_1 + \dots + v_s< A, v_1 + \dots + v_r < D\}.
\end{align*}
Integration of $  u_1^{a_1} \dots u_r^{a_r}  v_1^{b_1} \dots v_s^{b_s} (A-u_1-\dots-u_r-v_1 - \dots -v_s)^c$ over $T_{b,s}(k)$ can be expressed in terms of the coefficients $C_{m,i}(r,\underline{b},n)$ and $D_{m,i}(s,\underline{a},n)$:
\begin{align*}
&\mathop{\int \dots \int}_{T_{b,s}(k)}  u_1^{a_1} \dots u_r^{a_r}  v_1^{b_1} \dots v_s^{b_s} (A-u_1-\dots-u_r-v_1 - \dots -v_s)^c \,\mbox{d}u_1 \dots \mbox{d}u_r \mbox{d}v_1 \dots \mbox{d}v_s \\
&= \mathop{\int_\delta^1 \dots \int_\delta^1}_{A-D<u_1 + \dots + u_r < A}   u_1^{a_1} \dots u_r^{a_r}   \!\!\!\!\! \!\!\!\!\! 
\mathop{\int_0^\delta \dots \int_0^\delta}_{\substack{ v_1 + \dots + v_s < A-u_1-\dots - u_r }}   \!\!\!\!\! \!\!\!\!\!    
v_1^{b_1} \dots v_s^{b_s} (A-u_1-\dots-u_r-v_1 - \dots -v_s)^c \\
&+\mathop{\int_\delta^1 \dots \int_\delta^1}_{ u_1 + \dots + u_r \leq  A-D}   u_1^{a_1} \dots u_r^{a_r}  \sum_{n=0}^c \binom{c}{n} (A-D-u_1-\dots-u_r)^{c-n}  
\mathop{\int_0^\delta \dots \int_0^\delta}_{\substack{ v_1 + \dots + v_s < D }}         
v_1^{b_1} \dots v_s^{b_s} (D-v_1 - \dots -v_s)^n 
\\
&=\sum_{l=0}^\infty \!\!\!\!\! \!\!\!\!\! \mathop{\int_\delta^1 \dots \int_\delta^1}_{\max\{A-D,A-\delta l\}<u_1 + \dots + u_r < A-\delta(l+1)} \!\!\!\!\! \!\!\!\!\!\!\!\!\!\! \!\!\!\!\!\!\!\!\!\! \!\!\!\!\!  u_1^{a_1} \dots u_r^{a_r}    (A-u_1-\dots-u_r )^{s+b_1+\dots+b_k+c} \sum_{i=0}^{s+b_1+\dots+b_s+c} C_{l,i}(s,\underline{b},c)\left(\frac{\delta}{A-u_1-\dots-u_r}\right)^i  \\
&+\mathop{\int_\delta^1 \dots \int_\delta^1}_{ u_1 + \dots + u_r \leq  A-D}   u_1^{a_1} \dots u_r^{a_r}  \sum_{n=0}^c \binom{c}{n} (A-D-u_1-\dots-u_r)^{c-n}  D^{s+b_1+\dots+b_k+n}    \sum_{i=0}^{s+b_1+\dots+b_s+n} C_{\lfloor \frac{D}{\delta}\rfloor,i}(s,\underline{b},c)\left(\frac{\delta}{D}\right)^i.
\end{align*}
The computation above already replaced the integrals over $v_i < \delta$ by coefficients $C_{m,i}(s,\underline{b},n)$, and the remaining integrals over $u_i > \delta$ can similarly be replaced by $D_{m,i}(r,\underline{a},n)$. Overall, the entire integral can be expressed by multiplying a suitable matrix left and right by the coefficient vectors $C$ and $D$. 

\medskip The integral over $T_{b,s}(k)$ is itself  a  building block in our ultimate goal of integrating symmetric polynomials over  $T_k(\delta,\underline{A},\underline{B},\varepsilon)$.  My code  decomposes all such integrals into matrix multiplications and recursions.  The above discussion highlighted the key ideas behind my calculation process, but I will also upload my full code at a later time, when I have had time to insert more comments to make it reader friendly. 

\section{Final steps}\label{sec:proof}

To conclude the proof of $H_1 \leq 246$, we first briefly discuss our general strategy for proving bounds on $H_1$ via Proposition~\ref{prop:GPYsieve}, Proposition~\ref{prop:harman} and Proposition~\ref{prop:tupleconditions}.

\subsection{General Strategy} Recall that   $T_k(\delta,\underline{A},\underline{B},\varepsilon)$ is defined as follows:
\begin{align*}
T_k(\delta,\underline{A},\underline{B},\varepsilon)=\bigcup_{j=1}^n \left\{(t_1, \dots, t_k) \in [0,1]^k:     \sum_{i=1}^k t_i \in \left[A_{j-1}+\varepsilon,  A_j+\varepsilon \right) \mbox{ and }  \sum_{i \in I} t_i \leq B_{j,|I|} \mbox{ for }  I=\{i: t_i > \delta\} \right\}.
\end{align*}
We need to find good $\underline{A}$, $\underline{B}$, $\varepsilon$ and $\delta$. We start by   selecting some   vector $(A_1, \dots, A_n)$ and some $\delta >0$. 
We then take a look at Proposition~\ref{prop:tupleconditions} and choose parameters $(\xi_1, \xi_2, \xi_3)$ as follows:
\begin{align*}
&\xi_1 \approx 4A_n + \delta - \frac{2}{3}, \\
&\xi_2  \approx \max\left\{ \tfrac{4}{10}, \, 32A_n - 8 + 10 \delta, \, 12A_n - \tfrac{11}{4} + 4\delta \right\}, \\
&\xi_3 \approx \tfrac{11}{9} - \tfrac{28}{9} A_n - \tfrac{8}{9} \delta.
\end{align*}
Next we apply Proposition~\ref{prop:harman} to construct a minorant $\rho(n;x)$ for the prime indicator function. If the conditions on $\xi_1, \xi_2, \xi_3$ stated in Proposition~\ref{prop:harman} fail, or if $c_1$ is large, we need to decrease the size of $\delta$.

\medskip Having chosen $(\xi_1,\xi_2,\xi_3)$ and $\underline{A}$ and $\delta$, we then use Proposition~\ref{prop:tupleconditions} to find a suitable choice for $\underline{B}$, ideally with $B_{m,j}$ large. This determines the support $T_k(\delta,\underline{A},\underline{B},\varepsilon)$ and we can now compute matrices $\mathbf{M}_1$ and $\mathbf{M}_2$ and search for a good choice of vector $\underline{c}$. If we find $\underline{c}$ with $$\frac{\underline{c} \, \mathbf{M_2} \,  \underline{c}^T}{\underline{c} \, \mathbf{M_1} \, \underline{c}^T}>1,$$
then Proposition~\ref{prop:GPYsieve} tells us that $H_1 \leq H(k)$.

\subsection{Proof of Theorem~\ref{thm}}

Proposition~\ref{prop:GPYsieve}, Proposition~\ref{prop:harman} and Proposition~\ref{prop:tupleconditions} were all stated much more generally than is required for the proof of $H_1 \leq 240$. As will be discussed below, a quite simple choice of $T_k(\delta,\underline{A},\underline{B},\varepsilon)$ can be used to get this result, and we can work directly with $1_{\mathbb{P}}(n)$ rather than a minorant $\rho(n;x)$. We   do not need the full strength of Proposition~\ref{prop:tupleconditions} either. I developed these more general results with the intend of proving better bounds on $H_1$, but due to time limitations, this paper only proves the minimal possible improvement.

\begin{proof}[Proof of Theorem~\ref{thm}]
We choose $\varepsilon =  0.0075
$ and  $\underline{A} = (-\varepsilon, 0.253$) and $\delta = 0.028$.  Furthermore, we take $\underline{B} = (B_{1,1}, \dots, B_{1, \lfloor \frac{1}{\delta} \rfloor})$ with $B_{1,1}=B_{1,2} = 0.15$ and $B_{1,m} = 0.17$ for $m \geq 3$.

\medskip Next  we choose $\xi_1= 0.38 $  and $\xi_2=\xi_3= 0.4$. For this choice, conditions (I), (II) and (III) of Proposition~\ref{prop:tupleconditions} are satisfied. Furthermore, the upper bound on the sum over $I_1$ in each of the conditions (A), (B), (C) and (E) is greater than $0.34$. Since $B_{1,m}+B_{1,m'} \leq 0.34$, this means we can always simply choose $I_1 = \{1, \dots, m+m'\}$. In condition (D) we are done if we have $ \{1, \dots, m+m'\} = I_1 \cup I_2$ with bounds
\begin{align*}
&\sum_{i \in I_1} y_i \leq  0.32 \quad  \mbox{ and } \quad 
   \sum_{i \in I_2} y_i \leq  0.07.
\end{align*} 
 This is possible since $B_{1,m}+ B_{1,m'} \leq 0.32$ if $m\leq 2$ or $m' \leq 2$, and for tuples with $m\geq 3$ we must have a subsum between $0.02$ and $0.07$. Thus  by Proposition~\ref{prop:tupleconditions},  $f \in \mathcal{H}(\xi_1,\xi_2,\xi_3)$ have the desired equidistribution properties.   We also see that $(\xi_1,\xi_2,\xi_3)$ satisfies all   inequalities given in Proposition~\ref{prop:harman}, and this proposition tells us to use $\rho(n;x)=1_{\mathbb{P}}(n)$, which means $c_1=c_2=0$. We compute the relevant matrices $\mathbf{M}_1$ and $\mathbf{M}_2$ for $k=49$ and are able to obtain a vector $\underline{c}$ with 
  $$\frac{\underline{c} \, \mathbf{M_2} \,  \underline{c}^T}{\underline{c} \, \mathbf{M_1} \, \underline{c}^T}>1.$$
  Therefore  Proposition~\ref{prop:GPYsieve} tells us that $H_1 \leq 240$.
\end{proof}

\section*{Acknowledgements}

I would like to thank James Maynard,    Kannan Soundararajan, Kevin Ford and Jesse Thorner for their support and various helpful conversations  during different stages of this project. 

\bibliography{bibstad}
\bibliographystyle{acm} 

\end{document}